\documentclass[10pt]{amsart}
\usepackage{amssymb,mathrsfs,graphicx,extpfeil}
\usepackage{epsfig}
\usepackage{indentfirst, latexsym, amssymb, enumerate,amsmath,graphicx}
\usepackage{float}
\usepackage{colortbl}
\usepackage{epsfig,subfigure}
\usepackage{dsfont}
\usepackage[margin=1in]{geometry}

\usepackage{amsmath}
\title[Asymptotic analysis for a Vlasov-Fokker-Planck/Navier-Stokes system]{Asymptotic analysis for a Vlasov-Fokker-Planck/Navier-Stokes system in a bounded domain}

\author[Choi]{Young-Pil Choi}
\address[Young-Pil Choi]{\newline Department of Mathematics \newline Yonsei University, Seoul 03722, Korea (Republic of)}
\email{ypchoi@yonsei.ac.kr}

\author[Jung]{Jinwook Jung}
\address[Jinwook Jung]{\newline Research Institute of Basic Sciences \newline Seoul National University, Seoul  08826, Korea (Republic of)}
\email{warp100@snu.ac.kr}

\newtheorem{theorem}{Theorem}[section]
\newtheorem{lemma}{Lemma}[section]
\newtheorem{corollary}{Corollary}[section]
\newtheorem{proposition}{Proposition}[section]

\newtheorem{remark}{Remark}[section]

\newtheorem{definition}{Definition}[section]

\newcommand{\bbr}{\mathbb R}

\newcommand{\bbn} {\mathbb N}
\newcommand{\md}{\mathcal{D}}
\newcommand{\mf}{\mathcal{F}}
\newcommand{\bq}{\begin{equation}}
\newcommand{\eq}{\end{equation}}
\newcommand{\N}{\mathbb{N}}
\newcommand{\ms}{\mathcal{S}}

\newcommand{\e}{\varepsilon}

\newcommand{\R}{\mathbb R}
\newcommand{\lt}{\left}
\newcommand{\rt}{\right}
\newcommand{\om}{\Omega}
\newcommand{\pa}{\partial}
\makeatletter
\def\moverlay{\mathpalette\mov@rlay}
\def\mov@rlay#1#2{\leavevmode\vtop{%
   \baselineskip\z@skip \lineskiplimit-\maxdimen
   \ialign{\hfil$\m@th#1##$\hfil\cr#2\crcr}}}
\newcommand{\charfusion}[3][\mathord]{
    #1{\ifx#1\mathop\vphantom{#2}\fi
        \mathpalette\mov@rlay{#2\cr#3}
      }
    \ifx#1\mathop\expandafter\displaylimits\fi}
\makeatother
\newcommand{\cupdot}{\charfusion[\mathbin]{\cup}{\cdot}}

\newcommand{\mc}{\mathcal{C}}

\begin{document}

\date{\today}


%
\begin{abstract}
We study an asymptotic analysis of a coupled system of kinetic and fluid equations. More precisely, we deal with the nonlinear Vlasov-Fokker-Planck equation coupled with the compressible isentropic Navier-Stokes system through a drag force in a bounded domain with the specular reflection boundary condition for the kinetic equation and homogeneous Dirichlet boundary condition for the fluid system. We establish a rigorous hydrodynamic limit corresponding to strong noise and local alignment force. The limiting system is a type of two-phase fluid model consisting of the isothermal Euler system and the compressible Navier-Stokes system. Our main strategy relies on the relative entropy argument based on the weak-strong uniqueness principle. For this, we provide a global-in-time existence of weak solutions for the coupled kinetic-fluid system. We also show the existence and uniqueness of strong solutions to the limiting system in a bounded domain with the kinematic boundary condition for the Euler system and Dirichlet boundary condition for the Navier-Stokes system.
\end{abstract}
\maketitle \centerline{\date}

\allowdisplaybreaks
\tableofcontents
\section{Introduction}
\setcounter{equation}{0}

In this paper, we are interested in the asymptotic analysis of kinetic-fluid equations, consisting of the Vlasov-Fokker-Planck equation and the compressible isentropic Navier-Stokes equations where the coupling is through a drag force, also called a friction force, in a bounded domain with the specular reflection boundary condition. More specifically, let $f=f(x,\xi,t)$ be the number density of particles that are located at the position $x \in \om \subset \R^3$ with velocity $\xi \in \R^3$ at time $t > 0$, and $n = n(x,t)$  and $v = v(x,t)$ be the local density and velocity of the compressible viscous fluid, respectively. Then our main system reads
\begin{align}
\begin{aligned}\label{A-1}
&\partial_t f + \xi \cdot \nabla_x f + \nabla_\xi \cdot (F_d f) =  \nabla_\xi \cdot (\sigma\nabla_\xi f - \alpha(u - \xi)f), \quad (x,\xi,t) \in \om \times \R^3 \times \R_+,\\
&\partial_t n + \nabla_x \cdot (n v)=0, \quad (x,t) \in \om \times \R_+,\\
&\partial_t (n v) + \nabla_x \cdot (n v \otimes v) + \nabla_x p  -\mu\Delta_x v = -\int_{\bbr^3} m_pF_d f \,d\xi, 
\end{aligned}
\end{align}
subject to initial data:
\[
(f(x,\xi,t), n(x,t), v(x,t))|_{t=0} =: (f_0(x,\xi), n_0(x), v_0(x)), \quad (x,\xi) \in \om \times \R^3, 
\]
where $p=p(n)$ denotes the pressure given by $ p(n):= n^\gamma$ with $\gamma > 1$, $\mu > 0$ is the viscosity coefficient, and $\sigma > 0$ and $\alpha > 0$ represent the strengths of diffusive and local alignment forces, respectively. Throughout this paper, we set $\mu=1$ to simplify the presentation of computations.

The first two terms in the kinetic equation in \eqref{A-1} describe the free transport of dispersed particles in a fluid. The interactions between particles and fluid are taken into account through the drag force $F_d$ in the third term, which also appears as an external force in the Navier-Stokes momentum equations in \eqref{A-1}. Note that when the interaction between particles and fluid is ignored, i.e., $F_d \equiv 0$, then the kinetic equation in \eqref{A-1} becomes the nonlinear Vlasov-Fokker-Planck equation \cite{C2,K-M-T,Vil02}. Typically, the drag force $F_d$ depends on the relative velocity $v(x,t) - \xi$ and the viscosity coefficient $\mu$, see \cite{BGLM15, Oro81, RR52} for instance. To be more precise, one of the most typical formulae is given by
\bq\label{drag}
m_pF_d = c_d \mu (v(x,t) - \xi).
\eq
Here $r_p > 0$ is the radius of the particles and $m_p > 0$ is the mass of one single particle, i.e., $m_p = (4/3)\pi r_p^3  \rho_p$ where $\rho_p$ is the constant particle density. In the current work, we assumed that both $r_p$ and $m_p$ are constants. $c_d$ is the coefficient of drag force, which depends on $r_p$ and the Reynolds number. For the compressible fluid, it might be natural to consider the density-dependent viscosity, and as a consequence, the drag force is proportional to the fluid density, we refer to  \cite{BD06,CK15,CHJKpre2} for the case $F_d = n(v-\xi)$. However, we consider a simple case; the viscosity coefficient $\mu$ is constant, as mentioned above, and the drag force does not depend on the fluid density. For the sake of simplicity, we suppose that all the constants appeared in \eqref{drag} equal to 1. The terms on the right-hand side of the kinetic equation in \eqref{A-1} explains the diffusion and the interactions between particles trying to align their velocity with the local particle velocity $u = u(x,t)$ which is given by
\[
u(x,t) := \int_{\bbr^3} \xi f(x,\xi,t) \,d\xi \bigg/ \int_{\bbr^3} f(x,\xi,t) \,d\xi.
\]

This type of system describing the interactions between particle and fluid has received considerable attention due to the number of its applications; biotechnology \cite{BBKT03} and medicine \cite{BBJM05}, for instance. It can be also applied to the study of compressibility of droplets of spray, diesel engines, and sedimentation phenomena \cite{BBJM05,BBKT03,BGLM15,Oro81,Will58}. Among the various levels of possible descriptions, based on the volume fraction of the gas, our main system \eqref{A-1} considers the case of ``thin sprays'' in which the volume occupied by the particles is negligible compared to the volume occupied by the gas. We refer to \cite{Desv10,Oro81} for the classification for the modeling of interactions between particles and fluid.

\subsection{Boundary conditions} For the bounded domain $\om \subset \R^3$, we assume that the boundary $\pa \om$ is a smooth hypersurface. To state the boundary conditions for the density function $f$, which solves the kinetic equation in \eqref{A-1}, we introduce several notations:
\[
\{ (x,\xi) \in \pa\om\times \R^3\} =: \Sigma = \Sigma_+ \cupdot \Sigma_- \cupdot \Sigma_0,
\]
where $\Sigma_\pm$ and $\Sigma_0$ represent outgoing/incoming boundaries and grazing set given by
\[ 
\Sigma_{\pm} := \{ (x,\xi) \in \Sigma \, : \, \pm \,\xi \cdot r(x) >0\} \quad \mbox{and} \quad \Sigma_0:= \{ (x,\xi) \in \Sigma \, : \, \xi \cdot r(x) = 0\},
\]
respectively, where $r(x)$ is the outward unit normal vector to $x \in \partial\Omega$. We denote the traces of $f$ by $\gamma_{\pm} f(x,\xi,t) = f |_{\Sigma_{\pm}}$. We also introduce
\[ 
L^p(\Sigma_{\pm}) := \left\{ g(x, \xi) \, : \, \left( \int_{\Sigma_{\pm}} |g(x,\xi)|^p |\xi \cdot r(x) | \,d\sigma (x) d\xi \right)^{1/p} < \infty \right\},   
\]
where $d\sigma(x)$ denotes the Euclidean metric on $\partial \Omega$. For the hydrodynamic limit, we are mainly interested in the specular reflection boundary condition for $f$:
\bq\label{bdy_ref}
\gamma_- f(x,\xi,t) = \gamma_+ f (x, \mathcal{R}_x(\xi), t) \quad \mbox{for} \quad (x,\xi,t)\in\Sigma_- \times \R_+, 
\eq
where $\mathcal{R}_x : \R^3 \to \R^3,  \,\xi \mapsto \xi - 2(\xi\cdot r(x)) r(x)$ is the reflection operator, see Figure \ref{Fig_spfl}. Note that this operator preserves the magnitude, i.e., $|\mathcal{R}_x (\xi)| = |\xi|$. For the fluid system in \eqref{A-1}, we simply consider the following homogeneous Dirichlet boundary condition for the fluid velocity:
\bq\label{bdy_ho}
v(x,t) = 0 \quad \mbox{for} \quad (x,t) \in \pa \om \times \R_+.
\eq
\begin{figure}[!ht]
\centering
\includegraphics[width=.45\textwidth,clip]{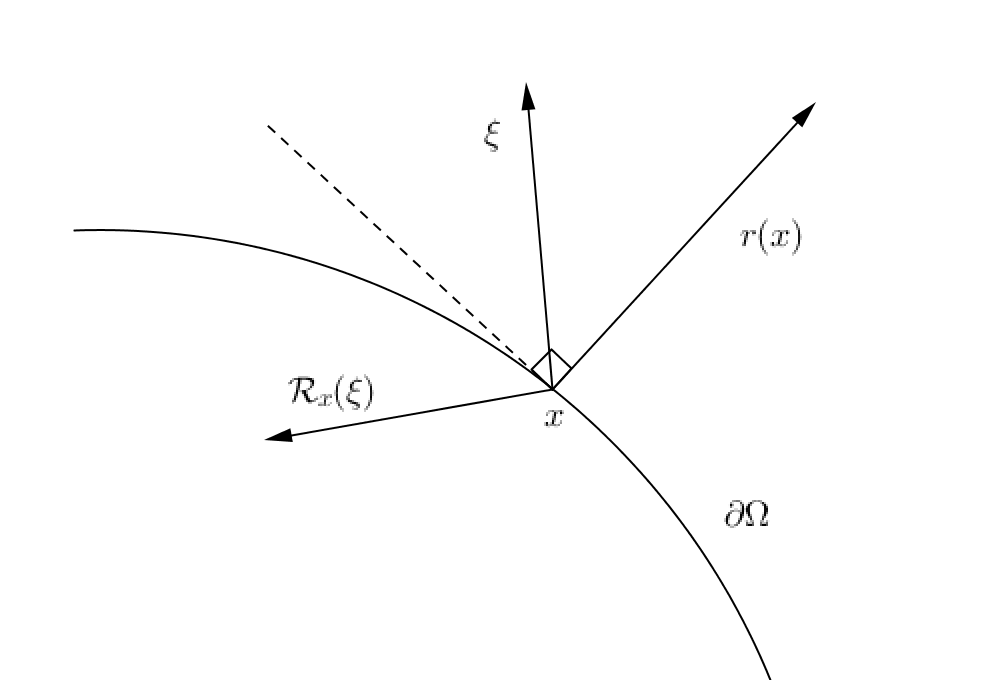}
\caption{
Specular reflection boundary condition
}
\label{Fig_spfl}
\end{figure}
\subsection{Formal hydrodynamic limit}
For the asymptotic analysis of the system \eqref{A-1}, we deal with the regime corresponding to strong noise and local alignment force, i.e., $\sigma = \alpha = 1/\e$ with $\e>0$ small enough: 
\begin{align}
\begin{aligned}\label{A-2}
&\partial_t f^\e + \xi \cdot \nabla_x f^\e + \nabla_\xi \cdot ((v^\e -\xi)f^\e) = \frac{1}{\e} \nabla_\xi \cdot (\nabla_\xi f^\e - (u^\e - \xi)f^\e),\\
&\partial_t n^\e + \nabla_x \cdot (n^\e v^\e)=0,\\
&\partial_t (n^\e v^\e) + \nabla_x \cdot (n^\e v^\e \otimes v^\e) + \nabla_x p(n^\e)  -\Delta_x v^\e = -\int_{\bbr^3} (v^\e -\xi)f^\e d\xi. 
\end{aligned}
\end{align}
It is worth noticing that the right hand side of the kinetic equation in \eqref{A-2} can be rewritten as
\[ 
\nabla_\xi \cdot [\nabla_\xi f^\varepsilon - (u^\varepsilon - \xi)f^\varepsilon] = \nabla_\xi \cdot \left(M_{f^\varepsilon} \nabla_\xi \left( \frac{f^\varepsilon}{M_{f^\varepsilon}} \right)\right),
\]
where $M_{f^\varepsilon} = M_{f^\varepsilon}(x,\xi,t)$ is the Maxwellian given by
\[
M_{f^\varepsilon}(x,\xi,t) := \frac{1}{(2\pi)^{3/2}} \exp\lt(-\frac{|\xi - u^\varepsilon(x,t)|^2}{2}\rt). 
\]
Let us briefly mention about the formal derivation of the limiting system of \eqref{A-2} as $\e \to 0$. By taking into account the local moments, we can derive a system of local balanced laws:
\begin{align}\label{eq_formal}
\begin{aligned}
&\pa_t \rho^\e + \nabla_x \cdot (\rho^\e u^\e) = 0,\cr
&\pa_t (\rho^\e u^\e) + \nabla_x \cdot (\rho^\e u^\e \otimes u^\e) +\nabla_x \cdot P^\e = \rho^\e(v^\e-u^\e),\cr
&\partial_t n^\e + \nabla_x \cdot (n^\e v^\e)=0,\\
&\partial_t (n^\e v^\e) + \nabla_x \cdot (n^\e v^\e \otimes v^\e) + \nabla_x p(n^\e)  -\Delta_x v^\e = - \rho^\e(v^\e -u^\e),
\end{aligned}
\end{align}
where $\rho^\e$, $u^\e$, and $P^\e$ denote the local particle density, local particle velocity, and pressure, respectively, given by
\[
\rho^\e(x,t) := \int_{\R^3} f^\e(x,\xi,t)\,d\xi, \quad u^\e(x,t) := \int_{\R^3} \xi f^\e(x,\xi,t)\,d\xi \bigg/  \int_{\R^3} f^\e(x,\xi,t)\,d\xi,
\]
and
\[
P^\e(x,t) := \int_{\R^3} (\xi - u^\e(x,t))\otimes (\xi -u^\e(x,t))f^\e(x,\xi,t)\,d\xi.
\]
Here $\cdot \otimes \cdot$ stands for $(a \otimes b)_{ij} = a_ib_j$ for $a = (a_1,a_2,a_3) \in \R^3, b= (b_1,b_2,b_3) \in \R^3$. Indeed, the continuity and momentum equations of Euler system in \eqref{eq_formal} can be obtained by multiplying the kinetic equation in \eqref{A-2} by $1$ and $\xi$, and integrating the resulting equations with respect to $\xi$. We notice that the system \eqref{eq_formal} is not closed in the sense that the Euler system in \eqref{eq_formal} can not be expressed in terms of $\rho^\e$ and $u^\e$; it still involves the dynamics of $f^\e$. We may further estimate a higher moment of $f^\e$ by introducing the local energy density: 
\[
\rho^\e\lt(e^\e + \frac{|u^\e|^2}{2}\rt) := \int_{\R^3} \frac{|\xi|^2}{2} f^\e\,d\xi.
\]
However, it is easy to check that the equation of local energy is still not closed. On the other hand, if we have
\[
\rho^\e  \to \rho \mbox{ and } u^\varepsilon \to u \quad \mbox{as} \quad \e \to 0,
\] 
then the kinetic density $f^\e$ satisfies
\bq\label{Max}
f^\varepsilon \to M_{\rho,u}:= \frac{\rho(x,t)}{(2\pi)^{3/2}} \exp\lt(-\frac{|\xi-u(x,t)|^2}{2}\rt) \quad \mbox{as} \quad \varepsilon \to 0. 
\eq
This implies
\[
P^\e \to \rho \mathbb{I}_{3 \times 3} \quad \mbox{as} \quad \e \to 0,
\]
where $\mathbb{I}_{3 \times 3}$ denotes the $3 \times 3$ identity matrix. As a consequence $\nabla_x \cdot P^\e$ converges toward $\nabla_x \rho$, and which gives the isothermal pressure law for the Euler system in \eqref{eq_formal}. Furthermore, if the corresponding the fluid density $n^\e$ and velocity $v^\e$ also converge toward $n$ and $v$ as $\e \to 0$, respectively, we infer the limiting functions $(\rho,u,n,v)$ are solutions to the following two-phase fluid system consisting of the isothermal Euler and the isentropic Navier-Stokes equations:
\begin{align}
\begin{aligned}\label{A-3}
&\partial_t \rho + \nabla_x \cdot (\rho u) = 0, \quad (x,t) \in \om \times \R_+,\\
&\partial_t (\rho u) + \nabla_x \cdot (\rho u \otimes u) + \nabla_x \rho = \rho(v-u),\\
&\partial_t n + \nabla_x \cdot (nv) = 0,\\
&\partial_t (nv) + \nabla_x \cdot (nv \otimes v) + \nabla_x p(n) - \Delta_x v  = -\rho(v-u). 
\end{aligned}
\end{align}
In order to derive the boundary condition for the limiting system, we multiply $(\xi \cdot r(x))$ to the above specular reflection boundary condition \eqref{bdy_ref}, and then integrate the resulting equation over the incoming velocities at $x \in \pa \om$. This gives the following equality:
\bq\label{est_00}
\int_{\xi \cdot r(x) < 0} \gamma_- f^\e(x,\xi,t) (\xi \cdot r(x))\,d\xi = \int_{\xi \cdot r(x) < 0} \gamma_+ f^\e(x,\mathcal{R}_x(\xi),t) (\xi \cdot r(x))\,d\xi.
\eq
We further use the change of variables $\xi_* = R_x(\xi)$ to have
\bq\label{est_01}
\int_{\xi \cdot r(x) < 0} \gamma_+ f^\e(x,\mathcal{R}_x(\xi),t) (\xi \cdot r(x))\,d\xi = -\int_{\xi_* \cdot r(x) > 0} \gamma_+ f^\e(x,\xi_*,t) (\xi_* \cdot r(x))\,d\xi_*.
\eq
We now combine \eqref{est_00} and \eqref{est_01} to get
\[
\int_{\xi \cdot r(x) < 0} \gamma_- f^\e(x,\xi,t) (\xi \cdot r(x))\,d\xi = -\int_{\xi_* \cdot r(x) > 0} \gamma_+ f^\e(x,\xi_*,t) (\xi_* \cdot r(x))\,d\xi_*,
\]
and this subsequently implies
\[
\int_{\R^3} \gamma f^\e(x,\xi,t) (\xi \cdot r(x))\,d\xi =0, \quad \mbox{i.e.,} \quad (\rho^\e u^\e)(x,t) \cdot r(x) = 0.
\]
Taking the limit $\e \to 0$ leads to the kinematic boundary condition for the Euler equations in \eqref{A-3} and homogeneous Dirichlet boundary condition for the Navier-Stokes equations in \eqref{A-3}:
\bq\label{A-3_bdy}
u \cdot r \equiv 0 \quad \mbox{and} \quad v \equiv 0 \quad \mbox{on} \quad \pa\om \times \R_+.
\eq

Before closing this subsection, we introduce several notations used throughout the paper. For functions $f(x,v)$ or $g(x)$, $\|f\|_{L^p}$ or $\|g\|_{L^p}$ denote the usual $L^p(\om \times \R^3)$-norm or $L^p(\om)$-norm. We also denote by $C$ a generic positive constant which may differ from line to line, and $C = C(\alpha,\beta,\dots)$ represents the positive constant depending on $\alpha,\beta,\dots$. For notational simplicity, we drop $x$-dependence of differential operators, i.e., $\nabla f := \nabla_x f$ and $\Delta f = \Delta_x f$. For any nonnegative integer $k$ and $p \in [1,\infty]$, $W^{k,p} := W^{k,p}(\Omega)$ stands for the $k$-th order $L^p$ Sobolev space. In particular, if $p=2$, we denote by $H^k := H^k(\Omega) = W^{k,2}(\Omega)$ for any $k \in \bbn$. Moreover, we let $\mathcal{C}^k(I;\mathcal{B})$ be the set of $k$-times continuously differentiable functions from an interval $I$ to a Banach space $\mathcal{B}$ and we introduce the following Banach spaces:
\[
\mathfrak{X}_r^s = \mathfrak{X}_r^s(T,\Omega) := \bigcap_{k=0}^{s-r}\mathcal{C}^k([0,T];H^{s-k}(\Omega)), \quad \mathfrak{X}_0^s =: \mathfrak{X}^s, \quad \mbox{and} \quad \|h(t)\|_{\mathfrak{X}_r^s}:=  \sum_{k=0}^{s-r} \left\| \frac{\partial^k f}{\partial t^k}(t)\right\|_{H^{s-k}}. 
\]

\subsection{Main results} In the present work, our main goal is to make the formal derivation discussed in the previous subsection rigorous. For this, we first need to develop some existence theories for the systems \eqref{A-1} (or \eqref{A-2}) and \eqref{A-3} with the boundary conditions mentioned above. More precisely, we need to investigate the existence of weak solutions to the Vlasov-Fokker-Planck/Navier-Stokes system \eqref{A-1} and the existence and uniqueness of strong solutions to the Euler/Navier-Stokes system \eqref{A-3} at least locally in time. 

We introduce the notion of weak solutions to the system \eqref{A-1} and state the result of global-in-time existence of weak solutions.

\begin{definition}\label{D2.1}
Let $T > 0$. We say a three-tuple $(f,n,v)$ is a weak solution to \eqref{A-1} with the boundary conditions \eqref{bdy_ref} and  \eqref{bdy_ho} on the time interval $[0,T]$ if it satisfies the following conditions:
\begin{itemize}
\item[(i)]
$f \in \mathcal{C}([0,T]; L_+^1(\Omega \times \bbr^3)) \cap L^\infty(0,T; (L^1\cap L^\infty)(\Omega \times \bbr^3))$, $|\xi|^2 f \in L^\infty(0,T; L^1(\Omega\times\bbr^3))$,
\item[(ii)]
$n \in \mathcal{C}([0,T];L_+^1(\Omega )) \cap L^\infty(0,T;L^\gamma(\Omega))$,
\item[(iii)]
$v \in L^2(0,T; H_0^1(\Omega))$, $n|v|^2 \in L^\infty(0,T;L^1(\Omega))$, $nv \in \mathcal{C}([0,T];L^{2\gamma/(\gamma+1)}_w (\Omega))\footnote{$L^p_w$ denotes the weak $L^p$ space, i.e., $L^p_w(\om) = \{ f: |\{x \in \om : |f(x)| > \lambda \}| \leq \lambda^{-p} \}$ for $p \in (0,\infty)$ and $L^\infty_w(\om) = L^\infty(\om)$.}$,
\item[(iv)] for any $\varphi \in \mathcal{C}_c^2(\bar\Omega \times \bbr^3 \times [0,T])$ with $\gamma_- \varphi(x,\xi,t) = \gamma_+ \varphi(x,\mathcal{R}_x(\xi),t)$ and $\varphi(\cdot,\cdot,T) =0$,
\begin{align*}\begin{aligned}
&\int_0^T  \int_{\Omega \times \bbr^3} f \lt( \partial_t \varphi + \xi \cdot \nabla \varphi + (v-\xi ) \cdot \nabla_\xi \varphi + \Delta_\xi \varphi + (u-\xi)\cdot \nabla_\xi \varphi \rt) dxd\xi dt\\
&\quad  +\int_{\Omega \times \bbr^3} f_0 \varphi(x,\xi,0) \,dxd\xi =0,
\end{aligned}\end{align*}
\item[(v)] for any $\phi \in \mathcal{C}_c^2(\bar\Omega \times [0,T])$ with $\phi(\cdot,T) =0$,
\begin{align*}\begin{aligned}
&\int_0^T \int_\om n (\pa_t \phi + v \cdot \nabla \phi)\,dxdt + \int_\om n_0 \phi(x,0)\,dx = 0,\cr
&\int_0^T \int_\om \lt(nv \pa_t \phi + n(v\otimes v) \nabla \phi + p \nabla \phi + v \Delta \phi - \phi\int_{\R^3} (v-\xi)f \,d\xi\rt)dxdt \cr
&\quad + \int_\om n_0 v_0 \phi(x,0)\,dx = 0.
\end{aligned}\end{align*}
\end{itemize}
\end{definition}
\begin{remark}Since we are not able to have $\gamma_\pm f \in L^1(0,T;L^1(\Sigma_\pm))$,  the additional condition for the test function $\varphi$ which appeared in Definition \ref{D2.1} (iv) is imposed.
\end{remark}

\begin{theorem}\label{T0.1}
Let $\gamma >3/2$, $T \in (0,\infty)$ and assume that the initial data $(f_0, n_0, v_0)$ satisfy
\begin{equation*}
f_0 \in (L_+^1\cap L^\infty)(\Omega \times \bbr^3), \quad n_0 \in L_+^1(\Omega), \quad \mbox{and} \quad \mathcal{F}(f_0, n_0, v_0) < \infty,
\end{equation*}
where $\mathcal{F}(f,n,v)$ is defined as
\[
\mathcal{F}(f, n, v) := \int_{\Omega \times \bbr^3 } f \left( \log f  +\frac{|\xi|^2}{2} \right) dxd\xi + \int_{\Omega} \frac{1}{2} n |v|^2 \,dx + \frac{1}{\gamma-1} \int_{\Omega} n^\gamma \,dx. 
\]
Then there exists at least one weak solution $(f,n,v)$ to the system \eqref{A-1} with the specular reflection boundary condition \eqref{bdy_ref} for $f$ and the homogeneous Dirichlet boundary condition \eqref{bdy_ho} for $v$ in the sense of Definition \ref{D2.1}. Moreover, the following entropy inequality holds:
$$\begin{aligned}
\mathcal{F}&(f,n,v)(t) + \int_0^t \md(f,v)(s)\,ds + \int_0^t \int_\Omega |\nabla v|^2\, dxds\\
&\le \mathcal{F}(f_0, n_0, v_0) + 3t \|f_0\|_{L^1(\Omega\times\bbr^3)},
\end{aligned}$$
where $\md(f,v)$ is given by
\[
\md(f,v) := \int_{\Omega \times \bbr^3} \frac{1}{f}|(u - \xi)f -\nabla_\xi f|^2 + |v-\xi|^2f \,dxd\xi.
\]
\end{theorem}
The study of existence theory for a system of kinetic equation and fluid system is by now well established. Let us summarize some of the previous works by dividing them into two cases: the coupling with incompressible and compressible fluids. In the coupling with an incompressible fluid case, the  global-in-time existence of weak solutions for the Vlasov/Stokes system is obtained in \cite{Ham98}, and this result is extended to the Vlasov/Navier-Stokes equations in the spatial periodic domain \cite{B-D-G-M} and in the bounded domain with the specular reflection boundary condition \cite{Yu13}. Other particle interaction forces are also considered in the Vlasov equations, for instance, velocity-alignment force \cite{B-C-H-K1,B-C-H-K2,CHJK19}, BGK collision operator \cite{CLYpre,CYpre}, particle breakup operator \cite{YY18}, and the weak/strong solutions are discussed. Interactions between charged particles and incompressible fluid are also studied in \cite{AIK14, AKS10} by taking into account the Vlasov-Poisson/Navier-Stokes system in a bounded domain. In the presence of diffusion, the global-in-time weak solutions of the Vlasov-Fokker-Planck/Navier-Stokes system are obtained in \cite{C-C-K,CKL11}. For this system, the global-in-time classical solutions near equilibrium are also studied in \cite{C-D-M,G-H-M-Z}. For the coupling with a compressible fluid, the local-in-time existence and uniqueness of classical solutions for the Vlasov and Vlasov-Boltzmann equation coupled with compressible Euler equations are investigated in \cite{BD06} and \cite{Math10}, respectively. For the Vlasov-Fokker-Planck/Navier-Stokes system, the global-in-time existence of weak solutions is found in \cite{M-V} in a bounded domain with various types of boundary conditions. We also refer to \cite{C3} for the a priori large-time behavior estimate of solutions showing the particle velocity is aligned with the fluid velocity and \cite{C4} for the finite-time blow-up phenomena.

In order to establish the global-in-time existence of weak solutions to the coupled kinetic-fluid system \eqref{A-1}, we employ a similar idea of \cite{M-V} in which the weak solutions are obtained for the system \eqref{A-1} without the local velocity alignment force $\nabla_\xi \cdot ((\xi - u)f)$. On the other hand, due to the lack of regularity of the local alignment force, we need one more step to regularize this term, and as a result, more careful analysis is required when we estimate the uniform bounds of solutions with respect to the regularization parameters. Even though our result on the hydrodynamic limit only works for the system \eqref{A-1} with the specular reflection boundary condition for the kinetic equation, we can even obtain the global-in-time existence of weak solutions to the system \eqref{A-1} with more general types of boundary conditions as in \cite{M-V}, see Section \ref{sec:3} for more detailed discussion.

In Section \ref{sec:4}, we discuss the existence and uniqueness of strong solutions to the limiting system \eqref{A-3}. Before introducing the notion of solutions, we reformulate the system \eqref{A-3} with new functions:  
\[
g := \log M \rho \quad \mbox{with} \quad M = |\om| > 0 \quad \mbox{and}  \quad h := n - \int_\Omega n \,dx =: n - n_c.
\]
Without loss of generality, we may assume that $n_c(t) = 1$ for $t \geq 0$ due to the conservation of mass. Then by using these newly defined functions, at the formal level, we can reformulate the system \eqref{A-3} as follows:
\begin{align}\label{E-3}
\begin{aligned}
&\partial_t g + u \cdot \nabla g + \nabla \cdot u = 0, \quad (x,t) \in \om \times \R_+,\\
&\partial_t u + u \cdot \nabla u + \nabla g = (v-u),\\
&\partial_t h + \nabla \cdot (hv) + \nabla \cdot v = 0,\\
&\partial_t v + v \cdot \nabla  v + \nabla p(1+h) - \frac{1}{1+h}\Delta v  = -\frac{e^{g}}{M(1+h)}(v-u),
\end{aligned}
\end{align}
subject to initial data and boundary conditions:
\begin{align}\label{E-4}
\begin{aligned}
&(g(x,0), u(x,0), h(x,0), v(x,0)) = (g_0(x), u_0(x), h_0(x), v_0(x)), \quad x \in \Omega, \\
&u(x,t) \cdot r(x)  = 0, \quad v(x,t) = 0, \quad (x,t)\in \partial\Omega \times \R_+.
\end{aligned}
\end{align}
We provide the definition of strong solutions to the system \eqref{E-3}--\eqref{E-4}.
\begin{definition}\label{D2.2}
Let $T \in (0,\infty)$ and $s \ge 4$. We say a four-tuple $(g,u,h,v)$ is a strong solution to \eqref{E-3}--\eqref{E-4} on the time interval $[0,T]$ if it satisfies the following conditions:
\begin{itemize}
\item[(i)] $(g, u, h, v) \in \mathfrak{X}^s(T,\Omega) \times  \mathfrak{X}^s(T,\Omega) \times \mathfrak{X}^s(T,\Omega) \times \mathfrak{X}^s(T,\Omega) $. 
\item[(ii)] $(g,u,h,v)$ satisfies \eqref{E-3}--\eqref{E-4} in the distributional sense. 
\end{itemize}
\end{definition}

\begin{remark} Let $T>0$. We can easily deduce that $(\rho,u,n,v) \in \mc^2(\om \times [0,T])$ solves the system \eqref{A-3} and \eqref{A-3_bdy} with $\rho > 0$ and $n > 0$ if and only if $(g,u,h,v) \in \mc^2(\om \times [0,T])$ solves the system \eqref{E-3} and \eqref{E-4} with $e^g > 0$ and $1+h > 0$.
\end{remark}

\begin{theorem}\label{T2.3}
Let $T \in (0,\infty)$ and $s\ge 4$. Suppose that the initial data $(g_0, u_0, h_0, v_0)$ satisfy the following regularity and smallness assumptions:
\begin{itemize}
\item[(i)] $(g_0, u_0, h_0, v_0) \in H^s(\Omega)\times H^s(\Omega)\times H^s(\Omega)\times H^s(\Omega)$.
\item[(ii)] The initial data satisfies the smallness condition:
\[
\|g_0\|_{H^s} + \|u_0\|_{H^s} + \|h_0\|_{H^s} + \|v_0\|_{H^s} < \e, 
\]
where $(1>)\,\e>0$ is a sufficiently small constant. 
\item[(iii)]
The initial data satisfies the compatibility conditions up to order s:
\[
\partial_t^k u(x,t) \cdot r(x)|_{t=0} =0, \quad \partial_t^k v(x,t)|_{t=0} =0, \quad x\in\partial\Omega, \quad k=0,1,\dots,s-1.
\]
\end{itemize}
Then, the initial-boundary value problem \eqref{E-3}--\eqref{E-4} has a unique solution 
$(g, u, h, v)$ on the time interval $[0,T]$ in the sense of Definition \ref{D2.2}. Moreover, we have 
\[
\sup_{0 \leq t \leq T}\lt(\|g(t)\|_{\mathfrak{X}^s} + \|u(t)\|_{\mathfrak{X}^s} + \|h(t)\|_{\mathfrak{X}^s} + \|v(t)\|_{\mathfrak{X}^s}\rt)< \e^{1/2}.
\]

\begin{remark}
Since $s\ge 4$, the solution $(g,u,h,v)$ obtained in Theorem \ref{T2.3} belongs to $\mc^2(\om \times [0,T])$ and thanks to the smallness condition, $e^g > 0$ and $1+h > 0$ also hold. Thus, $(\rho,u,n,v) \in \mc^2(\om \times [0,T])$ solves the system \eqref{A-3} and \eqref{A-3_bdy} and moreover, we can also deduce that
\[
\sup_{0 \leq t \leq T}\lt(\|\rho(t)\|_{\mathfrak{X}^s} + \|u(t)\|_{\mathfrak{X}^s} + \|n(t)\|_{\mathfrak{X}^s} + \|v(t)\|_{\mathfrak{X}^s}\rt)< \infty.
\]
\end{remark}
\end{theorem}
There are few studies on the Cauchy problem for the coupled fluid system through the drag force. The compressible Euler equations coupled with the incompressible/compressible Navier-Stokes equations are dealt with in \cite{C0,C,CK16}, and the global-in-time classical solutions and large-time behavior are provided under suitable assumptions on the initial data such as smallness and regularity. In these works, the smoothing effect from the viscosity in the Navier-Stokes system is crucially used via the drag force to prevent the possible singularity formation of solutions from the Euler equations. To the best of the authors' knowledge, the existence theory for the coupled Euler and Navier-Stokes equations in a bounded domain has never been studied before. In Section \ref{sec:4}, we study the global-in-time existence and uniqueness of strong solutions to the system \eqref{E-3} with the initial-boundary conditions \eqref{E-4}. One of the main difficulties in analyzing the existence of solutions arises from the fact that the kinematic boundary condition is given for the Euler equations in \eqref{E-3}. This makes it difficult to estimate the standard $H^s$-estimate of solutions. For that reason, we also estimate the vorticity, which is the curl of the fluid velocity, in the Euler solutions in \eqref{E-3}. We then use the estimate showing that the function in $H^s(\om)$ with the kinematic boundary condition can be bounded from above by the sum of $H^{s-1}(\om)$ norms of the divergence, curl of that function, and itself. See Section \ref{sec:4} for more details.

We now state our main result on the hydrodynamic limit showing that the weak solutions to the Vlasov-Fokker-Planck/Navier-Stokes system \eqref{A-2} converge to the strong solution to the Euler/Navier-Stokes system \eqref{A-3} as $\e$ goes to $0$. 
\begin{theorem}\label{T2.2}
Let $T>0$, $\gamma > 3/2$, and let $(f^\e, n^\e, v^\e)$ be a weak solution to the system \eqref{A-2} with the boundary conditions \eqref{bdy_ref} and \eqref{bdy_ho} on the time interval $[0,T]$ with the initial data $(f^\e_0, n^\e_0, v^\e_0)$ in the sense of Definition \ref{D2.1}. Let $(\rho,u,n,v)$ be a unique strong solution to the system \eqref{A-3} on the time interval $[0,T]$ with the initial data $(\rho_0,u_0,n_0,v_0)$ in the sense of Definition \ref{D2.2} satisfying
\[
\inf_{(x,t) \in \Omega \times [0,T]} \rho(x,t) > 0 \quad \mbox{and} \quad \inf_{(x,t) \in \Omega \times [0,T]}n(x,t) > 0.
\]
Suppose that the initial data $(f^\e_0, n^\e_0, v^\e_0)$ and $(\rho_0,u_0,n_0,v_0)$ are well-prepared so that they satisfy the following assumptions:
\begin{itemize}
\item[{\bf (H1)}] 
$$\begin{aligned}
&\quad \int_{\Omega}\left( \int_{\R^3} f_0^\varepsilon \left(1 +  \log f_0^\varepsilon  +\frac{|\xi|^2}{2}\right) d\xi + \frac{1}{2} n_0^\varepsilon |v_0^\varepsilon|^2 + \frac{1}{\gamma-1}(n_0^\e)^\gamma \right) dx \cr
&\qquad - \int_{\Omega} \left(\rho_0 \left(1 + \log \rho_0 + \frac{|u_0|^2}{2}\right) + \frac12 n_0|v_0|^2 + \frac{1}{\gamma-1}n_0^\gamma\right)dx \cr
&\qquad \quad = \mathcal{O}(\sqrt\e),
\end{aligned}$$
\item[{\bf (H2)}] 
$$\begin{aligned}
&\int_{\Omega} \rho_0^\e |u_0^\e - u_0|^2\,dx + \int_{\Omega} n_0^\e|v_0^\e - v_0|^2\,dx \cr
& \quad + \int_{\Omega}\int_{\rho_0}^{\rho_0^\e} \frac{\rho_0^\varepsilon - z}{z}\,dzdx + \int_{\Omega} \left(n_0^\e\int_{n_0}^{n_0^\e} \frac{p(z)}{z^2}\,dz - \frac{p(n_0)}{n_0}(n_0^\e - n_0)\right)dx \cr
&\qquad \quad = \mathcal{O}(\sqrt\e),
\end{aligned}$$
\end{itemize}
Then we have 
\begin{align*}\begin{aligned}
&\int_{\Omega} \rho^\e|u^\e - u|^2\,dx + \int_{\Omega} n^\e|v^\e - v|^2\,dx + \int_{\Omega}\int_{\rho}^{\rho^\e} \frac{\rho^\varepsilon - z}{z}\,dzdx \cr
&\quad + \int_{\Omega} \left(n^\e\int_n^{n^\e} \frac{p(z)}{z^2}\,dz - \frac{p(n)}{n}(n^\e - n)\right)dx\cr
&\quad + \int_0^t \int_{\Omega} |\nabla (v-v^\e)|^2\,dxds +\int_0^t \int_{\Omega} \rho^\varepsilon|( u^\varepsilon - v^\varepsilon) - (u-v)|^2\, dxds\\
& \qquad \le C\sqrt{\varepsilon},
\end{aligned}\end{align*}
where $C$ is a positive constant independent of $\e$.
\end{theorem}
\begin{remark}The condition on $\gamma$ follows from the weak solvability of the compressible isentropic Navier-Stokes system in \eqref{A-1}, see Theorem \ref{T0.1}. In fact, the result of hydrodynamic limit in Theorem \ref{T2.2} ``a priori'' holds for $\gamma \geq 1$. Thus the result in Theorem \ref{T2.2}, as a consequence Corollary \ref{cor_main}, can be extended to $\gamma \ge 1$ once the existence of weak solutions to the system \eqref{A-1} is resolved for the case $\gamma \geq 1$.
\end{remark}
\begin{remark}The assumptions {\bf(H1)}-{\bf(H2)} can be replaced as follows:
\begin{itemize}
\item[{\bf (H1)$'$}] 
$$\begin{aligned}
&\int_{\Omega}\left( \int_{\R^3} f_0^\varepsilon \left( \log f_0^\varepsilon  +\frac{|\xi|^2}{2}\right) d\xi \right)dx - \int_{\Omega}\rho_0 \left( \log \rho_0 + \frac{|u_0|^2}{2}\right)dx  = \mathcal{O}(\sqrt\e).
\end{aligned}$$

\item[{\bf (H2)$'$}] 
$$\begin{aligned}
&\|\rho_0^\e - \rho_0\|_{L^2} = \mathcal{O}(\sqrt\e), \quad \|n_0^\e - n_0\|_{L^\gamma} =\mathcal{O}(\sqrt{\e}),\cr
&\|u_0^\e - u_0\|_{L^\infty} =\mathcal{O}(\sqrt{\e}), \quad \mbox{and} \quad \|v_0^\e - v_0\|_{L^\infty} =\mathcal{O}(\sqrt{\e}).
\end{aligned}$$
\end{itemize}
Note that the above assumptions are stronger than {\bf(H1)}-{\bf(H2)}, i.e., {\bf (H1)$'$}-{\bf (H2)$'$} imply {\bf(H1)}-{\bf(H2)}. Indeed, {\bf(H1)} can be deduced from {\bf(H1)$'$}-{\bf(H2)$'$} as follows:
\begin{align*}
\begin{aligned}
&\int_{\Omega}\left( \int_{\R^3} f_0^\varepsilon \left(1 +  \log f_0^\varepsilon  +\frac{|\xi|^2}{2}\right) d\xi + \frac{1}{2} n_0^\varepsilon |v_0^\varepsilon|^2 + \frac{1}{\gamma-1}(n_0^\e)^\gamma \right) dx \cr
&\quad - \int_{\Omega} \left(\rho_0 \left(1 + \log \rho_0 + \frac{|u_0|^2}{2}\right) + \frac12 n_0|v_0|^2 + \frac{1}{\gamma-1}n_0^\gamma\right)dx \cr
&\qquad \le \|\rho_0^\e -\rho_0\|_{L^1}  +\mathcal{O}(\sqrt{\e}) + \int_\Omega \left(\frac{1}{2}(n_0^\e |v_0^\e|^2 - n_0 |v_0|^2) + \frac{1}{\gamma-1} ((n_0^\e)^\gamma - (n_0)^\gamma)\right) dx   \\
&\qquad \le \mathcal{O}(\sqrt{\e}) +  \frac{1}{2}\int_\Omega |n_0^\e-n_0| |v_0|^2 + n_0^\e (v_0^\e - v_0)(v_0^\e + v_0) \,dx\\
&\qquad \quad + \frac{\gamma}{\gamma-1}\int_{\Omega }\max\{ (n_0^\e)^{\gamma-1}, (n_0)^{\gamma-1}\} |n_0^\e - n_0|\,dx \\
&\qquad \le   \mathcal{O}(\sqrt{\e}) + \|v_0\|_{L^\infty}^2 \|n_0^\e - n_0\|_{L^1}+ \|n_0^\e\|_{L^1}\max\{\|v_0^\e\|_{L^\infty}, \|v_0\|_{L^\infty}\} \|v_0^\e - v_0\|_{L^\infty}\\
&\qquad \quad + \frac{\gamma}{\gamma-1} \left(\int_\Omega \max\{ (n_0^\e)^\gamma, (n_0)^\gamma\}\,dx\right)^{1/\gamma^*}\|n_0^\e - n_0\|_{L^\gamma} \\
&\qquad \le \mathcal{O}(\sqrt{\e}),
\end{aligned}
\end{align*}
where we used 
\[
\|v_0^\e\|_{L^\infty} \leq \|v_0^\e - v_0\|_{L^\infty} + \|v_0\|_{L^\infty} \leq \mathcal{O}(\sqrt\e) + \|v_0\|_{L^\infty},
\] 
and $\gamma^*$ denotes the H\"older conjugate of $\gamma$. In order to deduce {\bf(H2)} from {\bf(H1)$'$}-{\bf(H2)$'$}, we first observe 
\bq\label{est_rho0}
\int_{\rho_0}^{\rho_0^\e} \frac{\rho_0^\varepsilon - z}{z}\,dz = (\rho_0^\e - \rho_0)^2 \int_0^1 \int_0^{\tau} h''(s\rho_0^\e + (1-s) \rho_0)\,ds d\tau,
\eq
where $h(\rho) = \rho\log \rho$. By employing a similar argument as in \cite[Lemma 2.4]{LT13}, see also \cite[Remark 2.5]{LT13}, we consider two cases; (i) $0 \leq \rho^\e_0 \leq R_0$ and (ii) $\rho^\e_0 > R_0$ for some $R_0 > 0$. For the first case (i), we consider the following positive function:
$$\begin{aligned}
B(\rho^\e_0, \rho_0) &:= \int_{\rho_0}^{\rho_0^\e} \frac{\rho_0^\varepsilon - z}{z}\,dz \bigg/(\rho_0^\e - \rho_0)^2\cr
&=\int_0^1 \int_0^{\tau} h''(s\rho_0^\e + (1-s) \rho_0)\,ds d\tau,
\end{aligned}$$
for $\rho^\e_0 \in [0,R_0]$, $\rho_0 \in K := [\inf_{x\in\om} \rho_0(x), \sup_{x\in\om}\rho_0(x)]$ with $0<\inf_{x\in\om} \rho_0(x) \leq \sup_{x\in\om} \rho_0(x) < \infty$, and $\rho^\e_0 \neq \rho_0$. On the other hand, we readily check 
\[
\lim_{\rho^\e_0 \to \rho_0} B(\rho^\e_0, \rho_0) = \frac{h''(\rho_0)}{2} = \frac{1}{2\rho_0} > 0,
\]
and this implies $B(\rho^\e_0, \rho_0) \in \mc([0,R_0] \times K)$. Thus there exists $C>0$, which is independent of $\e$, such that $B(\rho^\e_0, \rho_0) \leq C$ for $\rho^\e_0 \in [0,R_0]$ and $\rho_0 \in K$, that is, 
\[
\int_{\rho_0}^{\rho_0^\e} \frac{\rho_0^\varepsilon - z}{z}\,dz  \leq C(\rho_0^\e - \rho_0)^2.
\]
For the other case (ii), we estimate
\[
\int_0^1 \int_0^{\tau} h''(s\rho_0^\e + (1-s) \rho_0)\,ds d\tau \leq \frac12\max\lt\{\frac{1}{\rho^\e_0}, \frac{1}{\rho_0} \rt\} \leq \frac12 \max\lt\{\frac{1}{R_0}, \frac{1}{\inf_{x\in\om} \rho_0(x)} \rt\} \leq C,
\]
for some $C>0$. Thus for both cases (i) and (ii) we obtain from \eqref{est_rho0} that
\[
\int_{\rho_0}^{\rho_0^\e} \frac{\rho_0^\varepsilon - z}{z}\,dz \leq C (\rho_0^\e - \rho_0)^2,
\]
where $C>0$ is independent of $\e$. This together with the first-order Taylor approximation yields
\begin{align*}
\begin{aligned}
&\int_{\Omega} \rho_0^\e |u_0^\e - u_0|^2\,dx + \int_{\Omega} n_0^\e|v_0^\e - v_0|^2\,dx \cr
& \quad + \int_{\Omega}\int_{\rho_0}^{\rho_0^\e} \frac{\rho_0^\varepsilon - z}{z}\,dzdx + \int_{\Omega} \left(n_0^\e\int_{n_0}^{n_0^\e} \frac{p(z)}{z^2}\,dz - \frac{p(n_0)}{n_0}(n_0^\e - n_0)\right)dx \cr
&\qquad \le \|\rho_0^\e\|_{L^1}\|u_0^\e - u_0\|_{L^\infty}^2 + \|n_0^\e\|_{L^1}\|v_0^\e - v_0\|_{L^\infty}^2 \\
&\qquad \quad + C\|\rho_0^\e - \rho_0\|_{L^2}^2 + \frac{1}{\gamma-1}\int_\Omega \left((n_0^\e)^\gamma - (n_0)^\gamma - \gamma (n_0)^{\gamma-1}(n_0-n_0^\e) \right)dx\\
&\qquad \le \mathcal{O}(\sqrt{\e}) + \frac{2\gamma}{\gamma-1} \left(\int_\Omega \max\{ (n_0^\e)^\gamma, (n_0)^\gamma\}\,dx\right)^{1/\gamma^*} \|n_0^\e - n_0\|_{L^\gamma} \\
&\qquad \le \mathcal{O}(\sqrt{\e}).
\end{aligned}
\end{align*}

\end{remark}
\begin{corollary}\label{cor_main}Suppose that all the assumptions in Theorem \ref{T2.2} hold. Then the following strong convergences of weak solutions $(f^\varepsilon, n^\varepsilon, v^\varepsilon)$ to the system \eqref{A-2} towards the strong solutions $(\rho,u,n,v)$ to the system \eqref{A-3} can be obtained:
$$\begin{aligned}
&(\rho^\e, n^\e) \to (\rho, n) \mbox{ a.e. and in }L^\infty(0,T; L^1(\Omega)) \times L^\infty(0,T; L^p(\Omega)) \ \forall p \in [1,\gamma],\cr
&(\rho^\e u^\e, n^\e v^\e) \to (\rho u, nv) \mbox{ a.e. and in } L^\infty(0,T; L^1(\Omega)) \times L^\infty(0,T; L^1(\Omega)),\quad \cr
&(\rho^\e u^\e \otimes u^\e, n^\e v^\e \otimes v^\e) \to (\rho u\otimes u, nv \otimes v) \mbox{ a.e. and in }L^\infty(0,T; L^1(\Omega))\times L^\infty(0,T; L^1(\Omega)), \quad \mbox{and}\cr
&\int_{\R^3} f^\e \xi\otimes \xi\,d\xi \to \rho u\otimes u + \rho \mathbb{I}_{3 \times 3} \quad \mbox{a.e.} \quad \mbox{and} \quad L^p(0,T;L^1(\om)) \quad \mbox{for} \quad 1 \leq p \leq 2
\end{aligned}$$
as $\e \to 0$. Moreover, we have
\[
\int_{\om \times \R^3} \int_{M_{\rho,u}}^{f^\e} \frac{f^\e - z}{z}\,dzdx\xi \leq \int_{\om \times \R^3} \int_{M_{\rho_0,u_0}}^{f^\e_0} \frac{f^\e_0 - z}{z}\,dzdx\xi + C\e^{1/4},
\]
where $M_{\rho,u}$ appears in \eqref{Max}. In particular, this implies
\[
\|f^\e - M_{\rho,u}\|_{L^1} \leq  C\lt(\int_{\om \times \R^3} \int_{M_{\rho_0,u_0}}^{f^\e_0} \frac{f^\e_0 - z}{z}\,dzdx\xi\rt)^{1/2} + C\e^{1/8}.
\]
\end{corollary}

There are several works on the asymptotic analysis for the kinetic equation coupled with the Navier-Stokes system. In \cite{C-G}, the asymptotic regime corresponding to the strong drag force and the strong diffusion in the system \eqref{A-1} without the local alignment force, i.e., $\alpha = 0$ is considered and the formal derivation of a two-phase fluid model from that system is studied. Later, this formal derivation is rigorously justified in \cite{M-V2}. See also \cite{G-J-V, G-J-V2} for other types of hydrodynamic limits. The kinetic equation in \eqref{A-1} coupled with the incompressible Navier-Stokes system through the drag force is discussed in \cite{C-C-K}, and the rigorous hydrodynamic limit to the isothermal Euler/incompressible Navier-Stokes system is obtained. Theorem \ref{T2.2} extends this result to the compressible fluid case as well as the initial-boundary value problem. Our main mathematical tool is the relative entropy, which is originally proposed to study the weak-strong uniqueness principle \cite{Da}, see also \cite{BNV16,BNV17, FJN12,Gou05,Y} for the applications to the kinetic and hydrodynamic equations. 

In order to emphasize one of our main results on the hydrodynamic limit, in Section \ref{sec:2}, we provide the details of the proof for Theorem \ref{T2.2} and Corollary \ref{cor_main}. As mentioned above, the standard relative entropy argument is employed, however, we need to be careful whenever we use the integration by parts and some Sobolev inequalities due to the kinematic boundary condition \eqref{A-3_bdy} for the Euler equations in \eqref{A-3}.

%
%
%
%
%

\section{Hydrodynamic limits: From kinetic-fluid to two-phase fluid system}\label{sec:2}
\setcounter{equation}{0}
In this section, we establish the rigorous derivation of system \eqref{A-3} from the system \eqref{A-2} with the specular reflection/homogeneous Dirichlet boundary conditions. Our main strategy relies on the relative entropy method. For this, we first estimate the entropy inequality for the Vlasov-Fokker-Planck/Navier-Stokes system \eqref{A-2}, and then show the relative entropy estimates. 

\subsection{Entropy estimates of the system \eqref{A-2}}
We first introduce the entropy $\mf$ and its dissipation rates $\md_1$ and $\md_2$ as follows:
\[
\mathcal{F}(f, n, v) := \int_{\Omega \times \bbr^3 } f \left( \log f  +\frac{|\xi|^2}{2} \right) dxd\xi + \int_{\Omega} \frac{1}{2} n |v|^2 \,dx + \frac{1}{\gamma-1} \int_{\Omega} n^\gamma \,dx,
\]
\[
\md_1(f) := \int_{\Omega \times \bbr^3} \frac{1}{f} |\nabla_\xi f - (u - \xi)f|^2 \,dxd\xi, 
\]
and
\[
\md_2(f, v) := \int_{\Omega\times\bbr^3} |v - \xi|^2 f \,dxd\xi + \int_\Omega |\nabla v|^2 \,dx.
\]
Before we move on to the entropy estimate of weak solutions to the initial-boundary value problem \eqref{A-2}, we provide the following lemma for the uniform estimates. For the proof, we refer to \cite{C-I-P}.
\begin{lemma}\label{LA.4}
Suppose that $f \ge 0$ and $|\xi|^2 f \in L^1(\Omega \times\bbr^3)$. Then for any $\delta>0$, there exists a constant $C=C(\delta, \Omega) > 0$ such that
\[ 
\int_{\Omega \times \bbr^3} f \log^- f \,dxd\xi \le C + \delta \int_{\Omega\times\bbr^3}|\xi|^2 f \,dxd\xi,
\]
where $\log^- f := \max\{0, -\log f\}$. Similarly, $|\xi|^2 \gamma f\in L^1(\Sigma \times (0,T))$ implies $\gamma f \log^- (\gamma f) \in L^1(\Sigma \times (0,T))$, and a similar estimate holds.
\end{lemma}
Without loss of generality, we assume $\|f_0^\e\|_{L^1}= 1$ for $\e> 0$ from now on. 
\begin{lemma}\label{L4.1}
Let $T>0$, and suppose that $(f^\e, n^\e, v^\e)$ is a weak solution to the system \eqref{A-2} on $[0,T)$ with the initial data $(f_0^\e, n_0^\e, v_0^\e)$ in the sense of Definition \ref{D2.1}. Then, we have
\begin{align*}\begin{aligned}
&\int_{\Omega \times \bbr^3} \left( \frac{|\xi|^2}{4} + |\log f^\e| \right) f^\e \,dx d\xi + \int_\Omega \left(n^\e \frac{|v^\e|^2}{2}  + \frac{1}{\gamma-1}(n^\e)^\gamma \right)\,dx\\
& \quad+ \frac{1}{\e}\int_0^t \md_1 (f^\e)(s)\,ds + \int_0^t \md_2(f^\e, v^\e)(s)\,ds\\
& \qquad \le  \mathcal{F}(f_0^\e, n_0^\e, v_0^\e)  +C(T).\\
\end{aligned}\end{align*}
\end{lemma}
\begin{proof}
We first notice from Theorem \ref{T2.1} that the following inequality holds:
\begin{align}
\begin{aligned}\label{D-2}
\mathcal{F}&(f^\e, n^\e, v^\e)(t) + \frac{1}{\e} \int_0^t \md_1(f^\e)(s)\,ds + \int_0^t \md_2(f^\e, v^\e)(s) \,ds \\
&\le \mathcal{F}(f_0^\e, n_0^\e, v_0^\e) + 3t \|f_0^\e\|_{L^1}.
\end{aligned}
\end{align}
On the other hand, by Lemma \ref{LA.4} we get
\begin{align*}\begin{aligned}
& \int_{\Omega \times \bbr^3} f^\e \log^- f^\e \,dxd\xi \le \frac{1}{4} \int_{\Omega \times \bbr^3} |\xi|^2 f^\e \,dxd\xi + C.
\end{aligned}\end{align*}
We then put the above inequality into \eqref{D-2} to conclude the desired result. 
\end{proof}

Next, we discuss the modified entropy inequality for later use.
\begin{lemma}\label{L4.2}
Let $T>0$, and suppose that $(f^\e, n^\e, v^\e)$ is a weak solution to the system \eqref{A-2} on $[0,T)$ with the initial data $(f_0^\e, n_0^\e, v_0^\e)$ in the sense of Definition \ref{D2.1}. Then, we have
\begin{align*}\begin{aligned}
&\mathcal{F}(f^\e, n^\e, v^\e)(t) + \frac{1}{2\e} \int_0^t \md_1(f^\e)(s)\,ds + \int_0^t \int_\Omega \rho^\e |u^\e - v^\e|^2 \,dxds + \int_0^t \int_\Omega |\nabla v^\e|^2 \,dxds\\
&\quad \le \mathcal{F}(f_0^\e, n_0^\e, v_0^\e) + C(T)\e.
\end{aligned}\end{align*}
\end{lemma}
\begin{proof}
A straightforward computation yields
\begin{align*}\begin{aligned}
&\frac{1}{2}\int_{\Omega^2 \times \bbr^6} f^\varepsilon(x,\xi)f^\varepsilon(y,\xi_*)|\xi-\xi_*|^2 \,dxd\xi dyd\xi_*  + \int_{\Omega} \rho^\varepsilon|u^\varepsilon -v^\varepsilon|^2 \,dx\\
&\quad =\frac{1}{2} \int_{\Omega^2} \rho^\varepsilon(x)\rho^\varepsilon (y)|u^\varepsilon(x) - u^\varepsilon(y)|^2 \,dydx  +  \int_{\Omega \times \bbr^3} |v^\varepsilon - \xi|^2 f^\varepsilon \,dxd\xi.
\end{aligned}\end{align*}
We then estimate the first term on the right-hand side of the above equality as in \cite{K-M-T3}:
\begin{align}\label{est_en1}
\begin{aligned}
\frac{1}{2}& \int_{\Omega^2} \rho^\varepsilon(x)\rho^\varepsilon (y)|u^\varepsilon(x) - u^\varepsilon(y)|^2 \,dydx\\
& = \int_{\Omega^2} \rho^\e (x) \rho^\e (y) (u^\e(x) - u^\e(y))\cdot u^\e(x) \,dxdy\\
&=\int_{\Omega^2 \times \bbr^6} f^\e(x,\xi)f^\e(y,\xi_*)(\xi-\xi_*)\cdot u^\e(x)\,dyd\xi_* dxd\xi\\
&=\int_{\Omega^2 \times \bbr^6} f^\e(y,\xi_*)(\xi-\xi_*)\cdot (f^\e(x,\xi)(u^\e(x)-\xi) - \nabla_\xi f(x,\xi))\,dyd\xi_*dxd\xi\\
& \quad + \int_{\Omega^2 \times \bbr^6}f^\e(y,\xi_*)f^\e(x,\xi)(\xi-\xi_*)\cdot  \xi \,dyd\xi_*dxd\xi\\
&\quad + \int_{\Omega^2 \times \bbr^6}f^\e(y,\xi_*)(\xi-\xi_*)\cdot \nabla_\xi f^\e(x,\xi)\,dyd\xi_*dxd\xi\\
& =: I_1 + I_2 + I_3,
\end{aligned}
\end{align}
where, by using the integration by parts, $I_3$ can be easily estimated as 
\[
I_3 = -3\|f^\e\|_{L^1(\Omega\times\bbr^3)}^2. 
\]
For $I_2$, we use $(x,\xi) \leftrightarrow (y,\xi_*)$ to get
\[
I_2 = \frac{1}{2}\int_{\Omega^2 \times \bbr^6}f^\e(y,\xi_*)f^\e(x,\xi)|\xi-\xi_*|^2 \,dyd\xi_*dxd\xi. 
\]
For the estimate of $I_1$, we set
\[
\widetilde{\md}^\e(x,\xi) := \frac{1}{\sqrt{f^\e(x,\xi)}}(f^\e(x,\xi)(u^\e(x)-\xi)-\nabla_\xi f^\e(x,\xi)). 
\]
Then we estimate
\begin{align*}\begin{aligned}
I_1&=\int_{\Omega^2 \times \bbr^6}\sqrt{f^\e(x,\xi)}f^\e(y,\xi_*)(\xi-\xi_*)\widetilde{\md}^\e(x,\xi)\,dyd\xi_* dxd\xi\\
&= \int_{\Omega^2 \times \bbr^3} \sqrt{f^\e(x,\xi)}\rho^\e(y)(\xi-u^\e(y))\widetilde{\md}^\e(x,\xi) \,dydxd\xi\\
& = \int_{\Omega} \left( \int_\Omega \rho^\e(y)\,dy \right) \left(\int_{\bbr^3} \xi \sqrt{f^\e(x,\xi)}\widetilde{\md}^\e(x,\xi) \,d\xi \right)dx\\
& \quad -  \int_{\Omega} \left( \int_\Omega \rho^\e(y) u^\e (y)\,dy \right) \left(\int_{\bbr^3}\sqrt{f^\e(x,\xi)}\widetilde{\md}^\e(x,\xi) \,d\xi \right)dx\\
& \le 2\|f^\e\|_{L^1(\Omega \times \bbr^3)}\left(\int_{\Omega \times \bbr^3}|\xi|^2 f^\e(x,\xi)\,dxd\xi \right)^{1/2}\left(\md_1(f^\e)\right)^{1/2}\\
& \le \frac{1}{2\e} \md_1(f^\e) + \e\|f^\e\|_{L^1(\Omega \times \bbr^3)}^2\left( \int_{\Omega\times\bbr^3} |\xi|^2 f^\e (x,\xi)\,dxd\xi\right).
\end{aligned}\end{align*}
We combine the estimates for $I_i, i=1,2,3$, and put it into \eqref{est_en1} to obtain the desired result.
\end{proof}

\subsection{Relative entropy estimates} In this subsection, we provide estimates regarding the relative entropy. For this purpose, we set
\[ 
U = \left( \begin{array}{c} \rho \\ m \\ n \\ w \end{array}\right), \quad A(U) := \left(\begin{array}{cccc} m & 0 & 0 & 0 
\\ (m \otimes m)/\rho & \rho\mathbb{I}_{3 \times 3} & 0 & 0  \\
0& 0 & w & 0 \\
0 & 0 & (w \otimes w)/n & n^\gamma \mathbb{I}_{3 \times 3} \end{array}\right), 
\]
and
\[
F(U) = \left( \begin{array}{c} 0 \\ \rho  (v-u) \\ 0 \\ -\rho  (v-u)  + \Delta  v \end{array} \right),
\]
where $\mathbb{I}_{3 \times 3}$ denotes the $3 \times 3$ identity matrix, $m := \rho u$, and $w := nv$. Note that the system \eqref{A-1} can be recast in the form of conservation of laws:
\[
U_t + \nabla \cdot A(U) = F(U). 
\]
Then, the corresponding macroscopic entropy $E(U)$ to the above system is given by
\[
E(U) :=  \frac{m^2}{2\rho} + \frac{w^2}{2n} + \rho \log \rho+ \frac{1}{\gamma-1}n^\gamma. 
\]
Here, $E(U)$ is indeed an entropy thanks to the existence of a flux $Q$ satisfying
\[D_jQ_i(U) = \sum_{k=1}^4 D_j A_{ki}(U) D_k E(U). \]
Then, this $Q$ satisfies
\[\nabla \cdot Q(U) = DE(U) \nabla \cdot A(U), \]
and thus
\[\partial_t E(U) + \nabla \cdot Q(U) = \rho |u-v|^2 + v \cdot \nabla v. \]
We also need the relative entropy functional $\mathcal{H}$ defined as
\[ 
\mathcal{H}(\bar U|U) := E(\bar U)-E(U)-DE(U)(\bar U-U), \quad \mbox{where} \quad \bar U = \left( \begin{array}{c} \bar\rho \\ \bar m \\ \bar n \\ \bar w \end{array}\right).
\]
Then after some computations, we find
\[
\mathcal{H}(\bar U|U) = \frac{\bar\rho}{2}|u-\bar{u}|^2 + \frac{\bar{n}}{2} |v-\bar{v}|^2 + P(\bar\rho| \rho) + \tilde{P}(\bar{n}|n), 
\]
where $P(x|y)$ and $\tilde{P}(x|y)$ are relative pressures given by
\[
P(x|y) := x\log x - y\log y + (y-x)(1 +\log y) 
\]
and
\[
\tilde{P}(x|y) := \frac{1}{\gamma -1} (x^\gamma - y^\gamma) + \frac{\gamma}{\gamma-1}(y-x)y^{\gamma-1},
\]
respectively.
Note that the relative pressures can be estimated as follows:
\[
P(x|y)= \int_{y}^{x} \frac{x - z}{z}\,dz\ge \frac{1}{2} \min \left\{ \frac{1}{x}, \frac{1}{y}\right\} (x-y)^2
\]
and
\[
\tilde{P}(x|y) \ge \gamma \min\left\{ x^{\gamma -2}, y^{\gamma-2}\right\} (x-y)^2.
\]
Focusing on these quantities, we derive a relation for the relative entropy functional.

\begin{lemma}\label{L4.3}
The relative entropy $\mathcal{H}$ satisfies the following equation:
\begin{align*}\begin{aligned}
& \int_\Omega \mathcal{H}(\bar U|U) \,dx + \int_0^t \int_\Omega |\nabla(v-\bar v)|^2 \,dxds + \int_0^t \int_\Omega\bar\rho|( \bar u - \bar v) - (u-v)|^2 \,dxds\\
&\quad =  \int_\Omega \mathcal{H}(\bar U_0|U_0) \,dx + \int_0^t \int_\Omega \partial_s E(\bar U) \,dxds + \int_0^t \int_\Omega  |\nabla \bar v|^2 \,dx ds + \int_0^t \int_\Omega \bar\rho | \bar u-  \bar v|^2 \,dxds \\
&\qquad - \int_0^t \int_\Omega DE(U)(\partial_s \bar U + \nabla \cdot A(\bar U) -F(\bar U)) \,dxds - \int_0^t \int_\Omega (\nabla DE(U)) : A(\bar U|U) \,dxds\\
& \qquad + \int_0^t\int_\Omega \left(\frac{ \bar n}{n} \rho - \bar\rho \right)(v- \bar v) (u-v) \,dxds + 2\int_0^t \int_\Omega \left(\frac{\bar n}{n} - 1\right)  \Delta v  \cdot (v - \bar v) \,dxds,
\end{aligned}\end{align*}
where $A:B= \sum_{i=1}^m \sum_{j=1}^n a_{ij} b_{ij}$ for $A = (a_{ij}), B= (b_{ij}) \in \bbr^{mn}$ and $A(\bar U|U)$ is the relative flux functional defined as
\[
A(\bar U|U) := A(\bar U) - A(U) - DA(U)(\bar U-U). 
\]
\end{lemma}
\begin{proof}
A direct calculation yields
$$\begin{aligned}
\frac{d}{dt} \int_{\Omega} \mathcal{H}(\bar U|U) \,dx & = \int_{\Omega} \partial_t E(\bar U) \,dx - \int_{\Omega} DE(U)(\pa_t \bar U + \nabla \cdot A(\bar U) -F(\bar U))\, dx\\
& \quad + \int_{\Omega} D^2 E(U) \nabla \cdot A(U)(\bar U-U) + DE(U)\nabla \cdot A(\bar U) \,dx\\
& \quad - \int_{\Omega} D^2 E(U) F(U) (\bar U-U) + DE(U)F(\bar U) \,dx\\
& =: \sum_{i=1}^4 J_i.
\end{aligned}$$
It suffices to estimate $J_3$ and $J_4$. For $J_3$, we recall from \cite{K-M-T3} that
\begin{align*}\begin{aligned}
D^2E(U)& \nabla \cdot A(U)(\bar U-U)\\
&= \nabla \cdot \Big(DQ(U)(\bar U-U)\Big) - \nabla \cdot \Big(DA(U)(\bar U-U)\Big)DE(U)\\
& = \nabla \cdot \Big(DQ(U)(\bar U-U)\Big) - \nabla \cdot \Big(A(\bar U) - A(U) - A(\bar U|U)\Big)DE(U).
\end{aligned}\end{align*}
Note that 
\begin{align*}\begin{aligned}
D&Q(U)(\bar U-U)\\
&=\sum_{k=1}^4 D_k Q(U)(\bar U_k - U_k)\\
&=\left( \begin{array}{c} -(u \otimes u) \cdot u \\ u \\ 0 \\ 0 \end{array}\right) ({\bar \rho} - \rho) + \left( \begin{array}{c} \left(-|u|^2/2 + \log \rho +1\right)\mathbb{I}_{3 \times 3} + 2u \otimes u \\ 0 \\ 0 \\ 0 \end{array}\right)(\bar{\rho} \bar{u} - \rho u)\\
& \quad + \left( \begin{array}{c} 0 \\[2mm] 0 \\[2mm] -(v \otimes v) \cdot v \\[2mm] \gamma n^{\gamma-1}v \end{array}\right) ({\bar n} - n) +  \left( \begin{array}{c} 0\\[2mm] 0 \\[2mm] \left(-|v|^2/2 + (\gamma/(\gamma-1))n^{\gamma-1}\right)\mathbb{I}_{3 \times 3}+ 2 v\otimes v  \\[2mm] 0  \end{array}\right)(\bar{n} \bar{v} - nv)\\
& = \left( \begin{array}{c} -({\bar \rho} - \rho)(u \otimes u) \cdot u + \left[\left(-|u|^2/2 + \log \rho +1\right)\mathbb{I}_{3 \times 3} + 2u \otimes u\right](\bar{\rho} \bar{u} - \rho u)\\[2mm] (\bar{\rho} - \rho)u \\[2mm]
 \left[ \left(-|v|^2/2 + (\gamma/(\gamma-1))n^{\gamma-1}\right)\mathbb{I}_{3 \times 3} + 2 v\otimes v\right](\bar{n} \bar{v} - nv) \\[2mm]
  \gamma n^{\gamma-1} v (\bar{n}- n) \end{array}\right).
\end{aligned}\end{align*}
Recall that $(\bar{v}, v)$ satisfy the homogeneous Dirichlet boundary condition and $(\bar{u}, u)$ satisfy the kinematic boundary condition, i.e. $u \cdot r \equiv 0 \equiv {\bar u} \cdot r$ on $\partial\Omega$. With the identity $(({\bf a} \otimes {\bf a})\cdot {\bf b})\cdot {\bf c} = ({\bf a}\cdot {\bf b})({\bf a}\cdot {\bf c})$ in mind, where ${\bf a},{\bf b},{\bf c} \in\R^3$, it can be checked that
\[
\int_\Omega \nabla \cdot \lt( DQ(U)(\bar U-U) \rt) dx = \int_{\partial\Omega} DQ(U)(\bar U-U)\cdot r(x) \,d\sigma(x) = 0.
\]
Moreover, we also find
\begin{align*}\begin{aligned}
D&E(U)\nabla\cdot A(U)\\
&= \frac{|u|^2}{2}\nabla\cdot(\rho u) + u \cdot (\rho u \cdot \nabla u) + \nabla \cdot ((\rho \log \rho) u) \\
&\quad + \frac{|v|^2}{2} \nabla\cdot(nv) + v \cdot (nv \cdot \nabla v) + \nabla \cdot \left( \frac{\gamma}{\gamma-1}n^\gamma v\right)\\
&= \frac{|u|^2}{2}\nabla\cdot(\rho u) + \frac{\rho u}{2} \cdot \nabla |u|^2 + \nabla \cdot ((\rho \log \rho) u) \\
&\quad + \frac{|v|^2}{2} \nabla\cdot(nv) + \frac{nv}{2} \cdot \nabla |v|^2 + \nabla \cdot \left( \frac{\gamma}{\gamma-1}n^\gamma v\right),
\end{aligned}\end{align*}
which implies
\begin{align*}\begin{aligned}
&\int_\Omega DE(U)\nabla \cdot A(U) \,dx \\
&\quad = \int_{\partial\Omega} \left(\frac{|u|^2}{2} \rho u  + (\rho \log \rho) u + \frac{|v|^2}{2}nv + \frac{\gamma}{\gamma-1}n^\gamma v\right)\cdot r(x) \,d\sigma(x) \cr
&\quad = 0.
\end{aligned}\end{align*}
Thus we have
\[
J_3 = \int_\Omega DE(U)  \nabla \cdot \lt( A(\bar U|U) \rt)dx.
\]
Here we notice that
\[
A(\bar U|U) = \left(\begin{array}{cccc}0 & 0 &0 &0 \\ \bar{\rho}( \bar{u} - u) \otimes (\bar{u}- u) & 0 & 0 & 0 \\ 0 & 0 &0 &0 \\
0 & 0 &\bar{n}(\bar{v} - v)\otimes (\bar{v} - v) & (\gamma-1)\tilde{P}(\bar{n} | n) \mathbb{I}_{3 \times 3} \end{array}\right),
\]
and this gives
\begin{align*}\begin{aligned}
J_3 &= \int_{\partial\Omega} \left(DE(U)^TA(\bar U|U)\right)\cdot r(x) \,d\sigma(x)-\int_\Omega \lt(\nabla DE(U)\rt) : A(\bar U|U)\,dx\\
 &=\int_{\partial\Omega} \left[\bar{\rho}( \bar{u} - u) \otimes (\bar{u}- u) \cdot u + \left(\bar{n}(\bar{v} - v)\otimes (\bar{v} - v)  + (\gamma-1)\tilde{P}(\bar{n} | n) \mathbb{I}_{3 \times 3}\right)\cdot v \right]\cdot r(x)\,d\sigma (x)\\
 &\quad -\int_\Omega \lt(\nabla DE(U)\rt) : A(\bar U|U)\,dx\\
 &=-\int_\Omega \lt(\nabla DE(U)\rt) : A(\bar U|U)\,dx.
 \end{aligned}\end{align*}
For the estimate of $J_4$, we can follow the proof of \cite[Lemma 2.3]{C-J} to get
\begin{align*}\begin{aligned}
J_4 &=\int_{\Omega}  |\nabla \bar v|^2 \,dx +  \int_{\Omega}\bar\rho | \bar u-  \bar v|^2 \,dx\\
& \quad +  \int_{\Omega} \left(\frac{ \bar n}{n} \rho - \bar\rho \right)(v- \bar v) (u-v) \,dx + 2\int_{\Omega}\left(\frac{\bar n}{n} - 1\right)  \Delta v  \cdot (v - \bar v) \,dx,
\end{aligned}\end{align*}
and this concludes the desired relation.
\end{proof}

\subsection{Hydrodynamic limit} In this subsection, we provide the details of proof for Theorem \ref{T2.2} and Corollary \ref{cor_main}, which show that weak solutions to the system \eqref{A-2} can be well approximated by the two-phase fluid system \eqref{A-3} when $\e>0$ is sufficiently small. Let 
\[
U := \left(\begin{array}{c} \rho \\ \rho u \\ n \\ nv \end{array}\right) \quad \mbox{and} \quad U^\varepsilon := \left(\begin{array}{c} \rho^\varepsilon \\ \rho^\varepsilon u^\varepsilon \\ n^\varepsilon \\ n^\varepsilon v^\varepsilon \end{array}\right) \quad \mbox{with} \quad \rho^\e = \int_{\R^3} f^\e\,d\xi \quad \mbox{and} \quad \rho^\e u^\e = \int_{\R^3} \xi f^\e\,d\xi,
\]
where $(f^\varepsilon, n^\varepsilon, v^\varepsilon)$ and $(\rho, u, n,v)$ are weak solutions to the system \eqref{A-2} and a unique strong solution to the system \eqref{A-3}, respectively.

Before proceeding the proof, we recall the following extension theorem whose proof can be found in \cite[Theorem 5.4.1]{Ev}. This enables us to use some Sobolev inequality in the bounded domain with the kinematic boundary condition in \eqref{A-3_bdy}.

\begin{theorem}\label{thm_exten} Suppose $\Omega \subset \R^3$ is a bounded domain with $\mathcal{C}^2$ boundary and choose a bounded open set $U$ where $\Omega$ is compactly embedded. Then for any $p \ge 1$, there exists a bounded linear operator $E : W^{2,p}(\Omega) \to W^{2,p}(\bbr^n)$ such that for any $u \in W^{2,p}(\Omega)$
\begin{itemize}
\item[(i)] $Eu = u$ a.e. in $\Omega$,
\item[(ii)] $Eu$ is supported within $U$, and
\item[(iii)] $\|Eu\|_{W^{2,p}(\R^n)} \le C \|u\|_{W^{2,p}(\Omega)}$,
\end{itemize}
where $C$ depends only on $p$, $\Omega$ and $U$. 
\end{theorem}
Moreover, we provide a technical lemma concerning the lower bound for the relative pressure.
\begin{lemma}\label{ineq}
Let $x, y>0$ and $\gamma>1$. If $0 < y_{min} \le y \le y_{max}<\infty$, then the following inequality holds:
\begin{align}\label{T1-1.1}
\begin{aligned}
\tilde{P}(x|y) &= \frac{1}{\gamma - 1}(x^\gamma - y^\gamma) + \frac{\gamma}{\gamma-1}(y-x)y^{\gamma-1} \cr
&\ge C  \left\{ \begin{array}{ll}
 (x-y)^2 & \textrm{if $y/2 \le x \le 2y$}\\[2mm]
\displaystyle (1+x^\gamma) & \textrm{otherwise}
  \end{array} \right.,
\end{aligned}
\end{align}
where $C = C(\gamma, y_{min}, y_{max})$ is a positive constant.
\end{lemma}
\begin{proof}By Taylor's theorem, we easily find 
\[
\frac{1}{\gamma - 1}(x^\gamma - y^\gamma) + \frac{\gamma}{\gamma-1}(y-x)y^{\gamma-1} \ge \gamma \min\{x^{\gamma-2},y^{\gamma-2}\} (x-y)^2.
\]
For notational simplicity, we denote by $\tilde{P}_\ell(x|y)$ the right hand side of the above inequality, and we split two cases depending on the choice of $\gamma$.\\

\noindent $\diamond$ (Case A: $1 < \gamma \leq 2$) If $y/2 \le x \le 2y$, we easily get
\[
\tilde{P}_\ell(x|y) \ge \gamma (2y)^{\gamma-2} (x-y)^2 \ge \gamma (2y_{max})^{\gamma-2} (x-y)^2.
\]
If $x >2y > y \ (> y_{min})$, i.e., $y/x < 1/2$, we obtain
\begin{align*}\begin{aligned}
\tilde{P}_\ell(x|y) &= \gamma x^{\gamma-2} |x-y|^2 = \gamma x^\gamma \lt| 1-\frac{y}{x}\rt|^2 \cr
&\ge \frac{\gamma x^\gamma}{4} = \frac{\gamma}{4}(1+x^\gamma) \lt(1 - \frac{1}{1+x^\gamma} \rt) \cr
&\ge \frac{\gamma}{4}(1+x^\gamma)\lt(1 - \frac{1}{1+y_{min}^\gamma} \rt).
\end{aligned}\end{align*}
On the other hand, if $x< y/2$, i.e., $x/y < 1/2$, we estimate
\begin{align*}\begin{aligned}
\tilde{P}_\ell(x|y) &\ge \gamma y^{\gamma-2} |x-y|^2 = \gamma y^\gamma \lt| 1-\frac{x}{y}\rt|^2\\
& \ge \frac{\gamma y^\gamma}{4} = \frac{\gamma}{4}(1+y^\gamma) \lt(1 - \frac{1}{1+y^\gamma} \rt)\\
& \ge \frac{\gamma}{4}(1+x^\gamma)\lt(1 - \frac{1}{1+y_{min}^\gamma} \rt).
\end{aligned}\end{align*}
This asserts the inequality \eqref{T1-1.1} for the case $\gamma \in (1,2]$. \\

\noindent $\diamond$ (Case B: $\gamma >2$) If $y/2 \le x \le 2y$, we similarly get
\[
\tilde{P}_\ell(x|y)\ge \gamma (y/2)^{\gamma-2} |x-y|^2 \ge \gamma (y_{min}/2)^{\gamma-2} |x-y|^2.
\]
For the rest of cases, we directly estimate the relative pressure $\tilde P(x|y)$ instead of $\tilde P_\ell(x|y)$. If $x>2y$, i.e. $y<x/2$, we obtain
\begin{align*}\begin{aligned}
\tilde P(x|y) &=\frac{1}{\gamma-1} x^\gamma + y^\gamma - \frac{\gamma}{\gamma-1}xy^{\gamma-1}\\
&\ge \frac{\left(1-\gamma2^{1-\gamma}\right)}{\gamma-1} x^\gamma +y^\gamma\\
&\ge \min\left\{ \frac{1-\gamma2^{1-\gamma}}{\gamma-1}, y_{min}^\gamma \right\}( 1+ x^\gamma),
\end{aligned}\end{align*}
where we used $\gamma>2$ to get $ 1 -\gamma 2^{1-\gamma} >0$. When $x<y/2$, we estimate
\[
\tilde{P}_\ell(x|y)\ge \frac{1}{\gamma-1} x^\gamma + \left(1 - \frac{\gamma }{2(\gamma-1)}  \right)y^\gamma\ge \min\left\{ \frac{1}{\gamma-1}, \left(1 - \frac{\gamma }{2(\gamma-1)}  \right)y_{min}^\gamma \right\}( 1+ x^\gamma),
\]
where we used $\gamma>2$ to get $1 -\gamma/(2(\gamma-1)) >0$. This completes the proof.
\end{proof}

\subsubsection{Proof of Theorem \ref{T2.2}} Replacing $\bar U$ with $U^\e$ in Lemma \ref{L4.3}, we find
\begin{align*}\begin{aligned}
& \int_\Omega \mathcal{H}(U^\varepsilon|U)\,dx + \int_0^t \int_\Omega |\nabla(v-v^\varepsilon)|^2\,dxds + \int_0^t \int_\Omega \rho^\varepsilon|( u^\varepsilon - v^\varepsilon) - (u-v)|^2\,dxds\\
&\quad  =  \int_\Omega \mathcal{H}(U^\varepsilon_0|U_0)\,dx \cr
&\qquad + \int_0^t \int_\Omega \partial_s E(U^\varepsilon) \,dxds + \int_0^t \int_\Omega |\nabla v^\varepsilon|^2 \,dxds + \int_0^t \int_\Omega \rho^\varepsilon | u^\varepsilon -  v^\varepsilon|^2 \,dxds \\
&\qquad - \int_0^t \int_\Omega DE(U)(\partial_s U^\varepsilon + \nabla \cdot A(U^\varepsilon) -F(U^\varepsilon)) \,dxds\\
& \qquad - \int_0^t \int_\Omega (\nabla DE(U)) : A(U^\varepsilon|U) \,dxds\\
& \qquad + \int_0^t\int_\Omega \left(\frac{ n^\varepsilon}{n} \rho - \rho^\varepsilon \right)(v- v^\varepsilon) \cdot(u-v) \,dxds\\
& \qquad + 2\int_0^t \int_\Omega \left(\frac{n^\varepsilon - n}{n}\right) \Delta v \cdot (v - v^\varepsilon) \,dxds\\
& \quad =: \sum_{i=1}^6 K_i.
\end{aligned}\end{align*}
We separately estimate $K_i, i=1,\dots,6$ as follows. \newline

\noindent $\diamond$ (Estimates for $K_1$): It follows from our standing assumption {\bf (H2)} that
\[
K_1 \leq C\sqrt\e.
\]

\noindent $\diamond$ (Estimates for $K_2$): Similar to \cite[Proposition 5.2]{C-C-K}, we estimate 
$$\begin{aligned}
K_2 &=  \int_{\bbr^d} E(U^\varepsilon) \,dx - \mathcal{F}(f^\varepsilon, n^\varepsilon, v^\varepsilon) \\ 
& \quad + \mathcal{F}(f^\varepsilon, n^\varepsilon, v^\varepsilon) + \int_0^t  \int_\Omega |\nabla v^\varepsilon|^2 \,dxds + \int_0^t \int_\Omega \rho^\varepsilon |u^\varepsilon - v^\varepsilon|^2 \,dxds - \mathcal{F}(f^\e_0, n^\e_0, v^\e_0) \\
&\quad + \mathcal{F}(f^\e_0, n^\e_0, v^\e_0) - \int_\Omega E(U_0)\,dx \cr 
& \le C(T)\e + \mathcal{F}(f^\e_0, n^\e_0, v^\e_0)  - \int_\Omega E(U_0)\,dx,
\end{aligned}$$
where we used the entropy inequality in Remark 4.1 and the fact that
\[
\int_\Omega E(U^\varepsilon) \,dx \le \mathcal{F}(f^\varepsilon, n^\varepsilon, v^\varepsilon).
\]
Then, the standing assumption {\bf (H1)} on the initial data implies
\[
K_2 \leq C\sqrt\e
\]
for some $C>0$ independent of $\e$.\newline

\noindent $\diamond$ (Estimates for $K_3$): From \eqref{A-2}, we obtain
$$\begin{aligned}
&\partial_t \rho^\e + \nabla \cdot (\rho^\e u^\e) = 0,\\
&\partial_t (\rho^\e u^\e) + \nabla \cdot (\rho^\e u^\e \otimes u^\e) + \nabla \rho^\e - \rho^\e(v^\e-u^\e) = \nabla \cdot \lt(\int_{\R^3} (u^\e \otimes u^\e - \xi \otimes \xi + \mathbb{I}_{3\times 3}
  )f^\e\,d\xi \rt) ,\\
&\partial_t n^\e + \nabla \cdot (n^\e v^\e) = 0,\\
&\partial_t (n^\e v^\e) + \nabla \cdot (n^\e v^\e \otimes v^\e) + \nabla p(n^\e) - \Delta v^\e  + \rho^\e(v^\e-u^\e) = 0
\end{aligned}$$
in the sense of distributions. This implies
$$\begin{aligned}
&- \int_0^t \int_\Omega DE(U)(\partial_s U^\varepsilon + \nabla \cdot A(U^\varepsilon) -F(U^\varepsilon)) \,dxds\cr
&\quad = - \int_0^t \int_\Omega D_m E(U)\lt(\nabla \cdot \lt(\int_{\R^3} (u^\e \otimes u^\e - \xi \otimes \xi + \mathbb{I}_{3\times 3}
  )f^\e\,d\xi \rt) \rt) \,dxds\cr
&\quad = -\int_0^t \int_{\partial\Omega\times\R^3}\left(\left(u^\e \otimes u^\e - \xi\otimes\xi +\mathbb{I}_{3\times 3}
  \right) u\right)\cdot r(x) \gamma f^\e  \,d\sigma(x)d\xi ds  \\
& \qquad +\int_0^t \int_\Omega \nabla u : \lt(\int_{\R^3} (u^\e \otimes u^\e - \xi \otimes \xi + \mathbb{I}_{3\times 3}
  )f^\e\,d\xi \rt)  \,dxds
\end{aligned}$$
due to $D_m E(U) = u$. Here, we notice that
\begin{align*}\begin{aligned}
&\int_{\partial\Omega\times\R^3}\left(\left(u^\e \otimes u^\e - \xi\otimes\xi + \mathbb{I}_{3\times 3}
  \right)u\right)\cdot r(x) \gamma f  \,d\sigma(x)d\xi \\
&\quad = \int_{\partial\Omega\times\R^3}(\xi \cdot r)(\xi \cdot u)\gamma f^\e \,d\sigma(x)d\xi\\
&\quad = \int_{\Sigma_+}|\xi \cdot r|(\xi \cdot u)\gamma_+ f^\e \,d\sigma(x)d\xi - \int_{\Sigma_-}|\xi \cdot r|(\xi \cdot u)\gamma_- f^\e \,d\sigma(x)d\xi\\
&\quad = \int_{\Sigma_+}|\xi \cdot r|(\xi \cdot u)\gamma_+ f^\e \,d\sigma(x)d\xi- \int_{\Sigma_-}|\xi \cdot r|(\xi \cdot u)\gamma_+ f^\e(\mathcal{R}_x(\xi)) \,d\sigma(x)d\xi\\
&\quad = 0,
\end{aligned}\end{align*}
where $\mathcal{R}_x(\xi) = \xi - 2(\xi\cdot r(x))r(x)$, and we used the change of variable $\mathcal{R}_x(\xi) \leftrightarrow \xi$ on the second term. We then follow the proof of \cite[Lemma 4.4]{K-M-T3} to get
\[
K_3 \le C\sqrt{\varepsilon},
\]
where $C =  C(\|\nabla u\|_{L^\infty})$ is a positive constant independent of $\e$. Since we assumed that $\partial\Omega$ is smooth, we can put $p=2$, $n=3$ in Theorem \ref{thm_exten} to obtain
\[
\|\nabla u\|_{L^\infty(\om)} = \|E(\nabla u)\|_{L^\infty(\R^3)} \le \|E(\nabla u)\|_{H^2(\R^3)} \le C\|\nabla u\|_{H^2(\om)} \le C\|u\|_{H^3(\om)}. 
\]
$\diamond$ (Estimates for $K_4$): Recall that
\[
A(U^\e|U) = \left(\begin{array}{cccc}0 & 0 &0 &0 \\ \rho^\e( u^\e - u) \otimes (u^\e- u) & 0 & 0 & 0 \\ 0 & 0 &0 &0 \\
0 & 0 &n^\e(v^\e - v)\otimes (v^\e - v) & (\gamma-1)\tilde{P}(n^\e | n) \mathbb{I}_{3\times 3}   \end{array}\right).
\]
This implies
$$\begin{aligned}
\int_\Omega |A(U^\varepsilon|U)| \,dx &\le  \int_\Omega \rho^\varepsilon|u^\varepsilon - u|^2 + n^\varepsilon |v^\varepsilon-v|^2 + 3(\gamma-1)\tilde{P}(n^\varepsilon | n) \,dx\cr
& \le C \int_\Omega \mathcal{H}(U^\varepsilon|U) \,dx,
\end{aligned}$$
where $C>0$ only depends on $\gamma$. Thus we obtain
\[
K_4 \le C \int_0^t \int_\Omega \mathcal{H}(U^\varepsilon|U) \,dx ds.
\]
$\diamond$ (Estimates for $K_5$): We divide $K_5$ into two terms:
\begin{align*}\begin{aligned}
K_5 &= \int_0^t \int_\Omega (\rho-\rho^\varepsilon)(v- v^\varepsilon)\cdot(u-v) \,dxds + \int_0^t \int_\Omega \rho \left( \frac{ n^\varepsilon - n}{n}\right)(v- v^\varepsilon)\cdot (u-v) \,dx ds\\
& =: K_5^1 +K_5^2.
\end{aligned}\end{align*}
For the estimate of $K_5^1$, we use the following elementary inequality:
\begin{equation}\label{minmax}
1 = \min\left\{x^{-1}, y^{-1} \right\}\max\left\{x, y \right\} \le \min\left\{x^{-1}, y^{-1} \right\} (x+y) \quad \mbox{for} \quad x,y>0
\end{equation}
to get
\begin{align*}\begin{aligned}
& \lt|\int_\Omega (\rho-\rho^\varepsilon)(v- v^\varepsilon)\cdot(u-v) \,dx \rt| \\
& \quad \le \left( \int_\Omega \min\left\{\frac{1}{\rho^\varepsilon}, \frac{1}{\rho} \right\} (\rho- \rho^\varepsilon)^2\,dx \right)^{1/2} \left(\int_\Omega (\rho + \rho^\varepsilon)|v- v^\varepsilon|^2 |u-v|^2 \,dx\right)^{1/2}\\
& \quad \le C \left( \int_\Omega \mathcal{H}(U^\varepsilon|U) \,dx \right)^{1/2}\left(\int_\Omega (\rho + \rho^\varepsilon)|v- v^\varepsilon|^2 |u-v|^2 \,dx\right)^{1/2}.
\end{aligned}\end{align*}
On the other hand, the second term on the right hand side of the above inequality can be estimated as
\begin{align*}\begin{aligned}
&\int_\Omega (\rho + \rho^\varepsilon)|v-v^\varepsilon|^2 |u-v|^2 \,dx\\
& \le \|\rho\|_{L^\infty}\|v-v^\varepsilon\|_{L^6}^2 \|u-v\|_{L^3}^2 \cr
&\quad + 2\int_\Omega \Big(\rho^\varepsilon|(u-u^\varepsilon) - (v-v^\varepsilon)|^2 + \rho^\varepsilon|u-u^\varepsilon|^2 \Big)|u-v|^2 \,dx\\
&  \le C \|\nabla(v - v^\varepsilon)\|_{L^2}^2 \|u-v\|_{L^3}^2 \cr
&\quad + 2\|u-v\|_{L^\infty}^2 \left(\int_\Omega \rho^\varepsilon|(u-u^\varepsilon) - (v-v^\varepsilon)|^2 \,dx + \int_{\bbr^d} \rho^\varepsilon|u-u^\varepsilon|^2 \,dx\right),
\end{aligned}\end{align*}
where we used Sobolev inequality and Poincar\'e inequality. Here, we once again use Theorem \ref{thm_exten} to have
\begin{align*}\begin{aligned}
\|\rho\|_{L^\infty(\om)} &= \|E\rho\|_{L^\infty(\R^3)} \le \|E\rho\|_{H^2(\R^3)} \le C \|\rho\|_{H^2(\Omega)} \quad \mbox{and}\\
\|u-v\|_{L^\infty(\om)} &=\|E(u-v)\|_{L^\infty(\R^3)} \le \|E(u-v)\|_{H^2(\R^3)} \le C\|u-v\|_{H^2(\Omega)}.
\end{aligned}\end{align*}
Since $s > 7/2$, by means of Young's inequality, we estimate
$$\begin{aligned}
K_5^1 &\le C\int_0^t \int_\Omega \mathcal{H}(U^\varepsilon|U)\,dxds + \frac18\int_0^t \int_\Omega |\nabla (v - v^\e)|^2\,dxds  \cr
&\quad + \frac{1}{2}\int_0^t \int_\Omega \rho^\varepsilon|(u-u^\varepsilon) - (v-v^\varepsilon)|^2 \,dxds,
\end{aligned}$$
where $C = C(\|\rho\|_{L^\infty}, \|u-v\|_{L^\infty(0,T;L^\infty)})$ is a positive constant. For the term $K_5^2$, we let $n_* := \inf_{x \in \om} n(x) > 0$ and first consider the case $\gamma \in (3/2, 2]$. We use the inequality \eqref{minmax} to get
\begin{align*}\begin{aligned}
&\lt|\int_\Omega \rho \left( \frac{ n^\varepsilon - n}{n}\right)(v- v^\varepsilon)\cdot(u-v) \,dx\rt|\\
&\quad\le \frac{\|\rho\|_{L^\infty}}{n_*}\int_\Omega |n^\varepsilon-n| |v^\varepsilon-v| |u-v| \,dx \\
&\quad\le C \left( \int_\Omega \min\left\{(n^\varepsilon)^{\gamma-2}, n^{\gamma-2} \right\} (n- n^\varepsilon)^2\,dx \right)^{1/2}\cr
&\qquad \qquad \times  \left(\int_\Omega \lt(n^{2-\gamma} + (n^\varepsilon)^{2-\gamma}\rt)|v- v^\varepsilon|^2 |u-v|^2\,dx\right)^{1/2}\\
&\quad \le C \left( \int_\Omega \mathcal{H}(U^\varepsilon|U) \,dx \right)^{1/2} \left(\int_\Omega \lt(n^{2-\gamma} + (n^\varepsilon)^{2-\gamma}\rt)|v- v^\varepsilon|^2  |u-v|^2 \,dx\right)^{1/2},
\end{aligned}\end{align*}
where $C = C(\|\rho\|_{L^\infty}, n_*, \gamma)$ is a positive constant. We further estimate
\begin{align*}\begin{aligned}
\int_\Omega& \lt(n^{2-\gamma} + (n^\varepsilon)^{2-\gamma}\rt)|v- v^\varepsilon|^2 |u-v|^2 \,dx \\
&\le \|n\|_{L^\infty}^{2-\gamma} \|v-v^\varepsilon\|_{L^6}^2 \|u-v\|_{L^3}^2 + \int_\Omega (n^{\varepsilon})^{2-\gamma} |v-v^\varepsilon|^2 |u-v|^2 \,dx\\
&\le C\|n\|_{L^\infty}^{2-\gamma}\|u-v\|_{L^3}^2\|\nabla(v-v^\varepsilon)\|_{L^2}^2 + \int_\Omega (n^{\varepsilon})^{2-\gamma} |v-v^\varepsilon|^2 |u-v|^2 \,dx.
\end{aligned}\end{align*}
For $\gamma=2$, we easily get 
\begin{displaymath}
\int_\Omega (n^\varepsilon)^{2-\gamma}|v- v^\varepsilon|^2 |u-v|^2 \,dx \leq C\|u-v\|_{L^3}^2\|\nabla(v-v^\varepsilon)\|_{L^2}^2.
\end{displaymath}
For $\gamma \in (3/2,2)$, we first use Young's inequality to obtain
\begin{align*}\begin{aligned}
\int_\Omega& (n^\varepsilon)^{2-\gamma}|v- v^\varepsilon|^2 |u-v|^2 \,dx \\
&\le \int_{\bbr^d} (n^{\varepsilon})^{2-\gamma} |v-v^\varepsilon|^{(4-2\gamma) + (2\gamma -2)} |u-v|^2 \,dx\\
& \le(2-\gamma)\int_\Omega n^\varepsilon |v-v^\varepsilon|^2 \,dx + (\gamma-1) \int_\Omega|v-v^\varepsilon|^2 |u-v|^{2/(\gamma-1)} \,dx \\
& \le(2-\gamma)\int_\Omega n^\varepsilon|v-v^\varepsilon|^2 \,dx + (\gamma-1)\|v-v^\varepsilon\|_{L^6}^2 \|u-v\|_{L^{\frac{3}{\gamma-1}}}^{\frac{2}{\gamma-1}}\\
& \le C \left(\int_\Omega \mathcal{H}(U^\varepsilon|U) \,dx + \int_\Omega |\nabla (v - v^\e)|^2\,dx\right),
\end{aligned}\end{align*}
where $C = C(\gamma, \|u-v\|_{L^\infty(0,T;L^\infty)})$ is a positive constant. Using the similar argument as in the estimate of $K_5^1$, we find   
\begin{equation}\label{K-5-2} 
K_5^2 \le C \int_0^t \int_\Omega \mathcal{H}(U^\varepsilon|U)\,dxds + \frac18\int_0^t \int_\Omega |\nabla (v - v^\e)|^2\,dxds,
\end{equation}
for any $\gamma \in (3/2,2]$. 

For $\gamma >2$, we use Lemma \ref{ineq} with $x = n^\e$, $y=n$, $y_{min} = n_*$, and $y_{max} = \|n\|_{L^\infty}$. Then we estimate
\begin{align*}\begin{aligned}
&\lt|\int_\Omega \rho \left( \frac{ n^\varepsilon - n}{n}\right)(v- v^\varepsilon)\cdot(u-v) \,dx\rt|\cr
&\quad \le \frac{\|\rho(u-v)\|_{L^\infty}}{n_*}\int_\Omega |n^\varepsilon-n| |v^\varepsilon-v| \,dx \\
&\quad \le C \left(\int_{\Omega \cap \{n/2 \le n^\e \le 2n\}} + \int_{\Omega\cap\{n/2 \le n^\e  \le 2n\}^c}\right)|n^\varepsilon-n| |v^\varepsilon-v|\,dx\\
&\quad =: K_5^{21} + K_5^{22},
\end{aligned}\end{align*}
where $C = C(n_*, \|\rho\|_{L^\infty}, \|u-v\|_{L^\infty})$ is a positive constant. For $K_5^{21}$, by Cauchy-Schwarz inequality and the inequality \eqref{T1-1.1}, we get 
\begin{align*}
\begin{aligned}
K_5^{21} &\le \left(\int_{\Omega\cap\{n/2 \le n^\e \le 2n\}} \frac{|n^\e - n|^2}{n^\e} \,dx \right)^{1/2}\left(\int_{\Omega\cap\{n/2 \le n^\e \le 2n\}} n^\e |v^\e - v|^2 \,dx\right)^{1/2}\\
&\le \left(\int_{\Omega\cap\{n/2 \le n^\e \le 2n\}} \frac{|n^\e - n|^2}{n/2} \,dx \right)^{1/2} \left(\int_\Omega \mathcal{H}(U^\e|U)\,dx\right)^{1/2}\\
&  \le C\left(\int_{\Omega\cap\{n/2 \le n^\e \le 2n\}} |n^\e - n|^2 \,dx \right)^{1/2} \left(\int_\Omega \mathcal{H}(U^\e|U)\,dx\right)^{1/2}\\
&\le C\int_\Omega \mathcal{H}(U^\e|U)\,dx,
\end{aligned}
\end{align*}
where $C = C(n_*, \|n\|_{L^\infty})$ is a positive constant. For $K_5^{22}$, note that  $\Omega\cap\{n/2 \le n^\e \le 2n\}^c = \Omega\cap(\{ n^\e > 2n\} \cup \{n^\e <n/2\})$. On the region $\Omega\cap\{ n^\e > 2n\}$, we use Cauchy-Schwarz inequality, Sobolev inequality, and Young's inequality to obtain
\begin{align*}
&\int_{\Omega\cap\{ n^\e  > 2n\}}|n^\varepsilon-n| |v^\varepsilon-v|\,dx  \\
&\quad \le \left( \int_\Omega \min\left\{(n^\varepsilon)^{\gamma-2}, n^{\gamma-2} \right\} (n- n^\varepsilon)^2\,dx \right)^{1/2} \left(\int_{\Omega\cap\{ n^\e  > 2n\}} \lt(n^{2-\gamma} + (n^\varepsilon)^{2-\gamma}\rt)|v- v^\varepsilon|^2 \,dx\right)^{1/2}\cr
&\quad \le \left(\int_\Omega \mathcal{H}(U^\e|U)\,dx\right)^{1/2} \left(\int_{\Omega\cap\{ n^\e  > 2n\}} \lt(n^{2-\gamma} + (2n)^{2-\gamma}\rt)|v- v^\varepsilon|^2 \,dx\right)^{1/2}\cr
&\quad \le C \left(\int_\Omega \mathcal{H}(U^\e|U)\,dx\right)^{1/2}\|v-v^\e\|_{L^6}\\
&\quad \le C\left(\int_\Omega \mathcal{H}(U^\e|U)\,dx\right)^{1/2}\|\nabla(v-v^\e)\|_{L^2}\\
&\quad \le C\int_\Omega \mathcal{H}(U^\e|U)\,dx + \frac{1}{16}\|\nabla(v-v^\e)\|_{L^2}^2,
\end{align*}
where $C = C(\gamma, \Omega, n_*, \|n\|_{L^\infty})$ is a positive constant. On the region $\Omega\cap\{ n^\e < n/2\}$, we note that $|n^\e - n| = n-n^\e > n/2$. Thus, we use Cauchy-Schwarz inequality, Sobolev inequality, and \eqref{T1-1.1} to yield
\begin{align*}
\int_{\Omega\cap\{ n^\e  <n/2\}}|n^\varepsilon-n| |v-v^\e|\,dx&\le C\left(\int_{\Omega\cap\{ n^\e  <n/2\}}|n^\varepsilon-n|^2\, dx\right)^{1/2}\|v - v^\e\|_{L^6}\\
&\le C\|\nabla(v - v^\e)\|_{L^2}\left(\int_{\Omega\cap\{ n^\e  <n/2\}}\frac{|n^\varepsilon-n|^\gamma}{|n^\varepsilon-n|^{\gamma-2}}\, dx\right)^{1/2}\\
&\le C\|\nabla(v - v^\e)\|_{L^2}\left(\int_{\Omega\cap\{ n^\e  <n/2\}}\frac{|n^\varepsilon-n|^\gamma}{(n/2)^{\gamma-2}} \,dx\right)^{1/2}\\
&\le C\|\nabla(v - v^\e)\|_{L^2}\left(\int_{\Omega\cap\{ n^\e  <n/2\}} \frac{\|n\|_{L^\infty}^\gamma}{(n_{*}/2)^{\gamma-2}} \left|\frac{n^\varepsilon}{n}+1\right|^\gamma dx\right)^{1/2}\\
&\le  C\|\nabla(v - v^\e)\|_{L^2}\left(\int_{\Omega\cap\{ n^\e  <n/2\}} (1+(n^\e)^\gamma)\, dx\right)^{1/2}\\
&\le  C\|\nabla(v - v^\e)\|_{L^2}\left(\int_{\Omega\cap\{ n^\e  <n/2\}} \mathcal{H}(U^\e|U) \,dx\right)^{1/2}\\
&\le C\int_\Omega \mathcal{H}(U^\e|U)dx + \frac{1}{16}\|\nabla(v - v^\e)\|_{L^2}^2,
\end{align*}
where $C = C(n_*, \|n\|_{L^\infty}, \gamma, \Omega)$ is a positive constant. Hence, we can also obtain the inequality \eqref{K-5-2} when $\gamma>2$.  Now, we collect the estimates for $K_5^1$ and $K_5^2$ to yield
$$\begin{aligned}
K_5 &\le C\int_0^t \int_{\Omega} \mathcal{H}(U^\varepsilon|U) \,dxds + \frac14\int_0^t \int_{\Omega}|\nabla (v - v^\e)|^2\,dxds \cr
&\quad + \frac12\int_0^t \int_{\Omega} \rho^\varepsilon|(u-u^\varepsilon) - (v-v^\varepsilon)|^2 \,dxds,\\
\end{aligned}$$
where $C = C(\gamma, n_*, \|\rho\|_{L^\infty}, \|u-v\|_{L^\infty(0,T;L^\infty)})$ is a positive constant. \newline

\noindent $\diamond$ (Estimates for $K_6$): Similarly to the estimate of $K_5$, when $\gamma \in (3/2,2]$, we get 
\begin{align*}\begin{aligned}
&2\left|\int_{\Omega}\left(\frac{n^\varepsilon - n}{n}\right) \Delta v\cdot (v - v^\varepsilon) \,dx\right|\\
&\quad \le \frac{2}{n_*} \int_{\Omega} |n^\varepsilon-n| |v-v^\varepsilon| |\Delta v| \, dx\\
&\quad \le  C \left( \int_{\Omega} \min\left\{(n^\varepsilon)^{\gamma-2}, n^{\gamma-2} \right\} (n- n^\varepsilon)^2\,dx \right)^{1/2} \left(\int_{\Omega} (n^{2-\gamma} + (n^\varepsilon)^{2-\gamma})|v- v^\varepsilon|^2 | \Delta v|^2\,dx\right)^{1/2}\\
&\quad \le C \int_{\Omega} \mathcal{H}(U^\varepsilon|U) \,dx + \frac18\int_{\Omega}|\nabla (v - v^\e)|^2\,dx.
\end{aligned}\end{align*}
If $\gamma>2$, we use almost the same estimates with $K_5^2$ to obtain
\begin{align*}\begin{aligned}
2\left|\int_{\Omega}\left(\frac{n^\varepsilon - n}{n}\right) \Delta v\cdot (v - v^\varepsilon) \,dx\right| &\le \frac{2\|\Delta v\|_{L^\infty}}{n_*} \int_\Omega |n^\e - n| |v-v^\e|\,dx\\
&\le C \int_{\Omega} \mathcal{H}(U^\varepsilon|U) \,dx + \frac18\int_{\Omega}|\nabla (v - v^\e)|^2\,dx.
\end{aligned}\end{align*}
This asserts
\[
K_6 \le C \int_0^t \int_{\Omega} \mathcal{H}(U^\varepsilon|U) \,dxds + \frac18\int_0^t \int_{\Omega}|\nabla (v - v^\e)|^2\,dxds, 
\]
where $C = C(\gamma, n_*, \|n\|_{L^\infty}, \Omega, \| \Delta v\|_{L^\infty(0,T;L^\infty)})$ is a positive constant.

By combining all of the above estimates, we have
\begin{align*}\begin{aligned}
&\int_{\Omega} \mathcal{H}(U^\varepsilon|U)\,dx + \frac12 \int_0^t \int_{\Omega} |\nabla (v-v^\e)|^2\,dxds +\frac{1}{2} \int_0^t \int_{\Omega} \rho^\varepsilon|( u^\varepsilon - v^\varepsilon) - (u-v)|^2\,dxds\\
&\quad \le  C\lt(\int_0^t \int_{\Omega} \mathcal{H}(U^\varepsilon|U)\,dxds + \sqrt{\varepsilon} \rt),
\end{aligned}\end{align*}
where $C = C(\gamma, n_*, \|\rho\|_{L^\infty}, \|u-v\|_{L^\infty}, \|n\|_{L^\infty}, \|\nabla v\|_{L^\infty}, \|\Delta v\|_{L^\infty}, \|\nabla u\|_{L^\infty},\Omega)$ is a positive constant.
Finally, we apply Gr\"onwall's inequality to the above to conclude the desired result.

\subsubsection{Proof of Corollary \ref{cor_main}}
In this part, we provide the strong convergences appeared in Corollary \ref{cor_main} by using the relative entropy inequality obtained in Theorem \ref{T2.2}. Since the convergence of $\rho^\e$, $\rho^\e u^\e$, and $\rho^\e u^\e \otimes u^\e$ can be obtained by the similar argument in \cite[Section 4]{K-M-T3}, we only show the strong convergence of $n^\e$, $n^\e v^\e$ and $n^\e v^\e \otimes v^\e$ below. \newline

\noindent $\diamond$ (Convergence of $n^\e$ to $n$): We now use the inequality \eqref{T1-1.1} to show the convergence of $n^\e$ to $n$. First, we have
\begin{align*}\begin{aligned}
\int_\Omega |n^\e - n|^\gamma \,dx &= \int_{\Omega \cap \{n/2 \le n^\e \le 2n \}} |n^\e -n|^\gamma \,dx + \int_{\Omega \cap\{n/2 \le n^\e \le 2n \}^c} |n^\e - n|^\gamma \,dx\cr
& =: L^\e_1 + L^\e_2. 
\end{aligned}\end{align*}
For $L^\e_1$, we find if $\gamma \in (3/2, 2]$,
\begin{align*}\begin{aligned}
L^\e_1 &\le \left(\int_{\Omega \cap \{n/2 \le n^\e \le 2n \}} \min\{(n^\e)^{\gamma-2}, n^{\gamma-2}\} |n^\e - n|^2 \,dx\right)^{\frac{\gamma}{2}} \cr
&\qquad \qquad \times \left(\int_{\Omega \cap \{n/2 \le n^\e \le 2n \}} \max\{(n^\e)^\gamma, n^\gamma\} \,dx\right)^{\frac{2-\gamma}{2}}\\
&\le C\left(\int_{\Omega \cap \{n/2 \le n^\e \le 2n \}} \mathcal{H}(U^\e|U) \,dx\right)^{\frac{\gamma}{2}} \lt( \Big(2\|n\|_{L^\infty}\Big)^\gamma|\Omega|\rt)^{\frac{2-\gamma}{2}} \to 0
\end{aligned}\end{align*}
as $\e \to 0$, where $C = C(\gamma)$ is a positive constant independent of $\e$. For $\gamma>2$, we use \eqref{T1-1.1} to obtain
\begin{align*}
L_1^\e &= \int_{\Omega \cap \{n/2 \le n^\e \le 2n \}} |n^\e -n|^2 |n^\e-n|^{\gamma-2} \,dx\\
&\le\int_{\Omega \cap \{n/2 \le n^\e \le 2n \}} |n^\e -n|^2 (3n)^{\gamma-2} \,dx\\
&\le C \int_{\Omega \cap \{n/2 \le n^\e \le 2n \}} |n^\e -n|^2\,dx\\
&\le C \int_\Omega \mathcal{H}(U^\e|U)dx \to 0
\end{align*}
as $\e \to 0$, where $C = C(\gamma, n_*, \|n\|_{L^\infty})$ is a positive constant. For the estimate of $L^\e_2$, we use \eqref{T1-1.1} to get
\begin{align*}\begin{aligned}
L^\e_2 &\le \int_{\Omega \cap\{n/2 \le n^\e \le 2n \}^c} \|n\|_{L^\infty}^\gamma \lt|\frac{n^\e}{n} + 1\rt|^\gamma dx\\
&\le \int_{\Omega \cap\{n/2 \le n^\e \le 2n \}^c} (2\|n\|_{L^\infty})^\gamma \left(\left(\frac{n^\e}{n}\right)^\gamma + 1 \right) dx\\
&\le \int_{\Omega \cap\{n/2 \le n^\e \le 2n \}^c} (2\|n\|_{L^\infty})^\gamma \left(\left(\frac{n^\e}{n_*}\right)^\gamma + 1 \right) dx\\
&\le C\int_{\Omega \cap\{n/2 \le n^\e \le 2n \}^c} (1+(n^\e)^\gamma) \,dx\\
&\le C \int_{\Omega \cap\{n/2 \le n^\e \le 2n \}^c} \mathcal{H}(U^\e | U)\,dx \to 0
\end{aligned}\end{align*}
as $\e \to 0$, where $C = C(\|n\|_{L^\infty}, n_*, \gamma)$ is a positive constant independent of $\e$. Thus we have the convergence $n^\e \to n$ in $L_{loc}^1(0,T;L^\gamma(\Omega))$ as $\e \to 0$, and this together with the integrability condition yields that it also holds in $L_{loc}^1(0,T;L^p(\Omega))$ with $p \in [1,\gamma]$.\newline

\noindent $\diamond$ (Convergence of $n^\e v^\e$ to $nv$): Similarly as before, we estimate
\begin{align*}\begin{aligned}
\int_\Omega |n^\e v^\e - n v| \,dx & \le \int_\Omega (n^\e| v^\e - v| + |n^\e - n| |v|) \,dx\\
&=:L^\e_3 + L^\e_4,
\end{aligned}\end{align*}
where $L^\e_3$ can be bounded by 
\[
L^\e_3 \le \left( \int_\Omega n^\e |v^\e - v|^2 \,dx\right)^{1/2} \left( \int_\Omega n^\e \,dx \right)^{1/2} \to 0,
\]
since $n^\e$ is integrable in $\Omega$. For the estimate of $L^\e_4$, we obtain
\[
L^\e_4 \le \|v\|_{L^\infty}|\Omega|^{\frac{\gamma-1}{\gamma}} \lt(\int_\Omega |n^\e - n|^\gamma \,dx\rt)^{1/\gamma} \to 0,
\]
as $\e \to 0$. This gives the desired result for the convergence of $n^\e v^\e$.\newline

\noindent $\diamond$ (Convergence of $n^\e v^\e \otimes v^\e$ to $nv \otimes v$): Note that the following identity holds:
\[
n^\e v^\e \otimes v^\e -nv \otimes v = n^\e v^\e \otimes (v^\e - v) + (n^\e v^\e -nv)\otimes nv  
\]
Thus, we get
\[
\begin{split}
&\int_\Omega \left| n^\e v^\e \otimes v^\e -nv \otimes v\right|\,dx \\
&\quad\le \int_\Omega n^\e |v^\e| |v^\e - v| + |n^\e v^\e -nv| |nv| \,dx\\
&\quad\le \left( \int_\Omega n^\e |v^\e|^2\,dx\right)^{1/2}\left(\int_\Omega n^\e |v^\e - v|^2 \,dx\right)^{1/2} + \|n^\e v^\e - nv\|_{L^1}\|nv\|_{L^\infty} \to 0,
\end{split}
\]
as $\e \to 0$. This the desired strong convergence of $n^\e v^\e \otimes v^\e$. \newline

\noindent $\diamond$ (Convergence of $\int_{\R^3} f^\e \xi\otimes \xi\,d\xi$ to $\rho u\otimes u + \rho \mathbb{I}_{3 \times 3}$): Adding and subtracting, we get
\[
\begin{split}
\int_{\R^3}&f^\e \xi \otimes \xi \,d\xi - (\rho u \otimes u + \rho \mathbb{I}_{3\times 3})\\
&= \lt(\int_{\R^3} f^\e \xi \otimes \xi \,d\xi- (\rho^\e u^\e \otimes u^\e + \rho^\e \mathbb{I}_{3\times3})\rt) + (\rho^\e u^\e \otimes u^\e - \rho u \otimes u) + (\rho^\e - \rho)\mathbb{I}_{3\times 3},
\end{split}
\]
where the first term on the right hand side can be rewritten as 
\[
\begin{split}
\int_{\R^3}& f^\e \xi \otimes \xi \,d\xi- (\rho^\e u^\e \otimes u^\e + \rho^\e \mathbb{I}_{3\times3})\\
&= \int_{\R^3} \Big[u^\e \sqrt{f^\e} \otimes ((u^\e -\xi)\sqrt{f^\e} -2\nabla_\xi \sqrt{f^\e}) + ((u^\e -\xi)\sqrt{f^\e} -2\nabla_\xi \sqrt{f^\e} \otimes \xi\sqrt{f^\e} \,\Big]\,d\xi
\end{split}
\]
due to \cite[Lemma 4.4]{K-M-T3}. This implies
$$\begin{aligned}
&\lt\|\int_{\R^3} f^\e \xi\otimes \xi\,d\xi - (\rho^\e u^\e\otimes u^\e + \rho^\e \mathbb{I}_{3 \times 3} )\rt\|_{L^1}\cr
&\quad \leq \lt(\int_{\Omega\times\R^3} f^\e |u^\e|^2 + f^\e |\xi|^2\,dxd\xi \rt)^{1/2}\lt(\int_{\Omega\times\R^3} \frac{1}{f^\e}|\nabla_\xi f^\e - (u^\e - \xi)f^\e|^2\,dxd\xi \rt)^{1/2}\cr
&\quad \leq C\sqrt\e \sup_{0 \leq t \leq T} \lt(\int_{\Omega\times\R^3}  f^\e |\xi|^2\,dxd\xi \rt)^{1/2}\lt(\frac{1}{2\e} \int_{\Omega\times\R^3} \frac{1}{f^\e}|\nabla_\xi f^\e - (u^\e - \xi)f^\e|^2\,dxd\xi  \rt)^{1/2}\cr
&\quad \leq C\sqrt\e\lt(\frac{1}{2\e} \int_{\Omega\times\R^3} \frac{1}{f^\e}|\nabla_\xi f^\e - (u^\e - \xi)f^\e|^2\,dxd\xi  \rt)^{1/2}.
\end{aligned}$$
Thus, we use the uniform bound estiamte in Lemma \ref{L4.1} that
$$\begin{aligned}
&\lt\|\int_{\R^3} f^\e \xi\otimes \xi\,d\xi - (\rho u\otimes u + \rho \mathbb{I}_{3 \times 3})\rt\|_{L^2(0,T;L^1(\om))}\cr
&\quad \leq \lt\|\int_{\R^3} f^\e \xi\otimes \xi\,d\xi - (\rho^\e u^\e\otimes u^\e + \rho^\e \mathbb{I}_{3 \times 3} )\rt\|_{L^2(0,T^;L^1(\om))} + \|\rho^\e u^\e \otimes u^\e - \rho u\otimes u\|_{L^2(0,T^;L^1(\om))}\cr
&\qquad  + \|\rho^\e - \rho\|_{L^2(0,T^;L^1(\om))}\cr
&\quad \leq C\sqrt{\e} + \|\rho^\e u^\e \otimes u^\e - \rho u\otimes u\|_{L^2(0,T^;L^1(\om))}+ \|\rho^\e - \rho\|_{L^2(0,T^;L^1(\om))} \to 0
\end{aligned}$$
as $\e \to 0$. \newline

\noindent $\diamond$ (Convergence of $f^\e$ to $M_{\rho,u}$):
Let us first recall the relative pressure $P$ functional:
\[
P(x|y) = x\log x - y\log y + (y-x)(1 + \log y) = x(\log x - \log y) + (y-x).
\]
This yields
\[
\int_{\om \times \R^3} P(f^\e | M_{\rho,u})\,dxd\xi = \int_{\om \times \R^3} f^\e (\log f^\e - \log \rho)\,dxd\xi + \int_{\om \times \R^3} f^\e \frac{|u - \xi|^2}{2}\,dxd\xi + \frac32\log(2\pi).
\]
Note that
\begin{align}\label{est_m0}
\begin{aligned}
\frac{d}{dt}\int_{\om \times \R^3} f^\e \log f^\e \,dxd\xi &= \int_{\om \times \R^3} \pa_t f^\e \log f^\e\,dxd\xi\cr
&= \int_{\om \times \R^3} \nabla_\xi f^\e \cdot (v^\e - \xi)\,dxd\xi - \frac1\e \int_{\om \times \R^3} \frac{\nabla_\xi f^\e}{f^\e} \cdot \lt(\nabla_\xi f^\e - (u^\e -\xi)f^\e \rt) dxd\xi
\end{aligned}
\end{align}
and
\begin{align}\label{est_m1}
\begin{aligned}
\frac{d}{dt}\int_{\om \times \R^3} f^\e \log \rho\,dxd\xi &= \frac{d}{dt}\int_\om \rho^\e \log \rho\,dx\cr
&= \int_\om \pa_t \rho^\e \log \rho\,dx + \int_\om \rho^\e \frac{\pa_t \rho}{\rho}\,dx\cr
&= \int_\om \rho^\e u^\e \cdot \frac{\nabla \rho}{\rho}\,dx + \int_\om \rho u \cdot \nabla \lt(\frac{\rho^\e}{\rho} \rt) dx \cr
&= \int_\om \rho^\e (u^\e - u) \cdot \frac{\nabla \rho}{\rho}\,dx + \int_\om u \cdot \nabla \rho^\e\,dx.
\end{aligned}
\end{align}
We next estimate
$$\begin{aligned}
\frac{d}{dt}\int_{\om \times \R^3} f^\e \frac{|u - \xi|^2}{2}\,dxd\xi 
&= \int_{\om \times \R^3} (u - \xi) \cdot (\pa_t u) f^\e\,dxd\xi + \int_{\om \times \R^3} (\pa_t f^\e) \frac{|u - \xi|^2}{2}\,dxd\xi\cr
&=: M_1+ M_2,
\end{aligned}$$
where we use the smoothness of the limiting system \eqref{A-3} to obtain
$$\begin{aligned}
M_1 &= \int_\om \rho^\e (u - u^\e) \cdot \pa_t u\,dx \cr
&= \int_\om \rho^\e (u - u^\e) \cdot (- u \cdot \nabla u - \rho^{-1}\nabla \rho + (v-u))\,dx \cr
&\leq C\int_\om \rho^\e |u - u^\e|\,dx - \int_\om \rho^\e (u - u^\e) \cdot \frac{\nabla \rho}{\rho}\,dx.
\end{aligned}$$
We notice that the last term on the right hand side of the above inequality also appears in \eqref{est_m1}. For the estimate of $M_2$, we note that
$$\begin{aligned}
M_2 &= -\int_{\pa\om\times\R^3}(\xi \cdot r(x)) |u-\xi|^2 \gamma f^\e\,d\sigma(x)d\xi +\int_{\om \times \R^3} \xi f^\e \otimes (u - \xi) :\nabla u\,dxd\xi \\
&\quad- \int_{\om \times \R^3} (v^\e - \xi) \cdot (u-\xi)f^\e\,dxd\xi + \frac1\e \int_{\om \times \R^3} (u - \xi) \cdot \lt( \nabla_\xi f^\e - (u^\e - \xi)f^\e\rt) dxd\xi.
\end{aligned}$$
Here, the integral on boundary becomes zero, since
$$\begin{aligned}
&\int_{\pa\om\times\R^3}(\xi \cdot r(x)) |u-\xi|^2\gamma f^\e \,d\sigma(x)d\xi\\
&\quad= \int_{\Sigma_+} |\xi\cdot r(x)| |u-\xi|^2\gamma_+ f^\e\,d\sigma(x)d\xi - \int_{\Sigma_-} |\xi\cdot r(x)| |u-\xi|^2 \gamma_-f^\e\,d\sigma(x)d\xi\\
&\quad= \int_{\Sigma_+} |\xi\cdot r(x)| |u-\xi|^2\gamma_+f^\e\,d\sigma(x)d\xi - \int_{\Sigma_+} \left|\Big(\xi-2(\xi\cdot r)r\Big)\cdot r\right|\lt|u-\Big(\xi-2(\xi\cdot r) r\Big)\rt|^2 \gamma^+f^\e \,d\sigma(x)d\xi\\
&\quad =0,
\end{aligned}$$
where we used $u \cdot r= 0$ to get 
\[
\lt|u-\lt(\xi-2\lt(\xi\cdot r\rt) r\rt)\rt| = |u-\xi|. 
\]
This yields
$$\begin{aligned}
M_2 &= \int_{\om \times \R^3} \xi f^\e \otimes (u - \xi) :\nabla u\,dxd\xi - \int_{\om \times \R^3} (v^\e - \xi) \cdot (u-\xi)f^\e\,dxd\xi \\
&\quad + \frac1\e \int_{\om \times \R^3} (u - \xi) \cdot \lt( \nabla_\xi f^\e - (u^\e - \xi)f^\e\rt) dxd\xi.
\end{aligned}$$
Now, we estimate the first term as
$$\begin{aligned}
&\int_{\om \times \R^3} \xi f^\e \otimes (u - \xi) :\nabla u\,dxd\xi\cr
&\quad = -\int_{\om \times \R^3}(u- \xi) f^\e \otimes (u - \xi) :\nabla u\,dxd\xi +\int_{\om\times\R^3} u\otimes (u-\xi)f^\e : \nabla u\,dxd\xi\\
 &\quad =- \int_{\om \times \R^3} \lt( (u - u^\e)\otimes (u - u^\e) +   (u^\e - \xi) \otimes (u^\e - \xi) \rt) f^\e  : \nabla u\,dxd\xi +\int_{\om\times\R^3} u\otimes (u-\xi)f^\e : \nabla u\,dxd\xi \cr
&\quad \leq \|\nabla u\|_{L^\infty} \int_\om \rho^\e |u - u^\e|^2\,dx + \|u\|_{L^\infty}\|\nabla u\|_{L^\infty}\int_\om \rho^\e |u-u^\e|\,dx \\
&\qquad- \int_{\om \times \R^3} \lt((u^\e - \xi)\sqrt{f^\e} - 2\nabla_\xi \sqrt{f^\e} \rt) \otimes (u^\e - \xi) \sqrt{f^\e} : \nabla u\,dxd\xi\cr
&\qquad - 2\int_{\om \times \R^3} \nabla_\xi \sqrt{f^\e} \otimes (u^\e - \xi) \sqrt{f^\e} : \nabla u\,dxd\xi\cr
&\quad \leq C\int_\om \rho^\e |u- u^\e|^2\,dx  + C\int_\om \rho^\e|u-u^\e|\,dx- \int_{\om \times \R^3} \nabla_\xi f^\e \otimes (u^\e - \xi) : \nabla u\,dxd\xi\cr
&\qquad + C\lt(\int_{\om \times \R^3} |u^\e - \xi|^2 f^\e\,dxd\xi \rt)^{1/2}\lt(\int_{\om \times \R^3} \frac{1}{f^\e} |\nabla_\xi f^\e - (u^\e - \xi)f^\e|^2\,dxd\xi \rt)^{1/2}\cr
&\quad = C\int_\om \rho^\e |u- u^\e|^2\,dx + \int_\om \nabla \rho^\e \cdot u\,dx \cr
&\qquad + C\lt(\int_{\om \times \R^3} |u^\e - \xi|^2 f^\e\,dxd\xi \rt)^{1/2}\lt(\int_{\om \times \R^3} \frac{1}{f^\e} |\nabla_\xi f^\e - (u^\e - \xi)f^\e|^2\,dxd\xi \rt)^{1/2}.
\end{aligned}$$
Thus we get
$$\begin{aligned}
&\frac{d}{dt}\int_{\om \times \R^3} f^\e \frac{|u - \xi|^2}{2}\,dxd\xi  \cr
&\quad \leq C\int_\om \rho^\e |u - u^\e|\,dx - \int_\om \rho^\e (u - u^\e) \cdot \frac{\nabla \rho}{\rho}\,dx + C\int_\om \rho^\e |u- u^\e|^2\,dx + \int_\om \nabla \rho^\e \cdot u\,dx\cr
&\quad + C\lt(\int_{\om \times \R^3} |u^\e - \xi|^2 f^\e\,dxd\xi \rt)^{1/2}\lt(\int_{\om \times \R^3} \frac{1}{f^\e} |\nabla_\xi f^\e - (u^\e - \xi)f^\e|^2\,dxd\xi \rt)^{1/2} \cr
&\quad - \int_{\om \times \R^3} (v^\e - \xi) \cdot (u-\xi)f^\e\,dxd\xi + \frac1\e \int_{\om \times \R^3} (u - \xi) \cdot \lt( \nabla_\xi f^\e - (u^\e - \xi)f^\e\rt) dxd\xi.
\end{aligned}$$
This combined with \eqref{est_m0} and \eqref{est_m1} gives
$$\begin{aligned}
&\frac{d}{dt}\int_{\om \times \R^3} P(f^\e | M_{\rho,u})\,dxd\xi \cr
&\quad \leq C\int_\om \rho^\e |u - u^\e|\,dx + C\int_\om \rho^\e |u- u^\e|^2\,dx   \cr
&\qquad + C\lt(\int_{\om \times \R^3} |u^\e - \xi|^2 f^\e\,dxd\xi \rt)^{1/2}\lt(\int_{\om \times \R^3} \frac{1}{f^\e} |\nabla_\xi f^\e - (u^\e - \xi)f^\e|^2\,dxd\xi \rt)^{1/2}\cr
&\qquad + \int_{\om \times \R^3} \nabla_\xi f^\e \cdot (v^\e - \xi)\,dxd\xi - \int_{\om \times \R^3} (v^\e - \xi) \cdot (u-\xi)f^\e\,dxd\xi\cr
&\qquad - \frac1\e \int_{\om \times \R^3} \frac{1}{f^\e} \lt(\nabla_\xi f^\e - (u - \xi) f^\e\rt) \cdot\lt( \nabla_\xi f^\e - (u^\e - \xi)f^\e \rt)dxd\xi\cr
&\quad \leq C \e^{1/4} + C\e \int_{\om \times \R^3} |\xi|^2 f^\e\,dxd\xi + \frac1{4\e}\int_{\om \times \R^3} \frac{1}{f^\e} |\nabla_\xi f^\e - (u^\e - \xi)f^\e|^2\,dxd\xi \cr
&\qquad + \int_{\om \times \R^3} \nabla_\xi f^\e \cdot (v^\e - \xi)\,dxd\xi - \int_{\om \times \R^3} (v^\e - \xi) \cdot (u-\xi)f^\e\,dxd\xi\cr
&\qquad - \frac1\e \int_{\om \times \R^3} \frac{1}{f^\e} \lt(\nabla_\xi f^\e - (u - \xi) f^\e\rt) \cdot\lt( \nabla_\xi f^\e - (u^\e - \xi)f^\e \rt)dxd\xi
\end{aligned}$$
due to Theorem \ref{T2.2} and
\[
\int_{\om \times \R^3} |u^\e - \xi|^2 f^\e\,dxd\xi \leq 2\int_{\om \times \R^3} \lt(|u^\e|^2 + |\xi|^2\rt) f^\e\,dxd\xi \leq 4\int_{\om \times \R^3} |\xi|^2 f^\e\,dxd\xi.
\]
We next estimate the last three terms on the right hand side of the above inequality. Note that
$$\begin{aligned}
&\int_{\om \times \R^3} \nabla_\xi f^\e \cdot (v^\e - \xi)\,dxd\xi - \int_{\om \times \R^3} (v^\e - \xi) \cdot (u-\xi)f^\e\,dxd\xi\cr
&\quad  = \int_{\om \times \R^3} \lt(\nabla_\xi f^\e - (u^\e - \xi)f^\e\rt) \cdot (v^\e - \xi)\,dxd\xi + \int_{\om \times \R^3} (v^\e - \xi) \cdot (u^\e - u)f^\e\,dxd\xi\cr
&\quad \leq \lt(\int_{\om \times \R^3} \frac{1}{f^\e} |\nabla_\xi f^\e - (u^\e - \xi)f^\e|^2 \,dxd\xi\rt)^{1/2}\lt(\int_{\om \times \R^3} |v^\e - \xi|^2 f^\e\,dxd\xi \rt)^{1/2} \cr
&\qquad + \lt(\int_\om |v^\e - u^\e|^2 \rho^\e\,dx \rt)^{1/2}\lt(\int_\om |u^\e - u|^2 \rho^\e\,dx \rt)^{1/2} \cr
&\quad \leq C\e \int_{\om \times \R^3} |v^\e - \xi|^2 f^\e\,dxd\xi  + \frac1{4\e}\int_{\om \times \R^3} \frac{1}{f^\e} |\nabla_\xi f^\e - (u^\e - \xi)f^\e|^2 \,dxd\xi\cr
&\qquad + C\e^{1/4} \lt(\int_{\om\times\R^3} |v^\e - \xi|^2 f^\e\,dxd\xi \rt)^{1/2}.
\end{aligned}$$
For the last term, by adding and subtracting, we estimate
$$\begin{aligned}
&- \frac1\e \int_{\om \times \R^3} \frac{1}{f^\e} \lt(\nabla_\xi f^\e - (u - \xi) f^\e\rt) \cdot\lt( \nabla_\xi f^\e - (u^\e - \xi)f^\e \rt)dxd\xi\cr
&\quad = - \frac1\e \int_{\om \times \R^3} \frac{1}{f^\e} | \nabla_\xi f^\e - (u^\e - \xi)f^\e |^2dxd\xi - \frac1\e \int_{\om \times \R^3} (\nabla_\xi f^\e - (u^\e - \xi)f^\e) \cdot (u^\e - u)\,dxd\xi\cr
&\quad = - \frac1\e \int_{\om \times \R^3} \frac{1}{f^\e} | \nabla_\xi f^\e - (u^\e - \xi)f^\e |^2dxd\xi. 
\end{aligned}$$
Combining the previous estimates and integrating the resulting inequality over the time interval $[0,t]$, we obtain
$$\begin{aligned}
&\int_{\om \times \R^3} P(f^\e | M_{\rho,u})\,dxd\xi + \frac1{2\e} \int_0^t\int_{\om \times \R^3} \frac{1}{f^\e} | \nabla_\xi f^\e - (u^\e - \xi)f^\e |^2dxd\xi ds \cr
&\quad \leq \int_{\om \times \R^3} P(f^\e_0 | M_{\rho_0,u_0})\,dxd\xi +  C \e^{1/4} + C\e \int_0^t\int_{\om \times \R^3} |\xi|^2 f^\e\,dxd\xi ds\cr
&\qquad +  C\e \int_0^t\int_{\om \times \R^3} |v^\e - \xi|^2 f^\e\,dxd\xi ds + C\e^{1/4} \lt(\int_0^t\int_{\om\times\R^3} |v^\e - \xi|^2 f^\e\,dx d\xi ds\rt)^{1/2}.
\end{aligned}$$
On the other hand, Theorem \ref{T0.1} and Lemma \ref{L4.2} imply that all of the integrals in time are uniformly bounded in $\e$, and thus we have
\bq\label{est_m3}
\int_{\om \times \R^3} P(f^\e | M_{\rho,u})\,dxd\xi \leq \int_{\om \times \R^3} P(f^\e_0 | M_{\rho_0,u_0})\,dxd\xi + C\e^{1/4}.
\eq
Since
$$\begin{aligned}
&\lt(\int_{\om \times \R^3} |f^\e - M_{\rho,u}|\,dxd\xi\rt)^2 \cr
&\quad \leq \lt(\int_{\om \times \R^3} (f^\e + M_{\rho,u})\,dxd\xi \rt)\lt(\int_{\om \times \R^3} \min\lt\{\frac{1}{f^\e}, \frac{1}{M_{\rho,u}}\rt \} |f^\e - M_{\rho,u}|^2\,dxd\xi \rt)\cr
&\quad \leq 4\int_{\om \times \R^3} P(f^\e | M_{\rho,u})\,dxd\xi,
\end{aligned}$$
we also easily get from \eqref{est_m3} the quantitative error estimate between $f^\e$ and $M_{\rho,u}$ in $L^1(\om \times \R^3)$. This completes the proof.

%
%
%
%
%
%
%
%

\section{Global existence of weak solutions to kinetic-fluid system}\label{sec:3}
\setcounter{equation}{0}
In this section, we prove the existence of weak solutions to the system \eqref{A-1} with general reflection boundary conditions. In order to state the boundary condition, inspired by \cite{M-V}, we write a reflection operator $\mathcal{B}$ as
\[ 
\mathcal{B}(f)(x,\xi) := \int_{\xi' \cdot r >0} B(t,x,\xi,\xi')f(x,\xi') |\xi' \cdot r(x)| \,d\xi', 
\]
where $B: \R_+ \times \om \times \R^3 \times \R^3 \to \R$ is called the scattering kernel which describes the probability that a particle with velocity $\xi'$ at time $t >0$ striking the boundary at $x \in \pa \om$ back-scatters to the domain with velocity $\xi$ at the same location $x$ and time $t$. 

In the current work, the following assumptions are imposed:
\begin{itemize}
\item[(i)] The reflection operator $\mathcal{B}$ is non-negative. 
\item[(ii)] For any $\xi' \in \bbr^3$ satisfying $\xi' \cdot r(x) >0$, we have 
\[\int_{\xi \cdot r(x) <0 } B(t,x,\xi,\xi')|\xi\cdot r(x)| \,d\xi =1. \]
\item[(iii)] For the Maxwellian distribution $M(\xi) := (2\pi)^{-3/2}\exp(-|\xi|^2)$, we have
\[\int_{\xi' \cdot r >0} B(t,x,\xi,\xi')M(\xi')|\xi'\cdot r(x)|\,d\xi' = M(\xi). \]
\item[(iv)] The operator $\mathcal{B}$ is a bounded operator from $L^p(\Sigma_+)$ into $L^p(\Sigma_-)$ for all $p \in [1,\infty]$, that is,
\[
\|\mathcal{B}\|_{\mathcal{L}(L^p(\Sigma_+), L^p(\Sigma_-))} \le 1. 
\]
\end{itemize}

As mentioned in Introduction, there are few previous works on the initial-boundary value problem for the kinetic equation coupled with the fluid equations. In particular, up to the authors' limited knowledge, there is only one work \cite{M-V} on the coupling with compressible fluids in a bounded domain with reflection type, also absorbing type, boundary conditions. Apart from the coupling fluid equations, considering the physically relevant boundary conditions for kinetic equations is a very hard problem, due to the lack of regularity of the trace of $f$ along the boundary. We refer to \cite{CKL19,GKTT17} for recent progress on collisional kinetic equations in a bounded domain. 

We now state our main result on the global-in-time existence of weak solutions for the system \eqref{A-1}.
\begin{theorem}\label{T2.1}
Let $\gamma >3/2$, $T \in (0,\infty)$ and assume that the initial data $(f_0, n_0, v_0)$ satisfy
\begin{equation}\label{init-1}
f_0 \in (L_+^1\cap L^\infty)(\Omega \times \bbr^3), \quad n_0 \in L_+^1(\Omega), \quad \mbox{and} \quad \mathcal{F}(f_0, n_0, v_0) < \infty,
\end{equation}
where $\mathcal{F}(f,n,v)$ is defined as
\[
\mathcal{F}(f, n, v) := \int_{\Omega \times \bbr^3 } f \left( \log f  +\frac{|\xi|^2}{2} \right) dxd\xi + \int_{\Omega} \frac{1}{2} n |v|^2 \,dx + \frac{1}{\gamma-1} \int_{\Omega} n^\gamma \,dx. 
\]
Then there exists at least one weak solution $(f,n,v)$ to the system \eqref{A-1} with the homogeneous Dirichlet boundary condition \eqref{bdy_ho} for $v$ and the following reflection boundary condition for $f$:
\[
\gamma_-f(x,\xi,t) = \mathcal{B}(\gamma_+ f)(x,\xi,t) \quad \forall (x,\xi,t) \in \Sigma_- \times \R_+
\]
in the sense of Definition \ref{D2.1}\footnote{The condition (iv) should be replaced by the following: for any $\varphi \in \mathcal{C}_c^2(\bar\Omega \times \bbr^3 \times [0,T])$ satisfying $\varphi(\cdot, \cdot,T) =0$ and $\gamma_+\varphi = \mathcal{B}^*\gamma_-\varphi$ on $\Sigma_+ \times [0,T]$,
$$\begin{aligned}
\int_0^T & \int_{\Omega \times \bbr^3} f \lt( \partial_t \varphi + \xi \cdot \nabla \varphi + (v-\xi ) \cdot \nabla_\xi \varphi + \Delta_\xi \varphi + (u-\xi)\cdot \nabla_\xi \varphi \rt)dxd\xi dt\\
&  +\int_{\Omega \times \bbr^3} f_0 \varphi(x,\xi,0) \,dxd\xi =0.
\end{aligned}$$}. Moreover, the following entropy inequality holds:
\[
\mathcal{F}(f,n,v)(t) + \int_0^t \md(f,v)(s)\,ds + \int_0^t \int_\Omega |\nabla v|^2\, dxds \le \mathcal{F}(f_0, n_0, v_0) + 3t \|f_0\|_{L^1(\Omega\times\bbr^3)},
\]
where $\md(f,v)$ is given by
\[
\md(f,v) := \int_{\Omega \times \bbr^3} \frac{1}{f}|(u - \xi)f -\nabla_\xi f|^2 + |v-\xi|^2f \,dxd\xi.
\]
\end{theorem}
\begin{remark} We recall from \cite{M-V} that the condition (ii) implies if $\gamma \varphi(x,\xi) $ is independent of $\xi$, then $\gamma_+ \varphi = \mathcal{B}^*\gamma_-\varphi$, where $\mathcal{B}^*$ is the adjoint operator. This implies that the weak formulation for $f$ holds for test function $\varphi$, which is independent of $\xi$.
\end{remark}
We also notice from \cite{C-I-P} that the above three conditions (i)-(iii) give the following lemma.
\begin{lemma}
Assume that $f$ satisfies $\gamma f \ge 0$, $\gamma_-f = \mathcal{B}\gamma_+f$ and $(1+|\xi|^2 +|\log \gamma f |)\gamma f \in L^1(\Sigma^\pm)$. Then we have
\[
\int_{\bbr^3} (\xi \cdot r(x)) \gamma f \,d\xi = 0 \quad \mbox{and} \quad \int_{\bbr^3} (\xi \cdot r(x)) \left( \frac{|\xi|^2}{2} + \log(\gamma f) \right) \gamma f \,d\xi \ge 0. 
\]
\end{lemma}
For the proof of Theorem \ref{T2.1}, motivated from \cite{C,K-M-T,M-V}, we first regularize the local alignment force term in the kinetic equation in \eqref{A-1} and approximate the system \eqref{A-1} as a nonhomogeneous Dirichlet boundary, absorbing-type boundary,  value problem by fixing the trace. We then show some uniform bound estimate and use the compactness arguments \cite{F,L} to pass to the limit of the regularization parameters. This asserts that weak solutions to the system \eqref{A-1} exist globally in time, and they satisfy the entropy inequality.

%
%
%
%
%
\subsection{Regularized and approximated system}
As mentioned above, we regularize system \eqref{A-1} as follows:
\begin{align}
\begin{aligned}\label{C-1}
&\partial_t f + \xi \cdot \nabla f + \nabla_\xi \cdot ((v -\xi)f) =\Delta_\xi f - \nabla_\xi \cdot \left( \left(u_\e - \xi \right) f \right), \quad (x,\xi,t) \in \om \times \R^3 \times \R_+,\\
&\partial_t n + \nabla \cdot (n v)=0, \quad (x,t) \in \om \times \R_+,\\
&\partial_t (n v) + \nabla \cdot (n v \otimes v) + \nabla p  -\Delta v = -\rho(u-v), 
\end{aligned}
\end{align}
where the regularized local particle velocity $u_\e = u_\e(x,t)$ is given by
\[
u_\e := \frac{\rho u}{\rho +\e}. 
\]
Here the solutions $(f,u,v)$ depend on the regularization parameter $\e$, however, we do not specify it for notational simplicity. 

We then establish the global-in-time existence of weak solutions to the regularized system \eqref{C-1} when the nonhomogeneous boundary condition is taken into account.
\begin{theorem}\label{T3.1}
Let $\gamma >3/2$, $T \in (0,\infty)$ and assume that the initial data $(f_0, n_0, v_0)$ satisfy \eqref{init-1} and $g$ satisfies
\begin{align}
\begin{aligned}\label{DC}
&g(x,\xi,t) \ge 0, \quad g \in (L^1\cap L^\infty)(\Sigma_- \times (0,T)), \quad \mbox{and}\\
&\int_0^T \int_{\Sigma_-} |\xi|^2 g(x,\xi) |\xi \cdot r(x)| \,d\sigma(x) d\xi dt <\infty.
\end{aligned}
\end{align}
Then, there exists at least one weak solution $(f,n,v)$ to the system \eqref{C-1} in the sense of Definition \ref{D2.1}\footnote{In the Dirichlet boundary case, we consider the following weak formulation for $f$ instead of (iv): for any $\varphi \in \mathcal{C}_c^2(\bar\Omega \times \bbr^3 \times [0,T])$ with $\varphi(\cdot,\cdot,T) = 0$,
$$\begin{aligned}
\int_0^T & \int_{\Omega \times \bbr^3} f \lt( \partial_t \varphi + \xi \cdot \nabla \varphi + (v-\xi ) \cdot \nabla_\xi \varphi + \Delta_\xi \varphi + (u-\xi)\cdot \nabla_\xi \varphi \rt) dxd\xi dt\\
&  +\int_{\Omega \times \bbr^3} f_0 \varphi(x,\xi,0) \,dxd\xi + \int_0^T \int_{\Sigma} (\xi \cdot r(x)) \gamma f \varphi \,d\sigma(x)d\xi dt =0.
\end{aligned}$$} where $f$ satisfies the Dirichlet boundary condition:
\bq\label{diri_bdy}
 \gamma_- f(x,\xi,t) = g(x,\xi), \quad \forall (x, \xi) \in \Sigma_-.
\eq
Moreover, $f$ satisfies the following additional bounds: 
\begin{align}
\begin{aligned}\label{C-2}
& \|f\|_{L^\infty(0,T;L^p(\Omega \times \bbr^3))} +  \|\nabla_\xi f^{\frac{p}{2}}\|_{L^2(\Omega \times \bbr^3 \times (0,T))}^{\frac{2}{p}} \cr
&\quad \le e^{\frac{C}{p'}T} \lt( \|f_0\|_{L^p(\Omega \times \bbr^3)}  +  \|g\|_{L^p(\Sigma_- \times (0,T))} \rt),\\
& \|\gamma_+ f\|_{L^p(\Sigma_+ \times (0,T))} \le \|f_0\|_{L^p(\Omega \times \bbr^3)} + \|g\|_{L^p(\Sigma_- \times (0,T))},
\end{aligned}
\end{align}
and the following entropy inequality:
\begin{align}
\begin{aligned}\label{T3-1}
&\mathcal{F}(f,n,v)(t) + \int_0^t \md_\e (f,v)(s)\,ds + \int_0^t \int_\Omega |\nabla v|^2 \,dxds\\
& \quad + \int_0^t \int_\Sigma (\xi \cdot r(x)) \left( \frac{|\xi|^2}{2} + \log \gamma f + 1 \right) \gamma f \,d\sigma(x) d\xi ds\\
&\qquad \le \mathcal{F}(f_0, n_0, v_0) + 3 \int_0^t \|f(\cdot,\cdot,s)\|_{L^1(\Omega \times \bbr^3)}\,ds,
\end{aligned}
\end{align}
where $\md_\e(f,v)$ is given by
\[
\md_\e(f,v) := \int_{\Omega \times \bbr^3} \frac{1}{f}|(u_\e - \xi)f -\nabla_\xi f|^2 + |v-\xi|^2f \,dxd\xi.
\]
\end{theorem}
\begin{proof}
Since the proof is rather lengthy and technical, we leave the proof in Appendix A.
\end{proof}

\begin{remark}Independently from Theorem \ref{T2.1}, we can also obtain the global-in-time existence of weak solutions to the system \eqref{A-1} with the homogeneous Dirichlet boundary condition for $v$ \eqref{bdy_ho} and the nonhomogeneous Dirichlet  boundary condition for $f$ \eqref{diri_bdy} by using the result of Theorem \ref{T3.1} and the compactness argument in Section \ref{sec_T2.1} below.
\end{remark}

By using Theorem \ref{T3.1}, we approximate the regularized system \eqref{C-1} in the following way: let $\delta \in (0,1)$ and construct a sequence $(f^{m+1}, n^{m+1}, v^{m+1})$ of solution to \eqref{C-1} with the Dirichlet boundary condition:
\[
\gamma_- f^{m+1} = (1-\delta)\mathcal{B}\gamma_+ f^m \quad \mbox{on} \quad \Sigma_- \times (0,T)
\]
for $m \in \N \cup \{0\}$ with $\gamma_+ f^0 = 0$. Then, the estimate \eqref{C-2} gives 
\begin{align*}\begin{aligned}
 &\sup_{t \in [0,T]}\|f^{m+1}(t)\|_{L^1(\Omega \times \bbr^3)} \le e^{\frac{CT}{p'}}\left(\|f_0\|_{L^1(\Omega \times \bbr^3)} + (1-\delta)\|\gamma_+ f^m\|_{L^1(\Sigma_+ \times (0,T))}\right) \quad \mbox{and}\\
& \|\gamma_+ f^{m+1} \|_{L^p(\Sigma_+ \times (0,T))} \le \|f_0\|_{L^p(\Omega \times \bbr^3)} + (1-\delta) \|\gamma_+ f^m\|_{L^p(\Sigma_+ \times (0,T))}, \quad \forall p \in [1, \infty]
\end{aligned}\end{align*}
due to the assumption (iv) of the reflection operator $\mathcal{B}$; $\|\mathcal{B}\|_{\mathcal{L}(L^p(\Sigma_+), L^p(\Sigma_-))} \le 1$.
By iterating the above estimates, we obtain
\begin{align*}\begin{aligned}
& \sup_{t \in [0,T]}\|f^{m+1}(t)\|_{L^1(\Omega \times \bbr^3)} \le e^{\frac{CT}{p'}}\left(\frac{1}{\delta}\|f_0\|_{L^1(\Omega \times \bbr^3)} + (1-\delta)^m\|\gamma_+ f^1\|_{L^1(\Sigma_+ \times (0,T))}\right) \quad \mbox{and}  \\
 &\|\gamma_+ f^{m+1}\|_{L^p(\Sigma_+ \times (0,T))} \le \frac{1}{\delta} \|f_0\|_{L^p} + (1-\delta)^m \|\gamma_+ f^1\|_{L^p(\Sigma_+ \times (0,T))}, \quad \forall p \in [1, \infty],
\end{aligned}\end{align*}
which gives $f^m$ and $\gamma_+ f^m$ are uniformly bounded in $m$. On the other hand, from the property of $\mathcal{B}$ we find
\begin{align*}\begin{aligned}
\int_0^T & \int_{\Sigma_-} |\xi \cdot r(x)| \left( \frac{|\xi|^2}{2} + \log (\gamma_- f^{m+1}) + 1 \right) \gamma_- f^{m+1} \,d\sigma(x) d\xi dt\\
& \le (1-\delta) \int_0^T \int_{\Sigma_+} |\xi \cdot r(x)| \left( \frac{|\xi|^2}{2} + \log(\gamma_+ f^m) +1 \right) \gamma_+ f^m \,d\sigma(x) d\xi dt,
\end{aligned}\end{align*}
and this yields
$$\begin{aligned}
&\mathcal{F}(f^{m+1},n^{m+1},v^{m+1})(t) + \int_0^t \md_\e(f^{m+1},v^{m+1})\,ds + \int_0^t \int_\Omega |\nabla v^{m+1}|^2 \,dxds\\
& \quad + \int_0^t \int_{\Sigma_+} (\xi \cdot r(x)) \left( \frac{|\xi|^2}{2} + \log( \gamma_+ f^{m+1}) + 1 \right) \gamma_+ f^{m+1} \,d\sigma(x) d\xi ds\\
&\qquad \le \mathcal{F}(f_0, n_0, v_0) + 3 \int_0^t \|f^{m+1}(s)\|_{L^1(\Omega \times \bbr^3)} \,ds\\
&\qquad \quad +(1-\delta)\int_0^t \int_{\Sigma_+} (\xi \cdot r(x)) \left( \frac{|\xi|^2}{2} + \log (\gamma_+ f^m) + 1 \right) \gamma_+ f^m \,d\sigma(x) d\xi ds.
\end{aligned}$$
Here, we use the uniform bound for $f^m$ and Lemma \ref{LA.4} to get
\begin{align*}\begin{aligned}
&\int_{\Omega \times \bbr^3} \left( \frac{|\xi|^2}{4} + |\log f^{m+1}| \right) f^{m+1} \,dx d\xi + \int_\Omega n^{m+1} \frac{|v^{m+1}|^2}{2}  + \frac{1}{\gamma-1}(n^{m+1})^\gamma \,dx\\
&\quad +\frac{1}{2} \int_0^t \int_{\Sigma_+} |\xi \cdot r(x)| \left( \frac{|\xi|^2}{2} + \log(\gamma_+f^{m+1}) +1 \right) \gamma_+f^{m+1} \,d\sigma(x)d\xi ds\\
& \quad+ \int_0^t \md_\e (f^{m+1},v^{m+1})\,ds + \int_0^t \int_\Omega |\nabla v^{m+1}|^2 \,dxds\\
& \qquad \le \mathcal{F}(f_0, n_0, v_0)  +\frac{3t}{\delta}\|f_0\|_{L^1(\Omega \times \bbr^3)} + 3t\|\gamma_+ f^1\|_{L^1(\Sigma_+ \times (0,T))} + C\\
& \qquad \quad + (1-\delta)\int_0^t \int_{\Sigma_+} |\xi \cdot r(x)| \left(\frac{|\xi|^2}{2} + \log (\gamma_+ f^m) +1 \right) \gamma_+ f^m \,d\sigma(x)d\xi ds,
\end{aligned}\end{align*}
where $C$ is a constant independent of $m$ and $\delta$ and this leads us to
\begin{align}
\begin{aligned}\label{C-11}
\frac{1}{2}& \int_0^t \int_{\Sigma_+} |\xi \cdot r(x)| \left( \frac{|\xi|^2}{2} + \log(\gamma_+f^{m+1}) +1 \right) \gamma_+f^{m+1} \,d\sigma(x)d\xi ds\\
& \le \mathcal{F}(f_0, n_0, v_0)  +\frac{3t}{\delta}\|f_0\|_{L^1(\Omega \times \bbr^3)} + 3t\|\gamma_+ f^1\|_{L^1(\Sigma_+ \times (0,T))} + C\\
&\quad+ (1-\delta)\int_0^t \int_{\Sigma_+} |\xi \cdot r(x)| \left(\frac{|\xi|^2}{2} + \log (\gamma_+ f^m) +1 \right) \gamma_+ f^m \,d\sigma(x)d\xi ds.
\end{aligned}
\end{align}
By iterating \eqref{C-11}, we get
\begin{align*}\begin{aligned}
\frac{1}{2}& \int_0^t \int_{\Sigma_+} |\xi \cdot r(x)| \left( \frac{|\xi|^2}{2} + \log(\gamma_+f^{m+1}) +1 \right) \gamma_+f^{m+1} \,d\sigma(x)d\xi ds\\
& \le \frac{1}{\delta}\left(\mathcal{F}(f_0, n_0, v_0)  +\frac{3t}{\delta}\|f_0\|_{L^1(\Omega \times \bbr^3)} + 3t\|\gamma_+ f^1\|_{L^1(\Sigma_+ \times (0,T))} + C\right)\\
&\quad+ (1-\delta)^m\int_0^t \int_{\Sigma_+} |\xi \cdot r(x)| \left(\frac{|\xi|^2}{2} + \log (\gamma_+ f^1) +1 \right) \gamma_+ f^1 \,d\sigma(x)d\xi ds,
\end{aligned}\end{align*}
and we again use Lemma \ref{LA.4} to obtain
\begin{align*}\begin{aligned}
\frac{1}{2}& \int_0^t \int_{\Sigma_+} |\xi \cdot r(x)| \left( \frac{|\xi|^2}{4} + |\log(\gamma_+f^{m+1})| +1 \right) \gamma_+f^{m+1} \,d\sigma(x)d\xi ds\\
& \le C+ \frac{1}{\delta}\left(\mathcal{F}(f_0, n_0, v_0)  +\frac{3t}{\delta}\|f_0\|_{L^1(\Omega \times \bbr^3)} + 3t\|\gamma_+ f^1\|_{L^1(\Sigma_+ \times (0,T))} + C\right)\\
&\quad + (1-\delta)^m\int_0^t \int_{\Sigma_+} |\xi \cdot r(x)| \left(\frac{|\xi|^2}{2} + \log (\gamma_+ f^1) +1 \right) \gamma_+ f^1 \,d\sigma(x)d\xi ds.
\end{aligned}\end{align*}
Since we have all the uniform estimates that enable us to use the arguments in Appendix \ref{app_3}, we can pass to the limit $m\to \infty$ and obtain the following  lemma. 
\begin{lemma}
Assume that the initial data $(f_0, n_0, v_0)$ satisfies \eqref{init-1}. Then, for every $\delta \in (0,1)$, there exists at least one weak solution to \eqref{A-1} with boundary value $\gamma f^\delta \in L^p(\Sigma_+ \times (0,T))$ such that
 \[ \gamma_- f^\delta = (1-\delta)\mathcal{B}\gamma_+ f^\delta \qquad \mbox{on } \  \Sigma_- \times (0,T).  \]
Moreover, the following entropy inequaliy holds:
\begin{align}
\begin{aligned}\label{C-12}
\mathcal{F}&(f^\delta,n^\delta,v^\delta)(t) + \int_0^t D_\e(f^\delta,v^\delta)(s)\,ds + \int_0^t \int_\Omega |\nabla v^\delta|^2 \,dxds\\
& \quad + \delta \int_0^t \int_{\Sigma_+} (\xi \cdot r(x)) \left( \frac{|\xi|^2}{2} + \log( \gamma_+ f^\delta) + 1 \right) \gamma_+ f^\delta \,d\sigma(x) d\xi ds\\
&\le \mathcal{F}(f_0, n_0, v_0) + 3 \int_0^t \|f^\delta (\cdot,\cdot,s)\|_{L^1(\Omega \times \bbr^3)} \,ds.
\end{aligned}
\end{align}
\end{lemma}

\subsection{Proof of Theorem \ref{T2.1}}\label{sec_T2.1} Now, we let the regularization parameter $\e$ in the system \eqref{C-1} be $\e = \delta$, and tend $\delta$ to 0 to get the desired result. Note that the rest of the proof is almost the same as in \cite{M-V}, except for two things; one is the $L^1$-norm of $f$ on the right hand side of \eqref{C-12}, and the other one is the convergence of $(u_\delta^\delta-\xi)f^\delta$ toward $(u-\xi)f$ in distributional sense. Thus in the rest of this subsection, we only focus on the convergence $(u_\delta^\delta-\xi)f^\delta \to (u-\xi)f$ as $\delta \to 0$ in the sense of distribution. 

It follows from Proposition \ref{PA.1} and the boundary condition that
\[
\begin{split}
\int_{\Omega\times\bbr^3}& (f^\delta)^p \,dxd\xi + (1-(1-\delta)^p)\int_0^t \int_{\Sigma_+} |\xi \cdot r(x)| (\gamma_+ f^\delta)^p \,d\sigma(x)d\xi ds\\
&\le \|f_0\|_{L^p(\Omega\times\bbr^3)}^p  + 3(p-1)\int_0^t \int_{\Omega\times\bbr^3} (f^\delta)^p \,dxd\xi ds,
\end{split} 
\]
which implies the uniform boundedness of $f^\delta$ in $L^p((0,T)\times\Omega)$, $p \in [1,\infty)$. Especially, when $p=1$, we directly get
\[
\int_{\Omega\times\bbr^3} f^\delta \,dxd\xi \le \|f_0\|_{L^1(\Omega\times\bbr^3)}. 
\]
Together with Lemma \ref{LA.4}, we deduce from \eqref{C-12} that
$$\begin{aligned}
\int_{\Omega \times \bbr^3}& \left( \frac{|\xi|^2}{4} + |\log f^\delta| \right) f^\delta \,dx d\xi + \int_\Omega n^\delta \frac{|v^\delta|^2}{2}  + \frac{1}{\gamma-1}(n^\delta)^\gamma \,dx\\
&\quad +\delta\int_0^t \int_{\Sigma_+} |\xi \cdot r(x)| \left( \frac{|\xi|^2}{4} + \log(\gamma_+f^\delta) +1 \right) \gamma_+f^\delta \,d\sigma(x)d\xi ds\\
& \quad+ \int_0^t D_\e (f^\delta,v^\delta)(s)\,ds + \int_0^t \int_\Omega |\nabla v^\delta|^2 \,dxds\\
& \le \mathcal{F}(f_0, n_0, v_0)  +3t\|f_0\|_{L^1(\Omega \times \bbr^3)} + C,
\end{aligned}$$
and this uniform bound enables us to exploit the velocity averaging lemma, Lemma \ref{LA.2} to get, up to a subsequence,
\begin{align*}\begin{aligned}
&\rho^\delta \to \rho \quad \quad \mbox{in} \quad L^p(\Omega \times (0,T)) \quad \mbox{and} \quad \mbox{a.e.},\\
&\rho^\delta u^\delta \to \rho u \quad \mbox{in} \quad L^p(\Omega \times (0,T)), \quad \forall p \in (1,5/4).
\end{aligned}\end{align*}
With these strong convergences, the convergence 
\[
\xi f^\delta \to \xi f \quad \mbox{in} \quad \mathcal{D}'(\om \times (0,T))
\] 
is clear. We now show the following convergence:
\[
u_\delta^\delta f^\delta \to u f \quad \mbox{in} \quad \mathcal{D}'(\om \times (0,T)).
\]
Although the proof is almost the same as that of \cite[Lemma 4.4]{K-M-T}, for the completeness of our work, we sketch it here. For $\varphi = \varphi(\xi) \in \mc_c^\infty(\bbr^3)$, let 
\[
\rho_\varphi^\delta := \int_{\bbr^3} f^\delta (x,\xi)\varphi(\xi)\,d\xi.
\] 
Then for a test function $\psi(x,\xi,t) := \phi(x,t)\varphi(\xi)$ with $\phi \in \mc^\infty_c(\om \times (0,T))$ we have
\[
\int_0^T \int_{\om \times \R^3} f^\delta u_\delta^\delta \psi\,dxd\xi dt = \int_0^T \int_\om   u_\delta^\delta \rho_\varphi^\delta \phi\,dx dt.
\]
Note that 
\[
\|u_\delta^\delta \rho_\varphi^\delta \|_{L^p} \leq \|\varphi\|_{L^\infty}\|\rho^\delta\|_{L^{p/(2-p)}}^{1/2}\|\sqrt{\rho^\delta} u_\delta^\delta\|_{L^2} < \infty,
\]
where we were able to use the uniform bound estimates of $\rho^\delta$ and the kinetic energy of $f^\delta$ since $p/(2-p) \in (1,5/3)$ for $p \in (1,5/4)$. This yields that there exists a limiting function $m \in L^\infty(0,T;L^p(\om))$ such that
\[
u_\delta^\delta \rho_\varphi^\delta \rightharpoonup m \quad \mbox{in} \quad L^\infty(0,T;L^p(\om))
\]
as $\delta \to 0$, up to a subsequence, for all $p \in (1,5/4)$. We now claim that $m = u\rho_\varphi$, where
\[
\rho_\varphi = \int_{\R^3} f \varphi\,d\xi \quad \mbox{and} \quad u = \int_{\R^3} \xi f\,d\xi \bigg/ \int_{\R^3} f\,d\xi.
\]
Let $\zeta>0$, and define a set 
\[
E_R^\zeta := \{(x,t) \in (B(0,R)\cap \om) \times (0,T) : \rho(x,t) > \zeta \}.
\] 
Then by the compactness of $\rho^\delta$ together with Egorov's theorem, for any $\lambda > 0$, there exists a set $C_\lambda \subset E_R^\zeta$ with $|E_R^\zeta \setminus C_\lambda | < \lambda$ on which $\rho^\delta$ uniformly converges to $\rho$. This further implies $\rho^\delta > \zeta/2$ in $C_\lambda$ for $\delta >0$ small enough. Thus we obtain
\[
u_\delta^\delta \rho^\delta_\varphi = \frac{\rho^\delta u^\delta}{\rho^\delta + \delta} \to m = u\rho_\varphi \quad \mbox{in} \quad C_\lambda.
\]
On the other hand, since $\lambda > 0$, $R > 0$, and $\zeta > 0$ are arbitrary, this further yields
\[
m = u\rho_\varphi \quad \mbox{on} \quad \{ \rho > 0\}.
\]
Hence we have
\[
\int_0^T \int_{\om \times \R^3} f^\delta u_\delta^\delta \psi\,dx\xi dt \to \int_0^T \int_\om u\rho_\varphi \phi\,dxdt = \int_0^T \int_{\om \times \R^3} f u \psi\,dx\xi dt
\]
for all test functions of the form $\psi(x,\xi,t) = \phi(x,t)\varphi(\xi)$. This completes the proof.

\subsection{Proof of Theorem \ref{T0.1}} Theorem \ref{T0.1} is a direct consequence of Theorem \ref{T2.1} since the specular reflection can be formulated through the reflection operator $\mathcal{B}$. More specifically, let us set the scattering kernel $B$ as
\[
B(t,x,\xi,\xi') = \left\{\begin{array}{lc} \displaystyle \frac{\delta_{\xi' - \mathcal{R}_x(\xi)}}{|\xi'\cdot r(x)|} & \mbox{ if }\ \xi' \cdot r(x) \neq 0\\[4mm]
0 & \mbox{ if } \ \xi'\cdot r(x) = 0\end{array}\right.,
\]
where $\delta_\cdot$ denotes the Dirac measure. Note that the reflection operator $\mathcal{R}_x$ satisfies $|\mathcal{R}_x(\xi)| = |\xi|$ and $|\mathcal{R}_x(\xi) \cdot r(x)| = |\xi \cdot r(x)|$. This yields that the kernel $B$ defined as above satisfies all the conditions (i)-(iv) for the reflection operator appeared in the beginning of Section \ref{sec:3}, and furthermore we can readily check
\[
\gamma_- f(x,\xi,t) = \mathcal{B}(\gamma_+ f)(x,\xi,t) = \gamma_+ f(x,\mathcal{R}_x(\xi),t).
\]
This concludes the proof of Theorem \ref{T0.1}.

%
%
%
%
%

\section{Global well-posedness of the two-phase fluid system}\label{sec:4}
\setcounter{equation}{0}
In this section, we prove the global well-posedness of the system \eqref{A-3}. As mentioned before, for the rigorous hydrodynamic limit, it suffices to show the existence and uniqueness of strong solutions to the system \eqref{A-3} at least locally in time. We first obtain the local well-posedness theory for the system \eqref{A-3}, and then extend it to the global existence theory by means of the continuity argument.

For the existence theory, we use the structure of symmetric hyperbolic system for the compressible Navier-Stokes equations in \eqref{E-3}. We rewrite the system \eqref{E-3} as
\begin{align}\label{E-3-1}
\begin{aligned}
&\partial_t g + u \cdot \nabla g + \nabla \cdot u = 0, \quad (x,t) \in \om \times \R_+,\\
&\partial_t u + u \cdot \nabla u + \nabla g = v-u,\\
& A^0 (\eta) \partial_t \eta + \sum_{j=1}^3 A^j(\eta) \partial_j \eta = A^0(\eta) E_1(h,v) + A^0 (\eta) E_2(g,u,\eta),
\end{aligned}
\end{align}
where $\eta := (h,v)^T$,
\begin{align*}\begin{aligned}
&A^0(\eta) := \left(\begin{array}{cc} \gamma(1 +h)^{\gamma-2} & 0 \\ 0 & (1 +h)\mathbb{I}_{3 \times 3} \end{array}\right),\\[2mm]
&A^j (\eta) := A^0 (\eta) \left( \begin{array}{cccc} v_j & (1 +h)\delta_{1j} & (1 +h)\delta_{2j} & (1+h)\delta_{3j} \\
\gamma(1+h)^{\gamma-2}\delta_{1j} & v_j & 0 & 0 \\
\gamma (1+h)^{\gamma-2}\delta_{2j} & 0 & v_j & 0 \\
\gamma (1 +h)^{\gamma-2}\delta_{3j} & 0 & 0 & v_j \end{array} \right),\\[2mm]
&E_{1}(h,v) := \frac{1}{1 +h} \left( \begin{array}{c} 0 \\ \Delta v_1 \\ \Delta v_2 \\ \Delta v_3 \end{array} \right), \quad \mbox{and} \quad E_{2}(g,u,\eta) := \frac{e^g}{M(1+h)}\left( \begin{array}{c} 0 \\ u_1-v_1 \\ u_2-v_2 \\ u_3-v_3 \end{array} \right),
\end{aligned}\end{align*}
where $\delta_{ij}$ denotes the Kronecker delta function, i.e., $\delta_{ij} =0$ if $i \neq j$ and $\delta_{ij} =1$ if $i = j$. To make all of the estimates simpler, without loss of generality, we may assume that the constant $M = |\om| = 1$ in the rest of this section.

\subsection{Local-in-time existence theory} 
Let us define the solution space:
\[
\ms^s_T(\om) := \lt\{ (g,u,h,v) : (g,u,h,v) \in \mathfrak{X}^s(T,\om) \times \mathfrak{X}^s(T,\om) \times \mathfrak{X}^s(T,\om) \times \mathfrak{X}^s(T,\om) \rt\}.
\]

\begin{theorem}\label{thm_local}
There exist small constants $\e_0>0$ and $T^*>0$ such that if
\[
\|g_0\|_{H^s} + \|u_0\|_{H^s} + \|\eta_0\|_{H^s} <\e_0, 
\]
then a unique strong solution $(g,u,\eta) \in \mathcal{S}^s_{T^*}(\om)$ of system \eqref{E-3} in the sense of Definition \ref{D2.2} corresponding to initial data $(g_0, u_0, \eta_0)$ exists up to time $t\le T^*$ and 
\[
\sup_{0 \leq t \leq T^*}\left(\|g(t)\|_{\mathfrak{X}^s} + \|u(t)\|_{\mathfrak{X}^s} + \|\eta(t)\|_{\mathfrak{X}^s}\right) < \e_0^{1/2}.
\]
\end{theorem}

\begin{remark}In the above theorem, we need to have some smallness assumption on the initial data, and this implies that no matter how the small initial data are, the life-span of solutions is finite. This is due to the fact that we cannot use the integration by parts properly because of the kinematic boundary condition for $u$, and it seems that condition cannot be removed, see Remark \ref{rmk_local} for more detailed discussion.
\end{remark}

\subsubsection{Approximate solutions} In this subsection, we linearize the system \eqref{E-3-1} and provide uniform estimates and prove their convergence toward a strong solution to \eqref{E-3-1}. To be specific, we consider the sequence of approximate solutions as
\begin{align}\label{E-5}
\begin{aligned}
&\partial_t g^{m+1} + u^m \cdot \nabla g^{m+1} + \nabla \cdot u^{m+1} = 0, \quad (x,t) \in \om \times \R_+,\\
&\partial_t u^{m+1} + u^m \cdot \nabla u^{m+1} + \nabla g^{m+1} = v^{m+1}-u^{m+1},\\
&\displaystyle A^0 (\eta^m) \partial_t \eta^{m+1} + \sum_{j=1}^3 A^j(\eta^m) \partial_j \eta^{m+1} \\
&\quad= A^0(\eta^m) E_1(h^m,v^{m+1}) + A^0 (\eta^m) E_2(g^m,u^m,\eta^m),
\end{aligned}
\end{align}
subject to initial data, compatibility conditions, and boundary conditions:
\begin{align*}
\begin{aligned}
&(g^{m+1}(x,0), u^{m+1}(x,0), \eta^{m+1}(x,0)) = (g_0(x), u_0(x), \eta_0(x)), \quad x \in \Omega, \\
&\pa_t^k u^{m+1}(x,t) \cdot r(x) |_{t=0} = 0  \quad \mbox{and} \quad \pa_t^k v^{m+1}(x,t) |_{t=0} = 0 \quad (x,t)\in  \partial\Omega \times \R_+,  \quad k=0,1,\dots,s-1, \quad \mbox{and}\cr
& u^{m+1}(x,t) \cdot r(x) = 0  \quad \mbox{and} \quad v^{m+1}(x,t) = 0 \quad (x,t) \in  \partial\Omega \times \R_+,  
\end{aligned}
\end{align*}
for all $m \in \N$ and the first iteration step:
\[
(g^0(x,t),u^0(x,t),\eta^0(x,t)) = (g_0(x), u_0(x), \eta_0(x)), \quad (x,t) \in \om \times \R_+.
\]
For notational simplicity, we set
$$\begin{aligned}
&W^m(t) := \|g^m(t)\|_{\mathfrak{X}^s}^2 + \|u^m(t)\|_{\mathfrak{X}^s}^2 = \sum_{\ell=0}^s \|\partial_t^\ell g^m\|_{H^{s-\ell}}^2 + \|\partial_t^\ell u^m\|_{H^{s-\ell}}^2,\\
&T^m(t) := \sum_{\ell=0}^s \|\partial_t^\ell g^m\|_{L^2}^2 + \|\partial_t^\ell u^m\|_{L^2}^2, \quad \mbox{and}\\
&V^m(t) := \|\omega^m(t)\|_{\mathfrak{X}^{s-1}}^2 = \sum_{\ell=0}^{s-1}\|\partial_t^\ell \omega^m(t)\|_{H^{s-1-\ell}}^2,
\end{aligned}$$
where $\omega^m := \nabla \times u^m$. Here, the local existence of solutions to the Euler equations of the linearized system \eqref{E-5} can be handled by arguments from the previous literature \cite{Ag, Eb79, Sc86,Zh10}. Moreover, since the solvability of the Navier-Stokes system in \eqref{E-5} is quite classical, we only provide the upper bound estimates of approximate solutions and their convergence. 

\subsubsection{Uniform bound estimates}
In this part, we present uniform-in-$m$ bound estimates of the approximate solutions.

\begin{proposition}\label{P5.1} Let $s\ge 4$ and $T>0$. Let $(g^m,u^m,\eta^m) \in \ms^s_T(\om)$ be a sequence of strong solutions to \eqref{E-5}. Then, one has
\[
W^{m+1}(t) \le C\lt(T^{m+1}(t)+V^{m+1}(t)+W^m(t) W^{m+1}(t) + \|v^{m+1}(t)\|_{\mathfrak{X}^s}^2\rt), 
\]
where $C > 0$ depends on $s$ and $\Omega$, but independent of $m$.
\end{proposition}
\begin{proof} It follows from the equations for $u^{m+1}$ in \eqref{E-5} that
\begin{align*}\begin{aligned}
\|\nabla g^{m+1}\|_{L^2}^2 &= \|\partial_t u^{m+1} +u^m \cdot \nabla u^{m+1} + u^{m+1}-v^{m+1} \|_{L^2}^2\\
&\le C\lt(\|\partial_t u^{m+1}\|_{L^2}^2 +W^m(t) W^{m+1}(t) + \|u^{m+1}\|_{L^2}^2 + \|v^{m+1}\|_{L^2}^2\rt),
\end{aligned}\end{align*}
where we used Sobolev embedding $H^2(\om) \hookrightarrow \mc^0(\bar \om)$ and $C > 0$ is independent of $m$. Similarly, we also estimate from the equations for $g^{m+1}$ in \eqref{E-5} that 
\[
\|\nabla \cdot u^{m+1}\|_{L^2}^2 \le C(\|\partial_t g^{m+1}\|_{L^2}^2 + \|u^m \cdot \nabla g^{m+1}\|_{L^2}^2) \le C\lt(\|\partial_t g^{m+1}\|_{L^2}^2 + W^m(t) W^{m+1}(t)\rt),
\]
where $C > 0 $ is independent of $m$. Now, we recall from \cite[Lemma 2.2]{Zh10}, see also \cite[Lemma 5]{B-B} that for ${\bf u} \in H^s(\Omega)$ with $ {\bf u} \cdot r \equiv 0$ on $\partial\Omega$, $u$ satisfies
\[
\|{\bf u}\|_{H^s} \le C(\|\nabla \times {\bf u}\|_{H^{s-1}} + \|{\bf u}\|_{H^{s-1}} +\|\nabla \cdot {\bf u}\|_{H^{s-1}}).
\]
This yields
\begin{align*}\begin{aligned}
\|u^{m+1}\|_{H^1}^2 &\le C(\|\omega^{m+1}\|_{L^2}^2 + \|u^{m+1}\|_{L^2}^2 + \|\nabla \cdot u^{m+1}\|_{L^2}^2)\\
&\le C(\|\omega^{m+1}\|_{L^2}^2 + \|u^{m+1}\|_{L^2}^2 +\|\partial_t g^{m+1}\|_{L^2}^2 + W^m(t) W^{m+1}(t)),
\end{aligned}\end{align*}
where $C > 0$ depends on $s$ and $\Omega$, but independent of $m$. From now, we can inductively repeat the process to conclude the desired result.
\end{proof}

\begin{proposition}\label{P5.2} Let $s\ge 4$ and $T>0$. Let $(g^m,u^m,\eta^m) \in \ms^s_T(\om)$ be a sequence of strong solutions to \eqref{E-5}. Then we have
\[
\frac{d}{dt}T^{m+1}(t) \le C\lt((W^m)^{1/2}(t) W^{m+1}(t) + \|v^{m+1}(t)\|_{\mathfrak{X}^s}^2\rt),  
\]
where $C > 0$ depends on $s$ and $\Omega$, but independent of $m$.
\end{proposition}
\begin{proof}
For zeroth order, we obtain
\begin{align*}\begin{aligned}
&\frac{1}{2}\frac{d}{dt}\|g^{m+1}\|_{L^2}^2 = -\int_\Omega g^{m+1} \nabla g^{m+1} \cdot u^m \,dx - \int_\Omega (\nabla \cdot u^{m+1}) g^{m+1} \,dx,\\
&\frac{1}{2}\frac{d}{dt}\|u^{m+1}\|_{L^2}^2 = - \int_\Omega (u^m \cdot u^{m+1})\cdot u^{m+1} \,dx - \int_\Omega \nabla g^{m+1} \cdot u^{m+1} \,dx \\
&\hspace{2.8cm}+ \int_\Omega (v^{m+1}-u^{m+1})\cdot u^{m+1}\, dx.
\end{aligned}\end{align*}
Thus, we get
\begin{align*}\begin{aligned}
&\frac{1}{2}\frac{d}{dt}\left(\|g^{m+1}\|_{L^2}^2 + \|u^{m+1}\|_{L^2}^2\right) \cr
&\quad = \frac{1}{2}\int_\Omega (\nabla \cdot u^m) \left(\|g^{m+1}\|_{L^2}^2 + \|u^{m+1}\|_{L^2}^2\right) dx + \int_\Omega (v^{m+1}-u^{m+1})\cdot u^{m+1}\, dx\\
&\quad \le C(W^m)^{1/2}(t) W^{m+1}(t) + \frac{1}{2}\|v^{m+1}\|_{L^2}^2,
\end{aligned}\end{align*}
where $C > 0$ depends on $s$ and $\Omega$. Here we used the boundary condition for $u$, Sobolev embedding $H^2(\om) \hookrightarrow \mc^0(\bar \om)$ and Young's inequality.

For the high-order estimates, for $1 \le \ell \le s$, we find
\begin{align*}\begin{aligned}
\frac{1}{2}\frac{d}{dt}\|\partial_t^\ell g^{m+1}\|_{L^2}^2 &= -\sum_{r=0}^\ell\binom{\ell}{r} \int_\Omega \nabla (\partial_t^r g^{m+1}) \cdot \partial_t^{\ell-r} u^m \partial_t^\ell g^{m+1} \,dx \cr
&\quad -\int_\Omega \nabla \cdot (\partial_t^\ell u^{m+1}) \partial_t^\ell g^{m+1} \,dx \quad \mbox{and}\\
\frac{1}{2}\frac{d}{dt}\|\partial_t^\ell u^{m+1}\|_{L^2}^2 &= \sum_{r=0}^\ell \binom{\ell}{r}\int_\Omega (\partial_t^r u^m \cdot \nabla) \partial_t^{\ell-r} u^{m+1} \cdot \partial_t^\ell u^{m+1} \,dx \cr
&\quad - \int_\Omega \nabla(\partial_t^\ell g^{m+1}) \cdot \partial_t^\ell u^{m+1} \,dx \\
&\quad + \int_\Omega \partial_t^\ell (v^{m+1}-u^{m+1}) \cdot \partial_t^\ell u^{m+1} \,dx.
\end{aligned}\end{align*}
This yields
\[
\frac{1}{2}\frac{d}{dt}\left(\|\partial_t^\ell g^{m+1}\|_{L^2}^2 + \|\partial_t^\ell u^{m+1}\|_{L^2}^2 \right) \le C(W^m)^{1/2}(t) W^{m+1}(t) + \frac{1}{2}\|\partial_t^\ell v^{m+1}\|_{L^2}^2,
\]
where $C > 0$ depends on $s$ and $\Omega$ and we used the boundary condition for $u$ ($\partial_t^\ell u^{m+1} \cdot r \equiv 0$ on $\partial \Omega$), H\"older inequality, Sobolev embedding, and Young's inequality. Hence we have the desired result.
\end{proof}

\begin{proposition}\label{P5.3}
Let $s\ge 4$ and $T>0$. Let $(g^m,u^m,\eta^m) \in \ms^s_T(\om)$ be a sequence of strong solutions to \eqref{E-5}. Then we have
\[ 
\frac{d}{dt}V^{m+1}(t) \le C\lt((W^m)^{1/2}(t) W^{m+1}(t) + \|v^{m+1}(t)\|_{\mathfrak{X}^s}^2\rt),
\]
where $C > 0$ depends on $s$ and $\Omega$, but independent of $m$.
\end{proposition}

\begin{proof}
We apply the curl operator to the momentum equations of the Euler system in \eqref{E-5} to obtain the following vorticity equation:
\[
\partial_t \omega^{m+1} + \omega^m \cdot \nabla u^{m+1} + u^m \cdot \nabla \omega^{m+1} = (\nu^{m+1} - \omega^{m+1}),
\]
where $\nu^{m+1} := \nabla \times v^{m+1}$. For the zeroth order estimate, we readily get
\begin{align*}\begin{aligned}
\frac{1}{2}\frac{d}{dt}\|\omega^{m+1}\|_{L^2}^2 &= -\int_\Omega (\omega^m \cdot \nabla u^{m+1})\cdot \omega^{m+1} \,dx - \int_\Omega (u^m \cdot \nabla \omega^{m+1})\cdot \omega^{m+1} \,dx \\
&\quad + \int_\Omega (\nu^{m+1} - \omega^{m+1})\cdot \omega^{m+1} \,dx\\
&\le \|\nabla u^m\|_{L^\infty}\|\omega^{m+1}\|_{L^2}^2 + \frac{1}{2}\|\nu^{m+1}\|_{L^2}^2\\
&\le C(W^m)^{1/2}(t)W^{m+1}(t) + \frac{1}{2}\|\nu^{m+1}\|_{L^2}^2,
\end{aligned}\end{align*}
where $C > 0$ depends on $s$ and $\Omega$, and we used the Sobolev embedding $H^2(\om) \hookrightarrow \mc^0(\bar \om)$ and Young's inequality. For the higher order estimate, let $\nabla_{t,x}^\alpha$ be a mixed partial derivative with respect to time and space with multi-index $\alpha$ with $1 \le |\alpha| \le s-1$. Then, similarly as before, we obtain
\begin{align*}\begin{aligned}
\frac{1}{2}\frac{d}{dt}\|\nabla_{t,x}^\alpha \omega^{m+1}\|_{L^2}^2 &= - \sum_{\mu \le \alpha}\binom{\alpha}{\mu} \int_\Omega (\nabla_{t,x}^\mu \omega^m \cdot \nabla (\nabla_{t,x}^{\alpha-\mu} u^{m+1})) \cdot \nabla_{t,x}^\alpha \omega^{m+1} \,dx \\
&\quad - \sum_{\mu \le \alpha}\binom{\alpha}{\mu} \int_\Omega (\nabla_{t,x}^\mu u^m\cdot \nabla (\nabla_{t,x}^{\alpha-\mu} \omega^{m+1})) \cdot \nabla_{t,x}^\alpha \omega^{m+1} \,dx\\
& \quad + \int_\Omega \nabla_{t,x}^\alpha (\nu^{m+1} - \omega^{m+1}) \cdot \nabla_{t,x}^\alpha \omega^{m+1}\,dx\\
&\le C(W^m)^{1/2}(t) W^{m+1}(t)+ \frac{1}{2}\|\nabla_{t,x}^\alpha \nu^{m+1}\|_{L^2}^2,
\end{aligned}\end{align*} 
where $C > 0$ depends only on $s$ and $\Omega$. We then sum over all $\alpha$ with $1\le |\alpha|\le s-1$ and combine this with the zeroth order estimate to get the desired result. 
\end{proof}
Now, it remains to estimate the uniform bounds of solutions to the Navier-Stokes system in  \eqref{E-5}.
\begin{lemma}\label{P5.4}
Let $s\ge 4$ and $T>0$. Let $(g^m,u^m,\eta^m) \in \ms^s_T(\om)$ be a sequence of strong solutions to \eqref{E-5}. Suppose that 
\[
1 + \inf_{x \in \om} h_0(x) > \delta_0, 
\]
\[
\sup_{0 \leq t \leq T}\|\eta^m(t)\|_{\mathfrak{X}^s} \leq M_1, \quad \mbox{and} \quad \sup_{0 \leq t \leq T}\|(g^m(t),u^m(t))\|_{\mathfrak{X}^s} \leq M_1'
\]
for some $\delta_0 > 0$, $M_1$, and $M_1'$. Then there exists $T_0 > 0$, which is independent of $m$, such that 
\begin{itemize}
\item[(i)] 
\[
1 + \sup_{0 \leq t \leq T_0}\inf_{x \in \om} h^m(x,t) > \delta_0
\]
for all $m \in \N$,
\item[(ii)]
\begin{align*}\begin{aligned}
&\|\eta^{m+1}(t)\|_{\mathfrak{X}^s}^2 + \frac{1+\e_0}{2} \int_0^t \|\nabla v^{m+1}(\tau)\|_{\mathfrak{X}^s}^2 \,d\tau\\
&\quad \leq \lt(\gamma(1+\e_0)^2\|\eta_0\|_{H^s}^2 + e^{CM_1'}(M_1^2 + (M_1')^2 )T_0\rt)\exp\lt(C_{M_1,M_1'}T_0 \rt)
\end{aligned}\end{align*}
for $t \leq T_0$, where $C_{M_1,M_1'} > 0$ is 
\[
C_{M_1, M_1'} = C(M_1 + M_1^r + M_1^{2r} + e^{CM_1'} + (M_1')^2 + M_1^{2r} (M_1')^{2r'}).
\]
Here $C>0$ is independent of $m$, and $r,r'\geq1$ depend on $s$, but independent of $m$.
\end{itemize}
\end{lemma}
\begin{proof} Since the proofs are rather lengthy, we leave it in Appendix \ref{app_b}.
\end{proof}

\begin{remark}\label{rmk_local0} The lower bound estimate in Lemma \ref{P5.4} (i) can be also obtained by choosing $M_1 > 0$ small enough instead of taking the small $T_0 > 0$. This further implies that Lemma \ref{P5.4} (ii) also holds up to any time $T>0$ if we take the value $M_1 > 0$ small enough.
\end{remark}

\begin{remark}\label{rmk_local}Lemma \ref{P5.4} (ii) implies that for any $\sqrt\gamma(1+\e_0)\e_0 < M_1$, 
there exist $T^* > 0$ such that if $\|\eta_0\|_{H^s} < \e_0$, then 
\[
\sup_{0 \leq t \leq T^*}\|\eta^{m+1}(t)\|_{\mathfrak{X}^s} < M_1.
\]
This is the standard way of having the uniform bound estimates for the approximated solutions using only the smallness assumption on the time $t$. However, as mentioned before, we require the smallness assumption on the initial data as well due to the kinematic boundary condition for $u^m$.  
\end{remark}
In the proposition below, we provide the uniform bound estimates in $m$ locally in time.
\begin{proposition}\label{L5.1}Let $s\ge 4$ and $T>0$. Let $(g^m,u^m,\eta^m) \in \ms^s_T(\om)$ be a sequence of strong solutions to \eqref{E-5}. There exist small constants $\e_0>0$ and $T^*>0$ such that if
\[
\|g_0\|_{H^s} + \|u_0\|_{H^s} + \|\eta_0\|_{H^s} <\e_0, 
\]
we have
\[
\sup_{0 \leq t \leq T^*}\lt(\|g^m(t)\|_{\mathfrak{X}^s} + \|u^m(t)\|_{\mathfrak{X}^s} + \|\eta^m(t)\|_{\mathfrak{X}^s}\rt) < \e_0^{1/2}
\]
for all $m\in\N$.
\end{proposition}
\begin{proof} Since $\|h_0\|_{H^s} < \e_0$ and $\e_0$ is sufficiently small, we can make $\delta_0$ in Lemma \ref{P5.4} to have a value $1/2$ for an instance, i.e., Lemma \ref{P5.4} (i) holds with $\delta_0 = 1/2$. Suppose that 
\[
\sup_{0 \leq t \leq T_0} \lt(\|g^m(t)\|_{\mathfrak{X}^s} + \|u^m(t)\|_{\mathfrak{X}^s} + \|\eta^m(t)\|_{\mathfrak{X}^s}\rt) < \e_0^{1/2}.  
\]
Then, by choosing $M_1 = M_1' = \e_0^{1/2}$ in Lemma \ref{P5.4} (ii), we easily find
\[
\|\eta^{m+1}(t)\|_{\mathfrak{X}^s}^2 \leq \lt(\gamma(1+\e_0)^2\e_0^2 + e^{C\sqrt{\e_0}}\e_0 T_0  \rt)\exp\lt(C_{\sqrt{\e_0},\sqrt{\e_0}}T_0\rt).
\]
We now choose $\e_0 > 0$ and $(T_0 \geq )T^* >0 $ small enough to have
\[
\sup_{0 \leq t \leq T^*}\|\eta^{m+1}(t)\|_{\mathfrak{X}^s}^2 < \e_0.
\]
On the other hand, it follows from Propositions \ref{P5.2} and \ref{P5.3} that
\begin{align}\label{est_tv}
\begin{aligned}
T^{m+1}(t) + V^{m+1}(t) &\le T(0) + V(0) + C\e_0^{1/2}\int_0^t \left(W^{m+1}(\tau) + \|\eta^{m+1}(\tau)\|_{\mathfrak{X}^s}^2\right) d\tau\\
&< \e_0^2 +  C\e_0^{3/2} T_0 +  C\e_0^{1/2}\int_0^t W^{m+1}(\tau)\,d\tau
\end{aligned}
\end{align}
for $t \leq T^*$. From Proposition \ref{P5.1}, we also find
\begin{align*}\begin{aligned}
(1-C\e_0^{1/2})W^{m+1}(t) &\le C\lt(T^{m+1}(t) + V^{m+1}(t) + \|\eta^{m+1}(t)\|_{\mathfrak{X}^s}^2\rt)\cr
& \leq C\e_0^2 +  C\e_0^{3/2} T^* +  C\e_0^{1/2}\int_0^t W^{m+1}(\tau)\,d\tau.
\end{aligned}\end{align*}
This together with \eqref{est_tv} gives
\[
W^{m+1}(t) \leq C\lt(\e_0^2 +  \e_0^{3/2} T^*\rt)\exp\lt(C\e_0^{1/2}T^* \rt)
\]
for $t \leq T^*$, where $C>0$ is independent of $m$. We finally choose $\e_0 > 0$ small enough so that the right hand side of the above inequality is less than $\e_0$. By the induction argument, we conclude the desired result.
\end{proof}

\subsubsection{Cauchy estimates}
In this part, we show the approximate solutions $\{(g^{m}, u^m, \eta^m) \}_{m \in \N}$ to the system \eqref{E-5} are Cauchy in $L^2$-space. 
\begin{lemma}\label{L5.2}
For $\e_0$ and $T^*$ chosen as in Proposition \ref{L5.1}, if 
\[
\|g_0\|_{H^s} + \|u_0\|_{H^s} + \|\eta_0\|_{H^s} <\e_0, 
\]
then we have
\[
\sup_{0 \le t \le T^*}\left(\|(g^{m+1} - g^m)(t)\|_{L^2}^2 + \|(u^{m+1} - u^m)(t)\|_{L^2}^2 + \|(\eta^{m+1} - \eta^m)(t)\|_{L^2}^2\right) \to 0
\]
as $m \to \infty$.
\end{lemma}
\begin{proof}The proof is divided into two steps:
\begin{itemize}
\item In (Step A), we provide the following Cauchy estimate for $(g^{m+1}, u^{m+1})$ which solves the Euler equations in \eqref{E-5}:
\begin{align}\label{cau_gu}
\begin{aligned}
&\|(g^{m+1} - g^m)(t)\|_{L^2}^2 + \|(u^{m+1} - u^m)(t)\|_{L^2}^2\\
&\quad \le C\int_0^t \|(v^{m+1} - v^m)(\tau)\|_{L^2}^2 +\|(u^m - u^{m-1})(\tau)\|_{L^2}^2\,d\tau,
\end{aligned}
\end{align}
where $C>0$ is independent of $m$.
\item In (Step B), we estimate the solutions $\eta^{m+1}$ to the Navier-Stokes equations in \eqref{E-5} and combine that with \eqref{cau_gu} to conclude the desired result. \newline
\end{itemize}

\noindent $\bullet$ (Step A) For notational simplicity, we set
\[
g^{m+1, m} = g^{m+1} - g^m, \quad u^{m+1, m} = u^{m+1} - u^m, \quad \mbox{and} \quad v^{m+1, m} = v^{m+1} - v^m
\]
for $m \in \N \cup \{0\}$. Then a straightforward computation yields 
\begin{align*}\begin{aligned}
\frac{1}{2}\frac{d}{dt}\|g^{m+1, m}\|_{L^2}^2 &= -\int_\Omega (u^m \cdot \nabla g^{m+1} - u^{m-1}\cdot \nabla g^m) (g^{m+1} - g^m)\,dx\\
&\quad - \int_\Omega \nabla \cdot (u^{m+1} - u^m) (g^{m+1} - g^m)\,dx\\
&= -\int_\Omega u^m \cdot \nabla(g^{m+1}-g^m)(g^{m+1} - g^m)\,dx\\
&\quad -\int_\Omega (u^m - u^{m-1})\cdot \nabla g^m (g^{m+1} - g^m)\,dx\\
&\quad - \int_\Omega \nabla \cdot (u^{m+1} - u^m) (g^{m+1} - g^m)\,dx\\
&\le C \left( \|g^{m+1, m}\|_{L^2}^2  + \|u^{m,m-1}\|_{L^2}\|g^{m+1, m}\|_{L^2}\right)\\
&\quad - \int_\Omega \nabla \cdot (u^{m+1} - u^m) (g^{m+1} - g^m)\,dx,
\end{aligned}\end{align*}
where $C>0$ is a constant independent of $m$ and we used Cauchy-Schwarz inequality. Similarly, we also find 
\begin{align*}\begin{aligned}
\frac{1}{2}\frac{d}{dt}\|u^{m+1,m}\|_{L^2}^2 &= -\int_\Omega (u^m \cdot u^{m+1} - u^{m-1} \cdot \nabla u^m) \cdot (u^{m+1}-u^m)\,dx\\
&\quad - \int_\Omega \nabla(g^{m+1} - g^m) \cdot (u^{m+1} - u^m)\,dx\\
&\quad + \int_\Omega (v^{m+1} - u^{m+1} -v^m + u^m)(u^{m+1} - u^m)\,dx\\
&\le C\|u^{m+1, m}\|_{L^2}^2 + C\|u^{m+1, m}\|_{L^2}\|u^{m,m-1}\|_{L^2}\\
&\qquad + C\|u^{m+1, m}\|_{L^2}\| v^{m+1, m}\|_{L^2} \\
&\quad - \int_\Omega \nabla(g^{m+1} - g^m) \cdot (u^{m+1} - u^m)\,dx,
\end{aligned}\end{align*}
where $C>0$ is a constant independent of $m$ and we used Cauchy-Schwarz inequality. Thus, we use Young's inequality and the kinematic boundary condition for $u^m$ to get
\begin{align*}\begin{aligned}
\frac{d}{dt}&\left(\|g^{m+1, m}\|_{L^2}^2 + \|g^{m+1, m}\|_{L^2}^2\right)\\
&\le C\left(\|g^{m+1, m}\|_{L^2}^2 + \|u^{m+1, m}\|_{L^2}^2 + \| v^{m+1, m}\|_{L^2}^2 +\|u^{m,m-1}\|_{L^2}^2\right).
\end{aligned}\end{align*}
We finally apply Gr\"onwall's lemma to the above together with using $g^{m+1, m}(0) = 0$ and $g^{m+1, m} = 0$ to complete the proof of Step A. \newline

\noindent $\bullet$ (Step B) Similarly as before, we first introduce simplified notations as follows:
\[
\eta^{m+1, m} := \eta^{m+1} - \eta^m, \quad A_m^0 := A^0(\eta^m), \quad A_m^j:=A^j(\eta^m), \quad E_1^m := E_1(h^m, v^{m+1}), 
\]
and
\[
E_2^m := E_2(g^m, u^m, \eta^m).
\]
Then, it follows from the Navier-Stokes equations in \eqref{E-5} that 
\begin{align*}\begin{aligned}
&A_m^0\partial_t \eta^{m+1,m} + \sum_{j=1}^3 A_m^j \partial_j\eta^{m+1, m}\\
&\quad = - (A_m^0- A_{m-1}^0)\partial_t \eta^m - \sum_{j=1}^3(A_m^j - A_{m-1}^j)\partial_j \eta^m + (A_m^0E_1^m - A_{m-1}^0 E_1^{m-1})\cr
&\qquad  + A_m^0 (E_2^m - E_2^{m-1}) +( A_m^0-A_{m-1}^0) E_2^{m-1}.
\end{aligned}\end{align*}
Thus,  we obtain
$$\begin{aligned}
&\frac{1}{2}\left( \partial_t \int_{\Omega} A_m^0 \eta^{m+1,m} \cdot \eta^{m+1,m} \,dx - \int_\Omega (\partial_t A^0_m) \eta^{m+1,m} \cdot \eta^{m+1,m} \,dx \right)\\
&\hspace{0.3cm}\qquad+ \frac{1}{2}\left\{ \sum_{j=1}^3 \int_{\Omega} \partial_j (A^j_m \eta^{m+1,m} \cdot \eta^{m+1,m})\, dx - \int_{\Omega} (\partial_j A^j_m) \eta^{m+1,m} \cdot \eta^{m+1,m} \,dx \right\}\\
&\hspace{0.3cm}\quad = - \int_{\Omega} (A^0_m - A^0_{m-1}) \partial_t \eta^m \cdot \eta^{m+1,m} \,dx - \sum_{j=1}^3 \int_{\Omega}(A^j_m - A^j_{m-1})\partial_j \eta^m \cdot \eta^{m+1,m} \,dx\\
&\hspace{0.3cm}\qquad + \int_{\Omega} (A^0_m E_1^m - A^0_{m-1}E_1^{m-1}) \eta^{m+1,m} \, dx + \int_{\Omega} A^0_m (E_2^m - E_2^{m-1}) \eta^{m+1,m}\,dx\\
&\hspace{0.3cm}\qquad + \int_{\Omega}(A^0_m - A^0_{m-1})E_2^{m-1} \eta^{m+1,m} \,dx.
\end{aligned}$$
Note that
\[
|A^0_m - A^0_{m-1}| + |A^j_m - A^j_{m-1}| \le C |\eta^{m,m-1}|
\]
and
\[
\int_{\Omega} (A^0_m E_1^m - A^0_{m-1}E_1^{m-1}) \eta^{m+1,m}\, dx = -\int_{\Omega} |\nabla v^{m+1,m}|^2 \,dx.
\]
We also estimate
\begin{align*}\begin{aligned}
\begin{aligned}
&\int_{\Omega} A^0_m (E_2^m - E_2^{m-1}) \eta^{m+1,m}\,dx \cr
& \quad \le C \|E_2^m - E_2^{m-1}\|_{L^2} \|\eta^{m+1,m}\|_{L^2} \\
&\quad \le C \left(\|g^{m,m-1}\|_{L^2} + \|u^{m,m-1}\|_{L^2} + \|\eta^{m,m-1}\|_{L^2} \right) \|\eta^{m+1,m}\|_{L^2}.
\end{aligned}
\end{aligned}\end{align*}
We combine all the estimates and use Young's inequality to have
\begin{align*}\begin{aligned}
\begin{aligned}
&\|\eta^{m+1,m}(t)\|_{L^2}^2 + \int_0^t \|\nabla v^{m+1,m}(\tau)\|_{L^2}^2\, d\tau\\
&\quad \le C \int_0^t \|\eta^{m+1,m}(\tau)\|_{L^2}^2 \,d\tau  \\
&\qquad +  C\int_0^t \Big(\|g^{m,m-1}(\tau)\|_{L^2} + \|u^{m,m-1}(\tau)\|_{L^2} + \|\eta^{m,m-1}(\tau)\|_{L^2} \Big) \|\eta^{m+1,m}(\tau)\|_{L^2} \,d \tau\\
& \quad \le C \int_0^t \|\eta^{m+1,m}(\tau)\|_{L^2}^2 \,d\tau +C \int_0^t \|\eta^{m,m-1}(\tau)\|_{L^2}^2 \,d\tau \\
&\qquad + C\int_0^t\|g^{m,m-1}(\tau)\|_{L^2}^2 \,d\tau + C\int_0^t \|u^{m,m-1}(\tau)\|_{L^2}^2 \,d\tau,
\end{aligned}
\end{aligned}\end{align*}
and subsequently, applying Gr\"onwall's lemma to the above yields
\begin{align*}\begin{aligned}
&\|\eta^{m+1,m}(t)\|_{L^2}^2 + \int_0^t \|\nabla v^{m+1,m}(\tau)\|_{L^2}^2 \,d\tau\cr
&\quad \leq C \int_0^t \|\eta^{m,m-1}(\tau)\|_{L^2}^2 \,d\tau + C\int_0^t\|g^{m,m-1}(\tau)\|_{L^2}^2 \,d\tau + C\int_0^t \|u^{m,m-1}(\tau)\|_{L^2}^2 \,d\tau,
\end{aligned}\end{align*}
where $C>0$ is independent of $m$. Now we combine this with \eqref{cau_gu} to have
\begin{align*}\begin{aligned}
&\|g^{m+1,m}(t)\|_{L^2}^2 + \|u^{m+1,m}(t)\|_{L^2}^2 + \|\eta^{m+1,m}(t)\|_{L^2}^2\\
&\quad \leq C \int_0^t \lt( \|g^{m,m-1}(\tau)\|_{L^2}^2 +  \|u^{m,m-1}(\tau)\|_{L^2}^2  + \|\eta^{m,m-1}(\tau)\|_{L^2}^2 \rt)d\tau.
\end{aligned}\end{align*}
This asserts
\[
\|g^{m+1,m}(t)\|_{L^2}^2 + \|u^{m+1,m}(t)\|_{L^2}^2 + \|\eta^{m+1,m}(t)\|_{L^2}^2 \leq \frac{(CT^*)^{m+1}}{(m+1)!},
\]
and this concludes the desired result.
\end{proof}

\subsubsection{Proof of Theorem \ref{thm_local}: existence and uniqueness of strong solutions} We now provide the details on the local-in-time existence result. 

By the Cauchy estimate in Lemma \ref{L5.2}, we can obtain limiting functions $(g,u,h,v)$ in $\mathcal{C}([0,T^*];L^2(\Omega))$ and due to this strong convergence, we easily check that $(g,u,h,v)$ satisfies \eqref{E-3} in the sense of distributions. For regularity of solutions and upper bound estimates, similarly as in \cite{CK16}, we deduce those results from the uniform bound in Proposition \ref{L5.1}, which implies the existence of weak limit $(\tilde{g}, \tilde{u}, \tilde h,\tilde v)$ in $\mathfrak{X}^s(T^*,\Omega)$. Thanks to the strong convergence, essentially we obtain $(g,u,h,v)=(\tilde{g}, \tilde{u}, \tilde h,\tilde v)$. For the uniqueness, if $(g,u,h,v)$ and $(\bar g,\bar u,\bar h,\bar v)$ are two strong solutions to the system \eqref{E-3}, then we directly use the Cauchy estimate in Lemma \ref{L5.2} to deduce
\[
\Delta(t) \le C\int_0^t \Delta(\tau)\,d\tau
\]
for $t \in [0,T^*]$, where 
\[
\Delta(t) := \|(g-\bar g)(t)\|_{L^2}^2 + \|(u-\bar u)(t)\|_{L^2}^2 + \|(h-\bar h)(t)\|_{L^2}^2+ \|(v-\bar v)(t)\|_{L^2}^2.
\]
Since $\Delta(0) = 0$, applying Gr\"onwall's lemma asserts $(g,u,h,v) \equiv (\bar g,\bar u,\bar h,\bar v)$ in $\mathcal{C}([0,T^*];L^2(\Omega))$. Moreover, by using the similar argument as in the proof of existence of solutions, we can show that they are the same in our desired solution space $\mathfrak{X}^s(T^*,\Omega)$.

\subsection{Global-in-time existence theory} In this part, we provide a priori estimates for strong solutions to \eqref{E-3}--\eqref{E-4} obtained in the previous subsection. Combined with the local-in-time existence result, this enables us to extend the life span of the strong solution to the system \eqref{E-3}--\eqref{E-4} up to any time $T > 0$. For this, we first set the following functionals similar to the previous subsection:
\begin{align*}\begin{aligned}
&W(t) := \|g(t)\|_{\mathfrak{X}^s}^2 + \|u(t)\|_{\mathfrak{X}^s}^2 = \sum_{\ell=0}^s \|\partial_t^\ell g\|_{H^{s-\ell}}^2 + \|\partial_t^\ell u\|_{H^{s-\ell}}^2,\\
&T(t) := \sum_{\ell=0}^s \|\partial_t^\ell g\|_{L^2}^2 + \|\partial_t^\ell u\|_{L^2}^2,\quad \mbox{and} \quad V(t) := \|\omega(t)\|_{\mathfrak{X}^{s-1}}^2 = \sum_{\ell=0}^{s-1}\|\partial_t^\ell \omega(t)\|_{H^{s-1-\ell}}^2,
\end{aligned}\end{align*}
where $\omega := \nabla \times u$. 

In the proposition below, we show the upper bound estimate for the functional $W$.
\begin{proposition}\label{P5.5}
Let $T>0$ and $(g,u,\eta)\in \ms^s_T(\om)$ be a strong solution to \eqref{E-3-1} in the sense of Definition \ref{D2.2}. Then we have
\[
W(t) \le C\lt(\|g_0\|_{H^s}^2 + \|u_0\|_{H^s}^2\rt) + C\lt(W^2(t) + \|v(t)\|_{\mathfrak{X}^s}^2\rt)  + C\int_0^t W^{3/2}(\tau) + \|v(\tau)\|_{\mathfrak{X}^s}^2\,d\tau 
\]
for $t \in (0,T)$, where $C>0$ only depends on $s$ and $\Omega$.
\end{proposition}
\begin{proof}It directly follows from Propositions \ref{P5.2} and \ref{P5.3} that
\[
\frac{d}{dt}(T(t) + V(t)) \le C\lt(W^{3/2}(t) + \|v(t)\|_{\mathfrak{X}^s}^2\rt) 
\]
for $t \in (0,T)$, where $C>0$ only depends on $s$ and $\Omega$. This gives
\bq\label{est_tv2}
T(t) + V(t) \leq \|g_0\|_{H^s}^2 + \|u_0\|_{H^s}^2 + C\int_0^t W^{3/2}(\tau) + \|v(\tau)\|_{\mathfrak{X}^s}^2\,d\tau.
\eq
Similarly, we obtain from Proposition \ref{P5.1} that
\[
W(t) \le C\lt(T(t)+V(t)+W^2(t) + \|v(t)\|_{\mathfrak{X}^s}^2\rt). 
\]
Combining this with \eqref{est_tv2} asserts
\begin{align*}\begin{aligned}
W(t) &\le C\lt(\|g_0\|_{H^s}^2 + \|u_0\|_{H^s}^2\rt) + C W^2(t) + C\|v(t)\|_{\mathfrak{X}^s}^2 \cr
&\quad + C\int_0^t W^{3/2}(\tau) + \|v(\tau)\|_{\mathfrak{X}^s}^2\,d\tau
\end{aligned}\end{align*}
for $t \in (0,T)$, where $C>0$ only depends on $s$ and $\Omega$. This completes the proof.
\end{proof}

We now state the details on the proof of our main result
\begin{proof}[Proof of Theorem \ref{T2.3}]
First we define the life-span of solutions to the system \eqref{E-3-1}:
\[
\mathcal{S}:= \sup\lt\{ t\geq 0\,:\, \sup_{0 \leq \tau \leq t}  \lt(\|g(\tau)\|_{\mathfrak{X}^s} + \|u(\tau)\|_{\mathfrak{X}^s} + \|h(\tau)\|_{\mathfrak{X}^s} + \|v(\tau)\|_{\mathfrak{X}^s}\rt)< \sqrt\e \rt\}.
\]
Since $0 \in \ms \neq \emptyset$, we can define $\widetilde T = \sup \ms$. Suppose that $\widetilde T < T$, i.e., 
\[
\sup_{0 \leq \tau \leq \widetilde T}  \lt(\|g(\tau)\|_{\mathfrak{X}^s} + \|u(\tau)\|_{\mathfrak{X}^s} + \|h(\tau)\|_{\mathfrak{X}^s} + \|v(\tau)\|_{\mathfrak{X}^s}\rt) = \sqrt\e.
\]
Similarly as in the proof of Lemma \ref{P5.4}, see also Remark \ref{rmk_local0}, we find that for sufficiently small $\e>0$
\begin{align*}\begin{aligned}
&\|\eta(t)\|_{L^2}^2 + (1+\e)\int_0^t\|\nabla v(\tau)\|_{L^2}^2\,d\tau\cr
&\quad \leq \gamma(1 + \e)^2\|\eta_0\|_{L^2}^2 + C(1+\e)e^{C\sqrt\e}\int_0^t \|u(\tau)\|_{L^2}^2 + \|\eta(\tau)\|_{L^2}^2\,d\tau
\end{aligned}\end{align*}
and
\begin{align*}\begin{aligned}
&\|\nabla_{t,x}^\alpha \eta(t)\|_{L^2}^2  + \frac{1+\e}{2}\int_0^t \|\nabla \nabla_{t,x}^\alpha v(\tau)\|_{L^2}^2\, d\tau\\
&\quad \leq \gamma(1+\e)^2\|\nabla_{t,x}^\alpha\eta_0\|_{L^2}^2 + e^{C\sqrt\e}\int_0^t \lt(\|u(\tau)\|_{\mathfrak{X}^s}^2 + \|v(\tau)\|_{\mathfrak{X}^s}^2 \rt) d\tau + C_{\sqrt\e, \sqrt\e} \int_0^t\|\eta(\tau)\|_{\mathfrak{X}^s}^2\,d\tau
\end{aligned}\end{align*}
for $1 \leq |\alpha| \leq s$, where $C_{\sqrt\e, \sqrt\e} > 0$ is given similarly as in Lemma \ref{P5.4} (ii). Combining all the above estimates yields
\bq\label{est_eta}
\|\eta(t)\|_{\mathfrak{X}^s}^2 \leq \gamma(1 + \e)^2\|\eta_0\|_{H^s}^2 + e^{C\sqrt\e}\int_0^t \lt(\|u(\tau)\|_{\mathfrak{X}^s}^2 + \|\eta(\tau)\|_{\mathfrak{X}^s}^2 \rt) d\tau.
\eq
We next combine this with Proposition \ref{P5.5} to get
\begin{align}\label{est_W}
\begin{aligned}
W(t) &\le C\e^2 + C\e W(t) + C\e^{1/2}\int_0^t  W(\tau) \,d\tau + C \|v(t)\|_{\mathfrak{X}^s}^2 + C \int_0^t  \|v(\tau)\|_{\mathfrak{X}^s}^2\,d\tau\cr
&\leq C\e^2 + C\e W(t) + C\int_0^t  W(\tau) \,d\tau + C \int_0^t  \|\eta(\tau)\|_{\mathfrak{X}^s}^2\,d\tau
\end{aligned}
\end{align}
for $t \leq \widetilde T$. We then again combine \eqref{est_eta} and \eqref{est_W} to obtain
\begin{align*}\begin{aligned}
\|g(t)\|_{\mathfrak{X}^s}^2 + \|u(t)\|_{\mathfrak{X}^s}^2 +\|\eta(t)\|_{\mathfrak{X}^s}^2 &\leq C\e^2 + C\int_0^t \lt(\|g(\tau)\|_{\mathfrak{X}^s}^2 + \|u(\tau)\|_{\mathfrak{X}^s}^2 +\|\eta(\tau)\|_{\mathfrak{X}^s}^2\rt)d\tau.
\end{aligned}\end{align*}
Hence, by Gr\"onwall's lemma we have
\[
\|g(t)\|_{\mathfrak{X}^s}^2 + \|u(t)\|_{\mathfrak{X}^s}^2 +\|\eta(t)\|_{\mathfrak{X}^s}^2 \leq C\e^2 e^{C\tilde T},
\]
that is,
\[
\|g(t)\|_{\mathfrak{X}^s} + \|u(t)\|_{\mathfrak{X}^s} +\|h(t)\|_{\mathfrak{X}^s} + \|v(t)\|_{\mathfrak{X}^s} \leq C\e e^{C\tilde T}
\]
for $t \leq \widetilde T$. On the other hand, by choosing $\e > 0$ small enough we can make the right hand side of the above inequality less than $\sqrt\e/2$, and this leads to the following contradiction:
\[
\sqrt\e = \sup_{0 \leq t \leq \widetilde T}\lt(\|g(t)\|_{\mathfrak{X}^s} + \|u(t)\|_{\mathfrak{X}^s} +\|h(t)\|_{\mathfrak{X}^s} + \|v(t)\|_{\mathfrak{X}^s}\rt) \leq \frac{\sqrt\e}{2}.
\]
This concludes $\sup \ms = T$ and completes the proof.
\end{proof}

%
%
%
%
%
%

\section*{Acknowledgments}
The work of Y.-P. Choi is supported by NRF grant (No. 2017R1C1B2012918 and 2017R1A4A1014735), POSCO Science Fellowship of POSCO TJ Park Foundation,  and Yonsei University Research Fund of 2019-22-021. The work of J. Jung is supported by NRF grant (No. 2016K2A9A2A13003815 and 2019R1A6A1A10073437).

\section*{Conflict of interest}
The authors declare that they have no conflict of interest.

\appendix

\section{Global-in-time weak solutions of Dirichlet boundary value problem}
\setcounter{equation}{0}

In this appendix, we provide the details of Proof of Theorem \ref{T3.1}. More precisely, we discuss the global-in-time existence of weak solutions to the following system:
\begin{align}
\begin{aligned}\label{main_diri}
&\partial_t f + \xi \cdot \nabla f + \nabla_\xi \cdot ((v -\xi)f) =\Delta_\xi f - \nabla_\xi \cdot \left( \left(u_\e - \xi \right) f \right), \quad (x,\xi,t) \in \om \times \R^3 \times \R_+, \\
&\partial_t n + \nabla \cdot (n v)=0, \quad (x,t) \in \om \times \R_+,\\
&\partial_t (n v) + \nabla \cdot (n v \otimes v) + \nabla p  -\Delta v = -\rho(u-v)
\end{aligned}
\end{align}
with the nonhomogeneous/homogeneous Dirichlet boundary conditions for $f$ and $v$: 
\[
 \gamma_- f(x,\xi,t) = g(x,\xi,t) \quad \mbox{for} \quad  (x, \xi,t) \in \Sigma_- \times \R_+
\]
and
\[ 
v(x,t) = 0 \quad \mbox{for} \quad (x,t) \in \pa \om \times \R_+,
\]
respectively. Here $u_\e$ is defined by
\[
u_\e = \frac{\rho}{\rho + \e}u.
\]
Let $d_0 > 5$, and we temporarily assume that the initial data $f_0$ and $g$ satisfy
\begin{align}
&\label{init-2} \int_{\Omega \times \bbr^3} |\xi|^d f_0(x,\xi) \,dx d\xi <\infty \quad \mbox{and}\\ 
&\label{DC-2}  \int_0^T\int_{\Sigma_-} |\xi|^d g(x,\xi,s) |\xi \cdot r(x)| \,d\sigma(x) d\xi ds < \infty, 
\end{align}
respectively, for all $d \in [0,d_0]$. In order to regularize the system \eqref{main_diri}, we introduce the truncation function $\chi_\lambda$ and the standard mollifier $h_k$ given by
\[
\chi_\lambda (v) := v \mathds{1}_{\{|v|\le \lambda\}} \quad 
\mbox{and} \quad h_k(x) := \frac{1}{k^3}h\lt(\frac{x}{k}\rt), 
\]
respectively, where $h : \om \to \R$ satisfies
\[
0 \leq h \in \mc^\infty_c(\om) \quad \mbox{and} \quad \int_\om h(x)\,dx = 1.
\]
Using these functions, we regularize the convection term, drag and the local alignment forces in the system \eqref{main_diri} as 
\begin{align}
\begin{aligned}\label{App-1}
&\partial_t f + \xi \cdot \nabla f + \nabla_\xi \cdot ((\chi_\lambda(v) -\xi)f) =\Delta_\xi f - \nabla_\xi \cdot \left( \left(\chi_\lambda\left(u_\e\right) - \xi \right) f \right),\\
&\partial_t n + \nabla \cdot (n v)=0,\\
&\partial_t (n_k v) + \nabla\cdot ((n v)_k \otimes v) + \nabla p  -\Delta v = -\rho(u-v)\mathds{1}_{\{|v|\le \lambda\}}.
\end{aligned}
\end{align}
Here we denoted by the subscript $k$ in the convection term the convolution with the mollifier $h_k(x)$ with respect to $x$, i.e., 
\[
(n v)_k(x) = ((nv) \star h_k)(x) = \int_\om (nv)(x) h_k(x-y)\,dy.
\]
To show the existence of solutions to the regularized system \eqref{App-1}, we decouple the system as
\begin{align}
\begin{aligned}\label{App-2}
&\partial_t f + \xi \cdot \nabla f + \nabla_\xi \cdot ((\chi_\lambda(\tilde v) -\xi)f) =\Delta_\xi f - \nabla_\xi \cdot \left( \left( \chi_\lambda(\tilde{u}) - \xi \right) f \right),\\
&\partial_t n + \nabla \cdot (n v)=0,\\
&\partial_t (n_k v) + \nabla \cdot ((n v)_k \otimes v) + \nabla p  -\Delta v = -\rho(u-\tilde v)\mathds{1}_{\{|\tilde v|\le \lambda\}}, \quad x \in \Omega,
\end{aligned}
\end{align}
where $\tilde u$ and $\tilde v$ are given functions in $L^2(\om \times (0,T))$. We now consider the operator 
\[
\mathcal{T} : L^2(\om \times (0,T)) \times L^2(\om \times (0,T)) \to L^2(\om \times (0,T)) \times L^2(\om \times (0,T))
\]
defined by
\[ 
\mathcal{T} ( \tilde{u}, \tilde{v}) = (u_\e, v),
\]
For notational simplicity, set denote by 
\[
\mathcal{S}:= L^2(\om \times (0,T)) \times L^2(\om \times (0,T)).
\]
In order to show the existence of solutions to the system \eqref{App-1}, we use the Schauder's fixed point theorem to show the existence of a fixed point of $\mathcal{T}$.

\begin{theorem}\label{TA.1}
The operator $\mathcal{T} : \mathcal{S} \to \mathcal{S}$ is well-defined, continuous and compact. Thus the operator $\mathcal{T}$ has a fixed point, and subsequently this asserts that there exists at least one weak solution $(f,n,v)$ to the system \eqref{App-1}.
\end{theorem}
In the following two subsections, we separately prove each property of $\mathcal{T}$.

\subsection{$\mathcal{T}$ is well-defined} We notice that
\[ 
\chi_\lambda(\tilde v) , \   \chi_\lambda (\tilde u) \in L^\infty(\Omega \times (0,T)). 
\]
Inspired by previous literature \cite{C, K-M-T, M-V}, we provide the following proposition.
\begin{proposition}\label{PA.1} Let $T>0$ and assume that the initial data $f_0$ and $g$ in the boundary condition satisfy
\[
f_0 \in (L_+^1\cap L^\infty)(\Omega \times \bbr^3) 
\]
and
\begin{align*}\begin{aligned}
\begin{aligned}
g \in (L_+^1\cap L^\infty)(\Sigma_- \times (0,T)), \quad \int_0^T \int_{\Sigma_-} |\xi|^2 g(x,\xi) |\xi \cdot r(x)|\, d\sigma(x) d\xi dt <\infty,
\end{aligned}
\end{aligned}\end{align*}
respectively. Then there exists a unique solution $f$ of the kinetic equation in \eqref{App-2} in the sense of distributions and satisfies the following integrability conditions:
\begin{align*}\begin{aligned}
&f \in L^\infty(0,T; (L_+^1\cap L^\infty)(\Omega \times \bbr^3)), \quad \nabla_\xi f \in L^\infty(0,T; L^2(\Omega \times \bbr^3)),\\
&|\xi|^2 f \in L^\infty(0,T; L^1(\Omega \times \bbr^3)), \quad \mbox{and} \quad \gamma_+ f \in L^2(\Sigma_+ \times (0,T)).
\end{aligned}\end{align*}
Moreover, the following estimates hold:
\begin{itemize}
\item[(i)]
For any $p \in [1,\infty]$, $L^p$-norm of $f$ and $\gamma_+ f$ can be uniformly bounded in $k$, $\lambda$ and $\e$:
\begin{align}
\begin{aligned}\label{App-3}
& \|f\|_{L^\infty(0,T;L^p(\Omega \times \bbr^3))} + \|\nabla_\xi f^{\frac{p}{2}}\|_{L^2(\Omega \times \bbr^3 \times (0,T))}^{\frac{2}{p}} \le e^{\frac{C}{p'}T}\hspace{-0.05cm}\lt( \|f_0\|_{L^p(\Omega \times \bbr^3)}  + \|g\|_{L^p(\Sigma_-  \times (0,T))} \rt) \quad \mbox{and}\\
&\qquad \|\gamma_+ f\|_{L^p(\Sigma_+ \times (0,T))} \le \|f_0\|_{L^p(\Omega \times \bbr^3)} + \|g\|_{L^p( \Sigma_- \times (0,T))}.
\end{aligned}
\end{align}
\item[(ii)]
If one has
\[
\|f\|_{L^\infty( \Omega \times \bbr^3 \times (0,T))}<M 
\]
and
\[
\int_{\Omega \times \bbr^3} |\xi|^m f(x,\xi,t)\,dxd\xi \le M, \quad \forall t \in (0,T), \quad \forall m \in [0,d_0],  
\]
then there exists a constant $C := C(M)$ such that for all $t \in [0,T]$,
\begin{align*}\begin{aligned}
&\|\rho(\cdot,t)\|_{L^p} \le C(M), \quad \forall p \in [1, (d_0+3)/3) \quad \mbox{and}\\
&\|(\rho u)(\cdot,t)\|_{L^p} \le C(M), \quad \forall p \in [1, (d_0+3)/4).
\end{aligned}\end{align*}
\end{itemize}
\end{proposition}
\begin{proof}
(i) Consider a weak solution of the following Vlasov-Poisson-Fokker-Planck equation:
\begin{equation}\label{App-4}
\partial_t f + \xi \cdot \nabla f + \nabla_\xi \cdot ((E_0-\beta \xi) f) = \Delta_\xi f
\end{equation}
with $E_0 \in L^\infty(\om \times (0,T))$. Then, by \cite[Lemma 3.4]{C}, we obtain
\begin{align*}\begin{aligned}
\frac{d}{dt}& \int_{\Omega \times \bbr^3} f^p \,dxd\xi + \int_\Sigma (\xi \cdot r(x)) (\gamma f)^p \,d\sigma(x) d\xi \\
& \quad -3\beta (p-1) \int_{\Omega \times \bbr^3} f^p \,dxd\xi + \frac{4(p-1)}{p} \int_{\Omega \times \bbr^3}|\nabla_\xi f^{p/2}|^2 \,dxd\xi = 0. 
\end{aligned}\end{align*}
In our case, $\beta = 2$ and $E_0 = \chi_\lambda(\tilde{v}) + \chi_\lambda (\tilde{u})$, and this consideration together with Gr\"onwall's lemma yields the desired result.\\

\noindent (ii) Let $p \in (1,\infty)$ and $q$ be the H\"older conjugate of $p$, i.e., $p$ and $q$ satisfy $1/p + 1/q = 1$. Then, for $r$ satisfying $rq/p >3$, which is equivalent to $p < (r+3)/3$, we get
\begin{align*}\begin{aligned}
\rho(x,t) &= \int_{\bbr^3} (1+|\xi|)^{r/p} f^{1/p} \frac{f^{1/q}}{(1+|\xi|)^{r/p}}\,d\xi\\
&\le \left(\int_{\bbr^3} (1+|\xi|)^r f \,d\xi \right)^{1/p} \left( \int_{\bbr^3} \frac{f}{(1+|\xi|)^{rq/p}} \,d\xi \right)^{1/q}\\
&\le C \|f(\cdot,\cdot,t)\|_{L^\infty(\Omega \times\bbr^3)}^{1/q} \left(\int_{\bbr^3} (1+|\xi|)^r f \,d\xi \right)^{1/p}\cr
&\le C M^{1/q} \left(\int_{\bbr^3} (1+|\xi|)^r f \,d\xi \right)^{1/p},
\end{aligned}\end{align*}
which implies
\[
\|\rho(\cdot,t)\|_{L^p}^p \le C \int_{\Omega \times \bbr^3} (1+|\xi|^r) f\,dx d\xi.
\]
This gives the desired result for $\rho$. Similarly, we can also obtain the estimate for $\rho u$. 
\end{proof}

Next, we discuss the boundedness of the velocity moments of the kinetic density $f$.

\begin{proposition}\label{PA.2}
For a weak solution $f$ to the kinetic equation in \eqref{App-2} established in Proposition \ref{PA.1}, its velocity moments satisfy the following boundedness condition:
\[
\sup_{t \in (0,T)} \left( \int_{\Omega \times \bbr^3} |\xi|^d f \,dx d\xi + \int_0^t \int_{\Sigma_+} |\xi|^d |\xi \cdot r(x)| (\gamma_+ f) \,d\sigma(x)d\xi ds \right) \le C(\lambda, d, T), \quad \forall d \in [0, d_0]. 
\]
\end{proposition}
\begin{proof} For $d \ge 2$, we set $d$-th moment of $f$ 
\[
m_d(f)(t) := \int_{\Omega \times \bbr^3} |\xi|^d f(x,\xi,t) \,dxd\xi.
\]
A direct computation gives that a weak solution $f$ to \eqref{App-4} satisfies
\begin{align*}\begin{aligned}
\frac{d}{dt} m_d(f)&= -\int_{\Sigma_+} |\xi|^d |\xi \cdot r(x)| (\gamma_+ f) \, d\sigma(x) d\xi + \int_{\Sigma_-} |\xi|^d |\xi \cdot r(x)| g \, d\sigma(x)d\xi\\
&\quad -\beta d m_d(f) + d(d+1)m_{d-2}(f) + d \int_{\Omega \times \R^3 \xi}(E_0 \cdot \xi)|\xi|^{d-2} f \,dx d\xi.
\end{aligned}\end{align*}
We then put $\beta = 2$ and $E_0 = \chi_\lambda(\tilde{v}) + \chi_\lambda (\tilde{u})$ to yield
\begin{align}
\begin{aligned}\label{Ap-6}
\frac{d}{dt}& m_d(f) + \int_{\Sigma_+} |\xi|^d |\xi \cdot r(x)| (\gamma_+ f)\,d\sigma(x) d\xi \\
&\le \int_{\Sigma_-} |\xi|^d |\xi \cdot r(x)| g \,d\sigma(x)d\xi + d(d+1 + 2\lambda^2) m_{d-2}(f).
\end{aligned}
\end{align}
By Proposition \ref{PA.1}, the zeroth moment of $f$, $m_0(f)$, is bounded, and this together with \eqref{Ap-6} and \eqref{DC-2} yields
\[
\sup_{t \in (0,T)} \left( \int_{\Omega \times \bbr^3} |\xi|^d f \,dx d\xi + \int_0^t \int_{\Sigma_+} |\xi|^d |\xi \cdot r(x)| (\gamma_+ f) \,d\sigma(x)d\xi ds \right) \le C(\lambda, d, T)
\]
for $d=0, 2, 4$.
Moreover, for $d \in [0,d_0] \setminus\{0,2,4\}$, we can find $l \in \bbn \cup \{ 0\}$ that satisfies $0 < d - 2l <2$ and thus we can estimate
\begin{align}\label{Ap-7}
\begin{aligned}
\int_{\Omega \times \bbr^3} |\xi|^{d-2l} f \,dxd\xi &\le \left(\int_{\Omega \times \bbr^3} |\xi|^2 f \,dxd\xi \right)^{\frac{d-2l}{2}} \left( \int_{\Omega \times \bbr^3} f \,dxd\xi \right)^{\frac{2+2l-d}{2}} \cr
&\le C(\lambda, d, T).
\end{aligned}
\end{align}
Thus, we repetitively use \eqref{Ap-6} and combine this with \eqref{Ap-7} to get the desired result.
\end{proof}

For later use, we consider the following identity.

\begin{proposition}\label{PA.3}
A weak solution $f$ to the kinetic equation in \eqref{App-2} satisfies the following identity:
\begin{align}
\begin{aligned}\label{Ap-8}
\frac{d}{dt}& \int_{\Omega \times \bbr^3} \left(\frac{|\xi|^2}{2} + \log f \right) f \,dxd\xi + \int_{\Omega\times\bbr^3} \frac{1}{f}|\nabla_\xi f - (\chi_\lambda(\tilde{u}) - \xi)f|^2 \,dxd\xi \\
&= -\int_{\Sigma}(\xi \cdot r(x)) \left( \frac{|\xi|^2}{2} + \log\gamma f + 1 \right) \gamma f \,d\sigma(x) d\xi + 3 \|f\|_{L^1(\Omega \times \bbr^3)}\\
& \quad + \int_{\Omega\times\bbr^3}\chi_\lambda(\tilde{u}) \Big(\chi_\lambda(\tilde{u}) - \xi \Big) f \,dxd\xi  + \int_{\Omega \times \bbr^3} (\chi_\lambda(\tilde{v})-\xi)\cdot \xi f \,dx d\xi.
\end{aligned}
\end{align}
\end{proposition}
\begin{proof}
We can easily check that a weak solution $f$ to the equation \eqref{App-4} satisfies
\begin{align}
\begin{aligned}\label{Ap-9}
&\frac{d}{dt} \int_{\Omega \times \bbr^3} \frac{|\xi|^2}{2} f \,dxd\xi + \int_{\Sigma}(\xi \cdot r(x)) \frac{|\xi|^2}{2} \gamma f \,d\sigma(x) d\xi + 2\beta \int_{\Omega \times \bbr^3} \frac{|\xi|^2}{2} f \,dxd\xi\cr
&\quad =  3\|f\|_{L^1(\Omega\times\bbr^3)} + \int_{\Omega \times \bbr^3} (E_0 \cdot \xi) f \,dx d\xi
\end{aligned}
\end{align}
and
\begin{align}
\begin{aligned}\label{Ap-10}
&\frac{d}{dt}\int_{\Omega \times \bbr^3} f\log f \,dxd\xi +\int_{\Sigma}(\xi \cdot r(x)) (\log\gamma f +1) \gamma f \,d\sigma(x) d\xi \cr
&\quad = 3\beta \|f\|_{L^1(\Omega\times\bbr^3)}-4 \int_{\Omega} |\nabla_\xi \sqrt{f}|^2 \,dxd\xi.
\end{aligned}
\end{align}
Putting $\beta = 2$, $E_0 = \chi_\lambda(\tilde{v}) + \chi_\lambda (\tilde{u})$ into \eqref{Ap-9} and \eqref{Ap-10}, we conclude the desired result.
\end{proof}
Next, we take into account the fluid equations in \eqref{App-2}. Since the Navier-Stokes system in \eqref{App-2} is exactly the same with that of \cite{M-V}, we can directly employ the results in \cite[Lemma 3.3]{M-V} to deduce the proposition below.
\begin{proposition}\label{PA.4}
Assume that the initial data $(n_0, v_0)$ satisfy \eqref{init-1}. Then there exists a unique weak solution to the fluid equations in \eqref{App-2} satisfying
\begin{equation}\label{Ap-11}
\frac{d}{dt} \lt(\int_\Omega \frac{1}{2}n_k |v|^2\,dx  + \frac{1}{\gamma-1} \int_\om n^\gamma \,dx\rt) + \int_\Omega |\nabla v|^2 dx = \int_\Omega \rho(u-\tilde{v})\mathds{1}_{\{|\tilde v|\le \lambda\}} v \,dx.
\end{equation}
Here we remind the reader the notation $n_k = n \star h_k$, where $h_k$ is appeared in the beginning of this section.
\end{proposition}

Now, we are ready to prove that $\mathcal{T}$ is well-defined.

\begin{lemma}\label{CA.1}
There exists a constant $C := C(\lambda, d_0, k,T,\e)$ such that
\[
\|\mathcal{T} (\tilde{u}, \tilde{v})\|_{\mathcal{S}} \le C(\lambda, d_0, k, T, \e)
\]
for $(\tilde{u}, \tilde{v}) \in \mathcal{S}$. 
\end{lemma}
\begin{proof} By combining Proposition \ref{PA.1} (ii) with Proposition \ref{PA.2}, we find
\[
\|(\rho u)(\cdot,t)\|_{L^p(\Omega)} \le C(\lambda, d_0,T), \quad \forall p \in \left[1,\frac{d_0+3}{4}\right). 
\]
Since $d_0>5$, $(d_0 + 3)/4 > 2$, and this enables us to take $p=2$ in the above inequality. Thus we can estimate
\[
\|u_\e(\cdot,t)\|_{L^2(\om)} \le \frac{1}{\e} \|(\rho u)(\cdot,t)\|_{L^2(\Omega)} \le C(\lambda, d_0, T,\e). 
\]
We next estimate $\|v\|_{L^2}$. For this, we first obtain
\begin{align*}\begin{aligned}
\left|\int_{\Omega} \rho(u-\tilde{v})\mathds{1}_{\{|\tilde{v}| \le \lambda\}} v \,dx \right|&\le \|(\rho u)(t)\|_{L^2(\Omega)} \|v\|_{L^2(\Omega)} + \lambda \|\rho\|_{L^2(\Omega)}\|v\|_{L^2(\Omega)}\\
& \le \|v\|_{L^2(\Omega)}^2 +\left( \|\rho u\|_{L^2(\Omega)}^2 + \lambda^2 \|\rho\|_{L^2(\Omega)}^2 \right).
\end{aligned}\end{align*}
This together with the bounded estimate of $\|\rho\|_{L^2(\Omega)}$ by Propositions \ref{PA.1} and \ref{PA.2} yields
\[
\frac{d}{dt} \lt(\int_\Omega \frac{1}{2}n_k |v|^2\,dx  + \frac{1}{\gamma-1} \int_\om n^\gamma \,dx\rt) +  \int_\Omega |\nabla v|^2 dx \le \|v\|_{L^2(\Omega)}^2 + C(\lambda, d_0, T,\e). 
\]
Since $\Omega$ is bounded and 
\[
\int_\Omega n \,dx = \int_\Omega n_0 \,dx >0,
\] 
we obtain $\inf_{x \in \Omega}n_k(x,t) \ge c_k>0$. Thus, we get
\[ 
\|v(\cdot,t)\|_{L^2(\Omega)}^2 \le C(k) \int_\Omega n_k |v|^2 \,dx, 
\]
and this gives
\[
\frac{d}{dt} \lt(\int_\Omega \frac{1}{2}n_k |v|^2\,dx  + \frac{1}{\gamma-1} \int_\om n^\gamma \,dx\rt) + \int_\Omega |\nabla v|^2 \,dx \le C(k) \int_\Omega n_k |v|^2 \,dx + C(\lambda, d_0, T,\e). 
\]
Finally, we integrate the above inequality with respect to $t$, use the fact $\inf_{x \in \Omega}n_k(x,t) \ge c_k>0$ and apply Gr\"onwall's lemma to conclude
\[
\|v(\cdot,t)\|_{L^2(\Omega)}\le C(\lambda, d_0, k , T, \e).
\]
This completes the proof.
\end{proof}

\subsection{$\mathcal{T}$ is compact.} To show the compactness of the operator $\mathcal{T}$, we need to provide the convergence of the sequence $\mathcal{T}(\tilde{u}^m, \tilde{v}^m) = (u_\e^m, v^m)$ up to a subsequence, where $(\tilde{u}^m, \tilde{v}^m)$ is uniformly bounded in $\mathcal{S}$. Then, it follows from the proof of \cite[Lemma 3.4]{M-V} that $v^m$ converges strongly in $L^2((0,T)\times\Omega)$ up to subsequences. Thus, it suffices to show the strong convergence of $u_\e^m$ up to subsequences. For this, we need the velocity averaging lemma, see \cite{P-S} for an instance.

\begin{lemma}\label{LA.1}
Let $T>0$, $\{f^m\}$ be bounded in $L_{loc}^p(\R^3 \times \R^3 \times (0,T))$ with $1 < p < \infty$ and $\{G^m\}$ be bounded in $L_{loc}^p(\R^3 \times \R^3 \times (0,T))$. If $f^m$ and $G^m$ satisfy
\[
\pa_t f^m + \xi \cdot \nabla f^m = \nabla_\xi^\alpha G^m, \quad f^m|_{t=0} = f_0 \in L^p(\R^3 \times \R^3), 
\]
for some multi-index $\alpha$ and $\varphi \in \mathcal{C}_c^{|\alpha|}(\R^3 \times \R^3)$, then
\[ 
\left\{ \int_{\bbr^3} f^m \varphi \,d\xi \right\} 
\]
is relatively compact in $L_{loc}^p(\bbr^3 \times (0,T))$.
\end{lemma}
By using the previous lemma, the following lemma can be proved similarly to \cite[Lemma 2.7]{K-M-T}.
\begin{lemma}\label{LA.2}
Let $T>0$, $\{f^m\}$  and $\{G^m\}$ be as in Lemma \ref{LA.1}. Assume that for $r \ge 2$,
\[
\sup_{m \in \bbn} \|f^m\|_{L^\infty(\R^3 \times \R^3 \times (0,T))} + \sup_{m \in \bbn}\| (|\xi|^{r} + |x|^{r}) f^m\|_{L^\infty(0,T;L^1(\R^3 \times \R^3))}  <\infty.
\]
Then, for any $\varphi(\xi)$ satisfying $|\varphi(\xi)| \le c|\xi|$ for $|\xi|$ large enough, the sequence
\[ 
\left\{ \int_{\R^3} f^m \varphi \,d\xi \right\} 
\]
is relatively compact in $L^q(\R^3 \times (0,T))$ for any $q \in [1, (3+r)/4)$.
\end{lemma}

Now, we are ready to prove the compactness of $\mathcal{T}$.

\begin{lemma}\label{CA.2}
For a uniformly bounded sequence $(\tilde{u}^m, \tilde{v}^m)$ in $\mathcal{S}$, the sequence $\mathcal{T}(\tilde{u}^m, \tilde{v}^m) = (u_\e^m, v^m)$ converges strongly in $\mathcal{S}$, up to a subsequence.
\end{lemma}
\begin{proof}
For the convergence of $\{u_\e^m\}$, we consider the following setting in Lemma \ref{LA.2}:
\[ 
r = d_0, \quad f^m = f^m \mathds{1}_{\Omega \times \R^3}, \quad\mbox{and}\quad G^m =[ \nabla_\xi f^m - (\chi_\lambda(\tilde{v}^m) + \chi_\lambda(\tilde{u}^m) -2\xi)f^m ] \mathds{1}_{\Omega \times \R^3}.
\]
Then, we can obtain the following strong convergence by taking $\varphi(\xi) =1$ and $\varphi(\xi) = \xi$ in Lemma \ref{LA.2}, respectively, up to a subsequence:
\begin{align*}\begin{aligned}
&\rho^m \to \rho \quad \quad \mbox{in} \quad L^2((0,T)\times \Omega) \quad \mbox{and a.e.},\\
&\rho^m u^m \to \rho u \quad \mbox{in} \quad L^2((0,T)\times \Omega).
\end{aligned}\end{align*}
This asserts the convergence of $\{u_\e^m\}$ up to a subsequence. For the fluid part, we again use the same argument as in \cite{M-V}, however we provide the proof for readers' convenience. We regularize the continuity equation in \eqref{App-2} to get
\[
\partial_t n_k + \nabla \cdot (nv)_k = 0.
\]
This gives
\[
n_k \partial_t v + (nv)_k \cdot \nabla v + \nabla p - \Delta v = -\rho(u-\tilde{v})\mathds{1}_{\{|\tilde{v}|\le \lambda\}}. 
\]
Then, the following can be obtained by exploiting $n_k \ge c_k >0$:
\[
\|\partial_t v^m\|_{L^2(0,T;H^{-1}(\Omega))} \le C(\lambda, d_0, k, T, \e)
\] 
uniformly in $m$. Here, we use Aubin-Lions lemma to obtain the convergence of $\{v^m\}$ up to a subsequence.
\end{proof}
\begin{proof}[ Proof of Theorem \ref{TA.1}]  We can deduce from Lemmas \ref{CA.1}--\ref{CA.2} that the operator $\mathcal{T}$ can be proven to be well-defined, continuous and compact. Thus, we prepared all the materials to use Schauder's fixed point theorem, and hence we have the existence of a fixed point, which becomes a weak solution to \eqref{App-1}. 
\end{proof}
Here, we provide the entropy inequality that is satisfied by a weak solution $(f,n,v)$ to \eqref{App-1}. For this, we define a functional $\mathcal{F}_k(f,n,v)$ and corresponding dissipation functional $D_{\lambda,\e}$ as
\[ 
\mathcal{F}_k (f, n, v) := \int_{\Omega \times \bbr^3 } f \left( \log f  +\frac{|\xi|^2}{2} \right) dxd\xi + \frac12\int_{\Omega} n_k |v|^2 \,dx + \frac{1}{\gamma-1} \int_{\Omega} n^\gamma \,dx
\]
and 
\[
\md_{\lambda,\e}(f,v) := \int_{\Omega \times \bbr^3} \frac{1}{f}|(\chi_\lambda(u_\e) - \xi))f -\nabla_\xi f|^2 + | (\chi_\lambda(v)-\xi)f|^2f \,dxd\xi, 
\]
respectively. Then we provide the relation between $\mathcal{F}_k(f,n,v)$ and $D_{\lambda,\e}(f,v)$ in the lemma below.
\begin{lemma}\label{LA.3}
Let $\gamma > 3/2$ and $T \in (0,\infty)$. Assume that a triplet $(f,n,v)$ is a weak solution to \eqref{App-1} corresponding to initial data $(f_0, n_0, v_0)$ and boundary condition $g$, which is established in Propositions \ref{PA.1} and \ref{PA.2}. Then we have
\begin{align*}\begin{aligned}
\begin{aligned}
\mathcal{F}_k &(f,n,v)(t) + \int_0^t \md_{\lambda,\e} (f,v)(s)\,ds + \int_0^t \int_\Omega |\nabla v|^2 \,dxds\\
&\le \mathcal{F}_k(f_0,n_0,v_0) - \int_0^t \int_\Sigma (\xi \cdot r(x)) \left(\frac{|\xi|^2}{2} + \log \gamma f + 1\right) \gamma f \,d\sigma(x) d\xi ds\\
&\quad + 3\int_0^t \|f(\cdot,\cdot,s)\|_{L^1(\Omega \times \bbr^3)} \,ds.
\end{aligned}
\end{aligned}\end{align*}
\end{lemma}
\begin{proof}
Since
\[  
\int_\Omega \rho(u-v) \mathds{1}_{\{|v|\le \lambda\}} v \,dx = -\int_{\Omega\times\bbr^3}\chi_\lambda(v) \lt(\chi_\lambda(v) - \xi \rt) f \,dxd\xi,  
\]
we combine \eqref{Ap-8} and \eqref{Ap-11} with $(u_\e, v)$ instead of $(\tilde{u}, \tilde{v})$ with the above relation to yield
\begin{align*}\begin{aligned}
\mathcal{F}_k &(f,n,v)(t) + \int_0^t \md_{\lambda,\e} (f,v)(s)\,ds + \int_0^t \int_\Omega |\nabla v|^2 \,dxds\\
&= \mathcal{F}_k(f_0,n_0,v_0) - \int_0^t \int_\Sigma (\xi \cdot r(x)) \left(\frac{|\xi|^2}{2} + \log \gamma f + 1\right) \gamma f \,d\sigma(x) d\xi ds\\
& \quad + \int_0^t  \int_{\Omega \times \bbr^3} \chi_\lambda (u_\e) \cdot (\chi_\lambda(u_\e)-\xi) f \,dx d\xi ds\\
&\quad + 3\int_0^t \|f(\cdot,\cdot,s)\|_{L^1(\Omega \times \bbr^3)} \,ds.
\end{aligned}\end{align*}
On the other hand, we find
\begin{align*}\begin{aligned}
 \int_{\Omega \times \bbr^3} \chi_\lambda (u_\e) \cdot (\chi_\lambda(u_\e)-\xi) f \,dx d\xi &= \int_\Omega \left( \rho \left|\frac{\rho u}{\rho + \e}\right|^2 - \frac{|\rho u|^2}{\rho+\e}  \right)\mathds{1}_{\{|u_\e|\le\lambda | \}}\,dx\\
& = \int_\Omega \frac{|\rho u|^2}{\rho + \e} \left(  \frac{\rho}{\rho + \e} -1 \right) \mathds{1}_{\{|u_\e|\le\lambda | \}}\,dx \cr
&\le 0,
\end{aligned}\end{align*}
and this asserts the desired result.
\end{proof}

\subsection{Proof of Theorem \ref{T3.1}}\label{app_3} Motivated from \cite{M-V}, we proceed to the proof for the existence of a weak solution to the system \eqref{A-1}. We assume that initial data $(f_0, n_0, v_0)$ satisfy \eqref{init-1} and $g$ satisfies \eqref{DC}. Then, we consider the sequences $(f_0^m, n_0^m, v_0^m)$ and $(g^m)$ approximating initial data and boundary data, respectively:
\begin{align*}\begin{aligned}
f_0^m \rightarrow f_0 \quad & \mbox{in } \ (L^1\cap L^\infty)(\Omega\times\bbr^3),\cr
n_0^m \rightarrow n_0 \quad & \mbox{in } \ L^\infty(\Omega),\cr
v_0^m \rightarrow v_0 \quad & \mbox{in } \ (L^\infty \cap H^1_0)(\om), \quad \mbox{and}\cr
g^m \rightarrow g \quad & \mbox{in } \ L^\infty(\Sigma \times (0,T)).
\end{aligned}\end{align*}
We will assume that they satisfy \eqref{init-1} and \eqref{DC} uniformly in $m$, and for each $m$, \eqref{init-2}--\eqref{DC-2} hold. For every $\lambda$, $k$, $\e$ and $m$, we denote $(f^m, n^m, v^m)$ by the solutions to \eqref{App-1} with initial data
\[(f^m(0), n^m(0), v^m(0)) =: (f_0^m, n_0^m, v_0^m). \]
Note that existence of $(f^m, n^m, v^m)$ is guaranteed by the previous estimates. \newline

\noindent $\bullet$ (Step A : Uniform boundedness) First, the estimates in \eqref{App-3} imply the existence of a constant $C$, which is independent of $k$, $\lambda$, $\e$ and $m$ and satisfies
\begin{equation}\label{Ap-13}
\|f^m\|_{L^\infty(0,T;L^p(\Omega\times\bbr^3))} + \|\gamma f^m \|_{L^p((0,T) \times\Sigma )} \le C
\end{equation}
for all $p \in [1,\infty]$. Here, we use Lemma \ref{LA.4} to get
\begin{align*}\begin{aligned}
-\int_0^t& \int_\Sigma (\xi \cdot r(x)) \left( \frac{|\xi|^2}{2} + \log \gamma f^m+1 \right) \gamma f^m \,d\sigma (x) d\xi ds\\
&\le -\frac{1}{2} \int_0^t \int_{\Sigma_+}|\xi \cdot r(x)| \left( \frac{|\xi|^2}{2} + \log^+ \gamma_+ f^m +1 \right) \gamma_+ f^m \,d\sigma (x) d\xi ds\\
& \quad + \int_0^t \int_{\Sigma_-}|\xi \cdot r(x)| \left( \frac{|\xi|^2}{2} + \log g^m +1 \right) g^m \,d\sigma (x) d\xi ds +C.
\end{aligned}\end{align*}
We also have
\begin{align*}\begin{aligned}
&\left|\int_{\Omega \times \bbr^3}\chi_\lambda (u_\e^m )\cdot \xi f^m \,dx d\xi \right| \\
&\quad \le \int_{\Omega \times \bbr^3}|u_\e^m \cdot \xi f^m| \,dx d\xi\\
&\quad \le \frac{1}{2} \int_{\Omega \times \bbr^3} |u_\e^m|^2 f^m \,dx d\xi + \frac{1}{2}\int_{\Omega \times \bbr^3} |\xi|^2 f^m \,dx d\xi\\
&\quad \le  \frac{1}{2} \int_{\Omega} \rho^m |u_\e^m|^2 \,dx + \frac{1}{2}\int_{\Omega \times \bbr^3} |\xi|^2 f^m \,dx d\xi\\
&\quad \le \frac{1}{2} \int_{\Omega} \rho^m \left|\frac{\rho^m u^m}{\rho^m + \e}\right|^2 dx+ \frac{1}{2}\int_{\Omega \times \bbr^3} |\xi|^2 f^m \,dx d\xi\\
&\quad \le \int_{\Omega \times \bbr^3} |\xi|^2 f^m \,dx d\xi,
\end{aligned}\end{align*}
where we used Young's inequality. Then the above estimates together with Gr\"onwall's lemma, Lemma \ref{LA.4}, and \eqref{App-3} with $p=1$ yield
\begin{align*}\begin{aligned}
&\int_{\Omega \times \bbr^3} \left( \frac{|\xi|^2}{4} + |\log f^m| \right) f^m \,dx d\xi + \int_\Omega n_k^m \frac{|v^m|^2}{2}  + \frac{1}{\gamma-1}(n^m)^\gamma \,dx\\
&\qquad +\frac{1}{2} \int_0^t \int_{\Sigma_+} |\xi \cdot r(x)| \left( \frac{|\xi|^2}{2} + \log^+(\gamma_+f^m) +1 \right) \gamma_+f^m \,d\sigma(x)d\xi ds\\
& \qquad + \int_0^t \md_{\lambda,\e} (f^m,v^m)(s)\,ds + \int_0^t \int_\Omega |\nabla v^m|^2 \,dxds\\
& \quad \le C \bigg( \mathcal{F}_k(f_0^m, n_0^m, v_0^m)  +\|f_0^m\|_{L^1(\Omega \times \bbr^3)} + \|g^m\|_{L^1((0,T)\times \Sigma_-)}\\
& \hspace{2cm} + \int_0^t \int_{\Sigma_-} |\xi \cdot r(x)| \left(\frac{|\xi|^2}{2} + \log g^m +1 \right) g^m \,d\sigma(x)d\xi ds +1  \bigg).
\end{aligned}\end{align*}
Thus, we can find a constant $C>0$ independent of $\lambda$, $k$, $m$ and $\e$ such that
\begin{align*}\begin{aligned}
&\int_{\Omega\times^3} ( 1+ |\xi|^2) f^m \,dxd\xi \le C, \quad \forall t \in [0,T],\\
& \int_0^T \int_{\Sigma_+} ( 1+|\xi|^2) \gamma_+ f^m |\xi \cdot r(x)| \,d\sigma(x) d\xi ds \le C, \quad \mbox{and}\\
&\| n^m\|_{L^\infty(0,T; L^1\cap L^\gamma(\Omega))} + \|\sqrt{n^m} v^m\|_{L^\infty(0,T; L^2(\Omega))} + \|\nabla v^m \|_{L^2((0,T)\times \Omega)} \le C.
\end{aligned}\end{align*}
Hence, Proposition \ref{PA.1} (ii) guarantees that there exists a constant $K>0$, independent of $\lambda$, $k$, $m$, and $\e$, such that
\[
\sup_{t \in [0,T]} \|\rho^m(t)\|_{L^p} + \sup_{t \in [0,T]}\| \rho^m u^m\|_{L^q} \le K
\]
for any $p \in [1,5/3)$ and $q \in [1,5/4)$. With these uniform bounds, we prove the existence of limit function $(f, n, v)$ and that they become a weak solution to \eqref{A-1}. \newline

\noindent $\bullet$ (Step B : Convergence toward weak limits) Here, we show that the sequence $(f^m, n^m, v^m)$ converges to a weak limit $(f,n,v)$. We first let $\lambda = m$ and tend $m$ to infinity to get the weak solution for the following system:
\begin{align}
\begin{aligned}\label{Ap-15}
&\partial_t f + \xi \cdot \nabla f + \nabla_\xi \cdot (v -\xi)f) =\Delta_\xi f - \nabla_\xi \cdot \left( \left(u_\e - \xi \right) f \right),\\
&\partial_t n + \nabla \cdot (n v)=0,\\
&\partial_t (n_k v) + \nabla\cdot ((n v)_k \otimes v) + \nabla p  -\Delta v = -\rho(u-v).
\end{aligned}
\end{align}
For the fluid part, we notice that $\Omega$ is bounded and homogeneous Dirichlet boundary condition is imposed on $v$. This enables us to use Poincar\'e inequality to get
\[ 
\| v\|_{L^6(\Omega)} \le C \|\nabla u\|_{L^2(\om)}.
\]
Thus, $v^m$ is bounded in $L^2(0,T; L^6(\Omega))$, and this implies
\[
\|n^m v^m\|_{L^2(0,T;L^{6/5}(\Omega))} \le \|v^m\|_{L^2(0,T; L^6(\Omega))} \|n^m\|_{L^2(0,T;L^{3/2}(\Omega))}
\]
so that $n^m v^m$ belong to $L^2(0,T;L^{6/5}(\Omega))$. Note that we also used $\gamma >3/2$ here. From now on, we can use classical results from \cite{F,L} to get the following convergences:
\begin{align*}\begin{aligned}
n^m \rightarrow n \quad & \mbox{in } \ L^1(\Omega \times (0,T))\cap\mathcal{C}([0,T];L_w^\gamma (\Omega)),\cr
v^m \rightharpoonup v \quad & \mbox{in } \  L^2((0,T);H_0^1(\Omega)), \quad \mbox{and}\cr
n^m v^m \rightarrow nv & \mbox{in } \ \mathcal{C}([0,T];L_w^{2\gamma/(\gamma+1)}(\Omega)).
\end{aligned}\end{align*}
Furthermore, we have
\[n_k^m v^m \rightharpoonup n_k v, \quad \mbox{and} \quad (n^m)^\gamma \rightharpoonup n^\gamma. \]
Next, we discuss the kinetic part. Uniform boundedness \eqref{Ap-13} gives a limit function $f$:
\[f^m \rightharpoonup f, \quad L^\infty(0,T;L^p(\Omega\times\bbr^3)). \]
We may also obtain the weak convergence of $\rho^m$ and $\rho^m u^m$ from the same argument as Proposition \ref{PA.1} (ii):
\begin{align*}\begin{aligned}
&\rho^m \rightharpoonup \rho \quad \quad \mbox{in} \quad L^\infty(0,T;L^p(\Omega)) \quad \forall p \in (1,5/3) \quad \mbox{and}\\
&\rho^m u^m \rightharpoonup \rho u \quad \mbox{in} \quad L^\infty(0,T;L^p( \Omega)) \quad \forall p \in (1,5/4).
\end{aligned}\end{align*}
In addition, we can apply $r=2$, $d=3$, $f^m := f^m \mathds{1}_{\Omega \times \R^3}$, and 
\[
G^m :=[ \nabla_\xi f^m - (\chi_\lambda(v^m) + \chi_\lambda(u_\e^m) -2\xi)f^m ] \mathds{1}_{\Omega \times \R^3}
\] 
to Lemma \ref{LA.2} to get the strong convergence up to a subsequence:
\begin{align*}\begin{aligned}
&\rho^m \to \rho \quad \quad \mbox{in} \quad L^p(\Omega \times (0,T)) \quad \forall p \in (1,5/4) \quad \mbox{and} \quad \mbox{a.e.},\\
&\rho^m u^m \to \rho u \quad \mbox{in} \quad L^p(\Omega \times (0,T)) \quad \forall p \in (1,5/4).\\
\end{aligned}\end{align*}
Next procedure is to show that a triplet of limit functions $(f,n,v)$ satisfies \eqref{Ap-15} in distributional sense. For this, we only need to show the convergence of the coupling terms and self-alignment term. For coupling terms, we use the strong convergence in $L^p(\Omega \times (0,T))$ and boundedness of $\rho^m$, $\rho^m u^m$ in $L^2(0,T;L^p(\Omega))$, $p \in [1,5/4)$ and boundedness of $v^m$ in $L^2(0,T;L^6(\Omega))$ to get
\[\rho^m(u^m-v^m) \rightharpoonup \rho(u-v) \quad \mbox{in }\ L^2(0,T;L^p(\Omega)), \quad p \in [1,5/4). \]
Here, we can find out that $\rho^m (u^m - v^m)\mathds{1}_{\{|v^m|>\lambda\}}$ goes to zero as $\lambda$ tends to infinity:
\begin{align*}\begin{aligned}
\|&\rho^m( u^m - v^m)\mathds{1}_{\{|v^m|>\lambda\}}\|_{L^1(\Omega \times (0,T))}\\
& \le \frac{1}{\lambda} \left( \|\rho^m u^m \cdot v^m \|_{L^1(\Omega\times(0,T))} + \|\rho^m |v^m|^2\|_{L^1(\Omega\times(0,T))} \right)\\
& \le  \frac{2}{\lambda} \left(\|\rho^m |u^m|^2\|_{L^1(\Omega\times(0,T))} + \|\rho^m |v^m|^2\|_{L^1(\Omega\times(0,T))} \right)\\
& \le  \frac{2}{\lambda} \left(\||\xi|^2 f^m \|_{L^1(\Omega\times\bbr^3 \times (0,T))} + \|\rho^m |v^m|^2\|_{L^1(\Omega\times(0,T))} \right) \\
& \le\frac{2}{\lambda}  \left(\||\xi|^2 f^m \|_{L^1(\Omega\times\bbr^3 \times (0,T))} + \|\rho^m\|_{L^2(0,T;L^{6/5}(\Omega))} \|v^m\|_{L^2(0,T; L^6(\Omega))} \right) \le \frac{C}{\lambda} \to 0.
\end{aligned}\end{align*}
Thus, we can obtain
\[
\rho^m(u^m-v^m)\mathds{1}_{\{|v^m|\le \lambda\}} \rightharpoonup \rho(u-v) \quad \mbox{in }\ L^2(0,T;L^p(\Omega)), \quad p \in [1,5/4).
\]
Similarly, we can also get
\[
(v^m-\xi)f^m \mathds{1}_{\{|v^m|>\lambda \}} \to 0 \quad \mbox{as } \ \lambda \to \infty. 
\]
Moreover, for $p \in [1,5/4)$, we have
\begin{align*}\begin{aligned}
\|&v^m f^m\|_{L^2(0,T;L^p(\Omega\times\bbr^3))}\\
&\le \|f^m\|_{L^\infty((0,T)\times\Omega\times\bbr^3)}^{\frac{p-1}{p}} \|\rho^m\|_{L^\infty(0,T;L^{\frac{5}{5-p}}(\Omega))}\|v^m\|_{L^2(0,T;L^5(\Omega))}\\
& \le C \|v^m\|_{L^2(0,T;H_0^1(\Omega))} \le C.
\end{aligned}\end{align*}
This implies the following weak convergence:
\[
(\chi_\lambda(v^m)-\xi)f^m \rightharpoonup (v-\xi)f \quad \mbox{in }\ L^2(0,T;L^p(\Omega)), \quad p \in [1,5/4). 
\]

\noindent Finally, we show the convergence of $(\chi_\lambda(u_\e^m) - \xi)f^m$ toward $(u_\e-\xi)f$ in distributional sense. For this, we notice that
\begin{align*}\begin{aligned}
\int_0^T\int_{\Omega \times \bbr^3} | u_\e^m f^m  \mathds{1}_{\{|u_\e^m|>\lambda \}}| \,dx d\xi ds &\le \frac{1}{\lambda}\int_0^T \int_{\Omega \times \bbr^3} | u_\e^m|^2 f^m\,dx d\xi ds\\
&\le \frac{1}{\lambda} \int_{\Omega\times\bbr^3}|\xi|^2 f^m \,dxd\xi ds \  \to 0
\end{aligned}\end{align*}
as $\lambda \to \infty$. This implies that the weak limit $(f,n,v)$ satisfies \eqref{Ap-15} in distributional sense and hence a weak solution to the system \eqref{Ap-15} for eack $k$. For the limit $k \to 0$, since all the uniform estimates still hold, it is similar to the previous limit $\lambda \to \infty$ and hence, the sequence of weak solutions $(f^k, n^k, v^k)$ to the system \eqref{Ap-15} converges to weak limit $(f, n,v)$ which becomes a weak solution to \eqref{C-1}. \newline

\noindent $\bullet$ (Step C : Entropy inequality) Now, it remains to show that weak solution $(f,n,v)$ exists up to a subsequence obtained in the previous subsection satisfies the entropy inequality \eqref{T3-1}. Since we have proved that $(f^{\lambda, k, m}, n^{\lambda,k,m}, v^{\lambda, k, m})$ converges to a weak solution $(f,n,v)$ as $\lambda$ and $m$ tends to infinity and $k$ tends to 0, we take the limit in \eqref{App-4} and use weak convergences from previous section, convexity of entropy and strong convergence of $(\rho^{\lambda, k,m})$ to yield
\begin{align*}\begin{aligned}
\mathcal{F}&(f,n,v)(t) + \int_0^t \md_\e(f,v)(s)\,ds + \int_0^t \int_\Omega |\nabla v|^2 \,dxds\\
& \quad + \int_0^t \int_{\partial\Omega \times \bbr^3} (\xi \cdot r(x)) \left( \frac{|\xi|^2}{2} + \log \gamma f + 1 \right) \gamma f \,d\sigma(x) d\xi ds\\
&\le \mathcal{F}(f_0, n_0, v_0) + 3 \int_0^t \|f(\cdot,\cdot,s)\|_{L^1(\Omega \times \bbr^3)} \,ds.
\end{aligned}\end{align*}

%
%
%
%
%

\section{Proof of Lemma \ref{P5.4}}\label{app_b}
In this appendix, we provide the details of proof of Lemma \ref{P5.4}.

\subsection{Uniform lower bound estimate of $h^m$} Consider a forward characteristic $X^{m+1}$ which solves the following ordinary differential equations:
\[
\frac{d X^m(t,x)}{dt} = v^{m-1}(X^m(t,x),x),
\] 
with the initial data:
\[
X^m(0,x) = x
\]
for all $m \in \N$. Since $v^{m-1} \in \mathfrak{X}(T,\om)$, the above characteristic equation is well-defined on the interval $[0,T]$, and this together with the continuity equation for $h^m$ in \eqref{E-5} yields
\bq\label{est_hm}
1 + h^m(X^m(t,x),x) = (1 + h_0(x)) \exp\lt(-\int_0^t (\nabla \cdot v^{m-1})(X^m(s,x),x)\,ds\rt).
\eq
We now use the assumption $\|\eta^m\|_{\mathfrak{X}^s(T,\om)} \leq M_1$ to get
\[
1 + \inf_{x \in \om} h^m(x,t) \geq \lt(1 +  \inf_{x \in \om} h_0(x)\rt)\exp\lt( -C_0M_1 T\rt),
\]
where $C_0>0$ is independent of $m$. This yields there exists $T_1 > 0$ independent of $m$ such that
\[
1 + \sup_{0 \leq t \leq T_1}\inf_{x \in \om} h^m(x,t) > \delta_0.
\]
\subsection{Zeroth order estimate} In the rest of this appendix, let us denote $A \leq B$ for $d \times d$ matrices $A = A_{ij}$ and $B = B_{ij}$ with $d \in \N$ when $A_{ij} \leq B_{ij}$ for $1\leq i,j \leq d$. 

We first easily obtain from \eqref{est_hm} that
\[
1 + \inf_{x \in \om} h^m(x,t) \leq \lt(1 +  \sup_{x \in \om} h_0(x)\rt)\exp\lt( C_0M_1 T\rt).
\]
This gives
\[
0 \leq A^0(\eta^m)(x,t) \leq \max\lt\{\gamma(1+h_0(x))e^{C_0M_1|2-\gamma| t}, (1+h_0(x))e^{C_0M_1t} \rt\}\mathbb{I}_{4 \times 4}.
\]
Similarly, we use the fact that the inverse of matrix $A^0(\eta^m)$ is given by
\[
(A^0(\eta^m)(x,t))^{-1} = \left(\begin{array}{cc} (1/\gamma)(1 +h^m(x,t))^{2-\gamma} & 0 \\ 0 & (1/(1+h^m(x,t)))\mathbb{I}_{3 \times 3} \end{array}\right)
\]
to have
\[
0 \leq (A^0(\eta^m)(x,t))^{-1} \leq \max\lt\{(1/\gamma)(1+h_0(x))e^{C_0M_1|2-\gamma| t}, (1+h_0(x))e^{C_0M_1t} \rt\}\mathbb{I}_{4 \times 4}.
\]
Since $\gamma \geq 1$ and $\|h_0\|_{H^s} < \e_0$, we can choose $(T_1 \geq)\, T_2 > 0$, which is independent of $m$, such that
\[
\sup_{0 \leq t \leq T_2}\|A^0(\eta^m)(\cdot,t)\|_{L^\infty} \leq \gamma (1+ \e_0)\mathbb{I}_{4 \times 4}
\]
and
\[
\sup_{0 \leq t \leq T_2} \|(A^0(\eta^m)(\cdot,t))^{-1}\|_{L^\infty} \leq (1+ \e_0)\mathbb{I}_{4 \times 4}.
\] 
We also find
\[
\partial_t A^0(\eta^m) \le C\|\partial_t h^m\|_{L^\infty} \mathbb{I}_{4\times4} \leq CM_1\mathbb{I}_{4\times4}
\]
and
\[
\partial_j A^j(\eta^m)\le C\|\nabla \eta^m\|_{L^\infty}\mathbb{I}_{4\times4}\leq CM_1\mathbb{I}_{4\times4}.
\]
We now estimate $\|\eta^m\|_{L^2}$. It follows from the Navier-Stokes equations in \eqref{E-3-1} that
\begin{align*}\begin{aligned}
&\frac{1}{2}\left\{\partial_t \int_\Omega A^0(\eta^m)\eta^{m+1} \cdot \eta^{m+1}\, dx - \int_\Omega \partial_t A^0(\eta^m)\eta^{m+1} \cdot \eta^{m+1} \,dx \right\}\\
&\quad + \frac{1}{2}\left\{ \sum_{j=1}^3 \int_\Omega \partial_j (A^j(\eta^m)\eta^{m+1} \cdot \eta^{m+1}) \,dx - \int_\Omega \partial_j A^i (\eta^m)\eta^{m+1} \cdot \eta^{m+1} \,dx \right\}\\
&\qquad = \int_\Omega A^0(\eta^m)(E_1(h^m,v^{m+1}) + E_2(g^m,u^m,\eta^m))\cdot \eta^{m+1}\, dx.
\end{aligned}\end{align*}
We then use the above observations to estimate
\begin{align*}\begin{aligned}
\int_\Omega A^0(\eta^m) E_1(h^m,v^{m+1})\eta^{m+1} \,dx &= \int_\Omega v^{m+1}  \cdot \Delta v^{m+1} \,dx = -\|\nabla v^{m+1} \|_{L^2}^2,\\
\int_\Omega A^0(\eta^m)E_2 (g^m,u^m,\eta^m) \eta^{m+1} \,dx &= \frac{1}{M}\int_\Omega e^{g^m} (u^m-v^m)\cdot v^{m+1} \,dx \\
&\le e^{\|g^m\|_{L^\infty}}(\|u^m\|_{L^2} +\|v^m\|_{L^2})\|v^{m+1}\|_{L^2},\\
&\leq C(M_1 + M_1') e^{CM_1'}\|v^{m+1}\|_{L^2},\\
\int_\Omega \partial_t A^0(\eta^m)\eta^{m+1} \cdot \eta^{m+1} \,dx &\le C\|\partial_t h^m\|_{L^\infty}\|\eta^{m+1}\|_{L^2}^2 \leq CM_1\|\eta^{m+1}\|_{L^2}^2, \quad \mbox{and}\\
\int_\Omega \partial_j A^j(\eta^m)\eta^{m+1} \cdot \eta^{m+1} \,dx &\le C\|\nabla \eta^m\|_{L^\infty}\|\eta^{m+1}\|_{L^2}^2 \leq CM_1\|\eta^{m+1}\|_{L^2}^2,
\end{aligned}\end{align*}
where $C>0$ is independent of $m$. This yields 
\begin{align*}\begin{aligned}
\frac{1}{1+\e_0}\|\eta^{m+1}(\cdot,t)\|_{L^2}^2 & \le \int_\Omega A^0(\eta^m)\eta^{m+1} \cdot \eta^{m+1} \,dx\\
&\le \int_\Omega A^0(\eta_0)\eta_0 \cdot \eta_0 \,dx + C(M_1 + M_1') e^{CM_1'}\int_0^t \|v^{m+1}(\cdot,\tau)\|_{L^2} \,d\tau \cr
&\quad + CM_1\int_0^t \|\eta^{m+1}(\cdot,\tau)\|_{L^2}^2 \,d\tau - \int_0^t\|\nabla v^{m+1}(\cdot,\tau)\|_{L^2}^2\,d\tau\\
&\leq \gamma(1 + \e_0)\|\eta_0\|_{L^2}^2 + C(M_1^2 + (M_1')^2) e^{CM_1'}T_2 \cr
&\quad + C(e^{M_1'} + M_1)\int_0^t \|\eta^{m+1}(\cdot,\tau)\|_{L^2}^2 \,d\tau - \int_0^t\|\nabla v^{m+1}(\cdot,\tau)\|_{L^2}^2\,d\tau
\end{aligned}\end{align*}
for $t \leq T_2$, where $C>0$ is independent of $m$. Hence we have
\begin{align}\label{est_0}
\begin{aligned}
&\|\eta^{m+1}(\cdot,t)\|_{L^2}^2 + (1+\e_0) \int_0^t\|\nabla v^{m+1}(\cdot,\tau)\|_{L^2}^2\,d\tau \cr
&\quad \leq \gamma(1+\e_0)^2\|\eta_0\|_{L^2}^2 + C(M_1^2 + (M_1')^2) e^{CM_1'}T_2 \cr
&\qquad + C(e^{M_1'} + M_1)\int_0^t \|\eta^{m+1}(\cdot,\tau)\|_{L^2}^2 \,d\tau
\end{aligned}
\end{align}
for $t \leq T_2$, where $C>0$ is independent of $m$.

\subsection{Higher order estimates} For higher order estimates, we choose a multi-index $\alpha$ with $1 \le |\alpha| \le s$ to get
\begin{align*}\begin{aligned}
&A^0 (\eta^m) \partial_t \nabla_{t,x}^\alpha \eta^{m+1} + \sum_{j=1}^3 A^j(\eta^m) \partial_j  \nabla_{t,x}^\alpha\eta^{m+1} \\
&\quad = \mathcal{R}_\alpha (\eta^m, \eta^{m+1}) + A^0(\eta^m)  \nabla_{t,x}^\alpha E_1(h^m,v^{m+1}) + A^0 (\eta^m)  \nabla_{t,x}^\alpha E_2(g^m,u^m,\eta^m),
\end{aligned}\end{align*}
where $R_\alpha(f_1, f_2)$ is defined as
\[
R_\alpha(f_1, f_2) := -A^0(f_1) \left( \left[  \nabla_{t,x}^\alpha, A^0(f_1)^{-1} \sum_{j=1}^3 A^j(f_1)\partial_j \right] f_2\right)
\]
and $[A,B] := AB-BA$ is the commutator operator. Then, we find
\begin{align*}\begin{aligned}
&\frac{1}{2}\left\{\partial_t \int_\Omega A^0(\eta^m)\nabla_{t,x}^\alpha\eta^{m+1} \cdot \nabla_{t,x}^\alpha\eta^{m+1} \,dx - \int_\Omega \partial_t A^0(\eta^m)\nabla_{t,x}^\alpha\eta^{m+1} \cdot \nabla_{t,x}^\alpha\eta^{m+1} \,dx \right\}\\
&\quad + \frac{1}{2}\left\{ \sum_{j=1}^3 \int_\Omega \partial_j (A^j(\eta)\nabla_{t,x}^\alpha\eta^{m+1} \cdot \nabla_{t,x}^\alpha\eta^{m+1})\, dx - \int_\Omega \partial_j A^j (\eta^m)\nabla_{t,x}^\alpha\eta^{m+1} \cdot \nabla_{t,x}^\alpha\eta^{m+1} \,dx \right\}\\
&\qquad = \int_\Omega R_\alpha (\eta^m, \eta^{m+1}) \nabla_{t,x}^\alpha \eta^{m+1} \,dx \\
&\qquad \quad + \int_\Omega A^0(\eta^m)\lt(\nabla_{t,x}^\alpha E_1(h^m,v^{m+1}) + \nabla_{t,x}^\alpha E_2(g^m,u^m,\eta^m)\rt)\cdot \nabla_{t,x}^\alpha\eta^{m+1} \,dx.
\end{aligned}\end{align*}
For the first term on the right hand side of the above inequality, we obtain
\begin{align*}\begin{aligned}
&\int_\Omega \mathcal{R}_\alpha(\eta^m,\eta^{m+1})\nabla_{t,x}^\alpha \eta^{m+1} \,dx \\
&\quad \le \|A^0(\eta^m)\|_{L^\infty} \sum_{j=1}^3 \left\| \left[ \nabla_{t,x}^\alpha, A^0(\eta^m)^{-1} A^j(\eta^m)\partial_j \right] \eta^{m+1} \right\|_{L^2}\|\nabla_{t,x}^\alpha \eta^{m+1}\|_{L^2}\\
&\quad \le \gamma(1+\e_0)\sum_{j=1}^3 \sum_{\substack{\mu\le\alpha,\  \mu \neq 0}}\left\| \nabla_{t,x}^\mu(A^0(\eta)^{-1} A^j(\eta)) \nabla_{t,x}^{\alpha-\mu} (\partial_j \eta) \right\|_{L^2}\|\nabla_{t,x}^\alpha \eta^{m+1}\|_{L^2}\\
&\quad \le \gamma(1+\e_0)\sum_{j=1}^3 \sum_{\substack{\mu\le\alpha, \ \mu \neq 0 \\ |\mu| \le |\alpha|-2}}\left\| \nabla_{t,x}^\mu(A^0(\eta^m)^{-1} A^j(\eta^m))\right\|_{L^\infty} \left\|\nabla_{t,x}^{\alpha-\mu} (\partial_j \eta^{m+1}) \right\|_{L^2}\|\nabla_{t,x}^\alpha \eta^{m+1}\|_{L^2}\\
&\qquad + \gamma(1+\e_0)\sum_{j=1}^3 \sum_{\substack{\mu\le\alpha ,\  \mu \neq 0 \\ |\mu|\ge |\alpha|-1}}\left\| \nabla_{t,x}^\mu(A^0(\eta^m)^{-1} A^j(\eta^m))\right\|_{L^2} \left\|\nabla_{t,x}^{\alpha-\mu} (\partial_j \eta^{m+1}) \right\|_{L^\infty}\|\nabla_{t,x}^\alpha \eta^{m+1}\|_{L^2}\\
&\quad \le CM_1^r\|\eta^{m+1}\|_{\mathfrak{X}^s}\|\nabla_{t,x}^\alpha \eta^{m+1}\|_{L^2},
\end{aligned}\end{align*}
where $C>0$ is independent of $m$ and $r\geq 1$ depends on $|\alpha|$, but independent of $m$. Next, we estimate
\begin{align*}\begin{aligned}
&\int_\Omega A^0(\eta^m) \nabla_{t,x}^\alpha E_1(h^m,v^{m+1})\nabla_{t,x}^\alpha \eta^{m+1} \,dx\\
&\quad = \int_\Omega (1+h^m) \nabla_{t,x}^\alpha \left(\frac{\Delta v^{m+1}}{1+h^m}\right) \nabla_{t,x}^\alpha v^{m+1} \,dx\\
&\quad = \int_\Omega \Delta(\nabla_{t,x}^\alpha v^{m+1}) \nabla_{t,x}^\alpha v^{m+1} dx + \int_\Omega (1+h^m)\nabla_{t,x}^\alpha\left(\frac{1}{1+h^m}\right)\Delta v^{m+1} \cdot \nabla_{t,x}^\alpha v^{m+1} \,dx\\
&\qquad + \sum_{\mu \le \alpha, \ \mu \neq 0} \binom{\alpha}{\mu} \int_\Omega (1+h^m)\nabla_{t,x}^\mu \left(\frac{1}{1+h^m}\right) \Delta (\nabla_{t,x}^{\alpha-\mu} v^{m+1})\cdot \nabla_{t,x}^\alpha v^{m+1} \,dx\\
&\quad =: I_1 + I_2 + I_3.
\end{aligned}\end{align*}
For $I_1$, we directly get
\[
I_1 = -\|\nabla (\nabla_{t,x}^\alpha v^{m+1})\|_{L^2}^2.
\]
For the estimate of $I_2$, we use the Poincar\'e inequality and Sobolev embedding to get
\begin{align*}\begin{aligned}
I_2 &\le (1+\|h^m\|_{L^\infty}) \left\| \nabla_{t,x}^\alpha \left(\frac{1}{1+h^m}\right)\right\|_{L^2} \|\nabla^2 v^{m+1} \|_{L^4} \|\nabla_{t,x}^\alpha v^{m+1} \|_{L^4}\\
&\le C M_1^r \|v^{m+1}\|_{\mathfrak{X}^s} \|\nabla(\nabla_{t,x}^\alpha v^{m+1})\|_{L^2},
\end{aligned}\end{align*}
where $C>0$ is independent of $m$ and $r\geq 1$ depends on $|\alpha|$, but independent of $m$. For $I_3$, we find
\begin{align*}\begin{aligned}
I_3 &\le \sum_{\mu \le \alpha, \ \mu \neq 0} \binom{\alpha}{\mu} (1+\|h^m\|_{L^\infty}) \|\nabla_{t,x}^\alpha v^{m+1}\|_{L^4} \|\Delta (\nabla_{t,x}^{\alpha-\mu} v^{m+1})\|_{L^2} \left\|\nabla_{t,x}^\mu \left(\frac{1}{1+h^m}\right)\right\|\\
&\le C M_1^r \|\nabla (\nabla_{t,x}^\alpha v^{m+1})\|_{L^2} \|\Delta v^{m+1}\|_{\mathfrak{X}^{|\alpha|-1}}.
\end{aligned}\end{align*}
Here we again used the Poincar\'e inequality and Sobolev embeddin, and $C>0$ is independent of $m$ and $r\geq 1$ depends on $|\alpha|$, but independent of $m$. Combining all of the above estimates, we obtain
\begin{align*}\begin{aligned}
&\int_\Omega A^0(\eta^m) \nabla_{t,x}^\alpha E_1(h^m,v^{m+1})\nabla_{t,x}^\alpha \eta^{m+1} \,dx\\
&\quad \le - \|\nabla(\nabla_{t,x}^\alpha v^{m+1})\|_{L^2}^2 + C M_1^r \|v^{m+1}\|_{\mathfrak{X}^s} \|\nabla(\nabla_{t,x}^\alpha v^{m+1})\|_{L^2}\cr
&\quad \leq -\frac12\|\nabla(\nabla_{t,x}^\alpha v^{m+1})\|_{L^2}^2  + CM_1^{2r}\|\eta^{m+1}\|_{\mathfrak{X}^s}^2,
\end{aligned}\end{align*}
where $C>0$ is independent of $m$ and $r\geq 1$ depends on $|\alpha|$, but independent of $m$. Moreover, we estimate
\begin{align*}\begin{aligned}
&\int_\Omega A^0(\eta^m) \nabla_{t,x}^\alpha E_2(g^m,u^m,\eta^m) \cdot \nabla_{t,x}^\alpha \eta^{m+1} \, dx\\
&\quad = \int_\Omega (1+h^m) \nabla_{t,x}^\alpha \left(\frac{e^{g^m}}{1+h^m} (u^m-v^m)\right) \nabla_{t,x}^\alpha v^{m+1} \,dx\\
&\quad = \int_\Omega e^{g^m} \nabla_{t,x}^\alpha (u^m-v^m)\cdot \nabla_{t,x}^\alpha v^{m+1} \,dx\\
&\qquad + \int_\Omega (1+h^m) \nabla_{t,x}^\alpha\left(\frac{e^{g^m}}{1+h^m}\right)(u^m-v^m)\cdot\nabla_{t,x}^\alpha v^{m+1} \,dx\\
&\qquad +\sum_{\substack{\mu \le \alpha \\ \mu \neq 0, \alpha}} \binom{\alpha}{\mu}\int_\Omega (1+h^m) \nabla_{t,x}^\mu\left(\frac{e^{g^m}}{1+h^m}\right) \nabla_{t,x}^{\alpha-\mu} (u^m-v^m) \nabla_{t,x}^\alpha v^{m+1} \,dx\\
&\quad \le e^{\|g^m\|_{L^\infty}}\left(\|\nabla_{t,x}^\alpha u^m\|_{L^2}^2 + \|\nabla_{t,x}^\alpha v^m\|_{L^2}^2 + \|\nabla_{t,x}^\alpha v^{m+1}\|_{L^2}^2\right)\\
&\qquad + (1+\|h^m\|_{L^\infty})\left\|\nabla_{t,x}^\alpha\left(\frac{e^{g^m}}{1+h^m}\right)\right\|_{L^2}(\|u^m\|_{L^\infty} +\|v^m\|_{L^\infty})\|\nabla_{t,x}^\alpha v^{m+1} \|_{L^2}\\
&\qquad +\sum_{\substack{\mu \le \alpha \\ \mu \neq 0, \alpha}} \binom{\alpha}{\mu}(1+\|h^m\|_{L^\infty}) \left\|\nabla_{t,x}^\mu\left(\frac{e^{g^m}}{1+h^m}\right)\right\|_{L^4} \|\nabla_{t,x}^{\alpha-\mu} (u^m-v^m)\|_{L^4} \|\nabla_{t,x}^\alpha v^{m+1}\|_{L^2}\\
&\quad \le e^{CM_1'}\left(M_1^2 + (M_1')^2 + \|\nabla_{t,x}^\alpha v^{m+1}\|_{L^2}^2\right) + e^{CM_1'}(M_1 + M_1')(M_1' + M_1^r (M_1')^{r'})\|\nabla_{t,x}^\alpha v^{m+1} \|_{L^2}.
\end{aligned}\end{align*}
Here we used Poincar\'e inequality, H\"older inequality and Sobolev embedding, and $C>0$ is independent of $m$, and $r,r'\geq1$ depend on $|\alpha|$, but independent of $m$. Next, we notice that
\[
\int_\Omega \partial_t A^0(\eta^m)\nabla_{t,x}^\alpha\eta^{m+1} \cdot \nabla_{t,x}^\alpha\eta^{m+1} \,dx \le CM_1\|\nabla_{t,x}^\alpha\eta^{m+1}\|_{L^2}^2
\]
and
\[
\int_\Omega \partial_j A^j(\eta^m)\nabla_{t,x}^\alpha\eta^{m+1} \cdot \nabla_{t,x}^\alpha\eta^{m+1}\,dx \le CM_1\|\nabla_{t,x}^\alpha\eta^{m+1}\|_{L^2}^2.
\]
Thus we have
\begin{align*}\begin{aligned}
&\frac{1}{1 + \e_0}\|\nabla_{t,x}^\alpha \eta^{m+1}\|_{L^2}^2 \\
&\quad \le \int_\Omega A^0(\eta^m) \nabla_{t,x}^\alpha \eta^{m+1} \cdot \nabla_{t,x}^\alpha \eta^{m+1} \,dx\\
&\quad \le \gamma(1+\e_0)\|\nabla_{t,x}^\alpha\eta_0\|_{L^2}^2 + CM_1(1+M_1^{r-1})\int_0^t \|\eta^{m+1}(\tau)\|_{\mathfrak{X}^s}^2\,d\tau +CM_1^{2r} \int_0^t\|\eta^{m+1}(\tau)\|_{\mathfrak{X}^s}^2\,d\tau\cr
&\qquad + \int_0^t e^{CM_1'}\left(M_1^2 + (M_1')^2 + \|\nabla_{t,x}^\alpha v^{m+1}(\tau)\|_{L^2}^2\right) \, d\tau \\
&\qquad + \int_0^t e^{CM_1'}(M_1 + M_1')(M_1' + M_1^r (M_1')^{r'})\|\nabla_{t,x}^\alpha v^{m+1} (\tau)\|_{L^2} \,d\tau\\
&\qquad- \frac12\int_0^t \|\nabla (\nabla_{t,x}^\alpha v^{m+1})(\tau)\|_{L^2}^2\, d\tau
\end{aligned}\end{align*}
for $t \leq T_2$, where $C>0$ is independent of $m$ and $r,r'\geq1$ depend on $|\alpha|$, but independent of $m$. By using the fact $\eta^{m+1} = (h^{m+1},v^{m+1})$ and Young's inequality, we further obtain
\begin{align*}\begin{aligned}
&\|\nabla_{t,x}^\alpha \eta^{m+1}\|_{L^2}^2  + \frac{1+\e_0}{2}\int_0^t \|\nabla (\nabla_{t,x}^\alpha v^{m+1})(\tau)\|_{L^2}^2\, d\tau\\
&\quad \leq \gamma(1+\e_0)^2\|\nabla_{t,x}^\alpha\eta_0\|_{L^2}^2 + e^{CM_1'}(M_1^2 + (M_1')^2 )T_2 + C_{M_1, M_1'} \int_0^t\|\eta^{m+1}(\tau)\|_{\mathfrak{X}^s}^2\,d\tau,
\end{aligned}\end{align*}
where $C_{M_1, M_1'} > 0$ is given by
\bq\label{cmm}
C_{M_1, M_1'} = C(M_1 + M_1^r + M_1^{2r} + e^{CM_1'} + (M_1')^2 + M_1^{2r} (M_1')^{2r'}).
\eq
We sum the above inequality over $1 \le |\alpha|\le s$ and combine the resulting one with the zeroth-order estimate \eqref{est_0} to yield
\begin{align*}\begin{aligned}
&\|\eta^{m+1}\|_{\mathfrak{X}^s}^2 + \frac{1+\e_0}{2} \int_0^t \|\nabla v^{m+1}(\tau)\|_{\mathfrak{X}^s}^2 \,d\tau\\
&\quad \leq \gamma(1+\e_0)^2\|\eta_0\|_{H^s}^2 + e^{CM_1'}(M_1^2 + (M_1')^2 )T_2 + C_{M_1, M_1'} \int_0^t\|\eta^{m+1}(\tau)\|_{\mathfrak{X}^s}^2\,d\tau.
\end{aligned}\end{align*}
We finally apply Gr\"onwall's lemma to the above to conclude 
\begin{align*}\begin{aligned}
&\|\eta^{m+1}\|_{\mathfrak{X}^s}^2 + \frac{1+\e_0}{2} \int_0^t \|\nabla v^{m+1}(\tau)\|_{\mathfrak{X}^s}^2 \,d\tau\\
&\quad \leq \lt(\gamma(1+\e_0)^2\|\eta_0\|_{H^s}^2 + e^{CM_1'}(M_1^2 + (M_1')^2 )T_2\rt)\exp\lt(C_{M_1,M_1'}T_2 \rt)
\end{aligned}\end{align*}
for $t \leq T_2$, where $C>0$ is independent of $m$ and $C_{M_1,M_1'} > 0$ is appeared in \eqref{cmm}.

%
%
%
%
%

\end{document}